\documentclass[a4paper,twoside,12p]{article}      
\usepackage{amsmath,amssymb,amsfonts,amsthm,amscd,dsfont}
\usepackage{leftidx} 
\usepackage{graphicx}                 
\usepackage{hyperref}
\usepackage[english]{babel}
\usepackage{color}                    
\usepackage{blkarray}
\usepackage{slashbox}
\usepackage{mathrsfs}
\usepackage{epstopdf}
\usepackage{indentfirst}
\usepackage{enumerate}
\usepackage{url}         
\usepackage{a4wide}
\usepackage{comment}
\usepackage{colonequals} 
\allowdisplaybreaks

\newtheorem{theorem}{Theorem}[section]
\newtheorem{proposition}[theorem]{Proposition}

\newtheorem{lemma}[theorem]{Lemma}

\theoremstyle{definition}
\newtheorem{remark}[theorem]{Remark}

\newtheorem{example}{Example}[section]
\newtheorem{assumption}{Assumption}[section]

\numberwithin{equation}{section}

\def\div{\mathop{\mathrm{div}}\nolimits}
\def\!{\mathop{\mathrm{!}}}

\newcommand{\Keywords}[1]{\par\indent
{\small{\textbf{Key words and phrases.} \/} #1}}
\newcommand{\subjclass}[1]{\par\indent{ \textbf{AMS subject classification.} #1}}
\def\R{\mathbb{ R}}

\def\E{\mathbb{ E}}

\def\Q{\mathcal{Q}}
\def\L{\mathcal{L}}
\def\P{\mathcal{P}}

\def\bRb{\mathbb{R}}
\def\bNb{\mathbb{N}}

\newcommand{\gaga}{\left|\left|}
\newcommand{\drdr}{\right|\right|}
\newcommand{\eee}{{\rm e}}
\newcommand{\dd}{{\rm d}}
\newcommand{\ddx}{{\rm d}x}
\newcommand{\ddy}{{\rm d}y}

\newcommand{\ddt}{{\rm d}t}
\newcommand{\dds}{{\rm d}s}

\newtheorem{thm}{Theorem}[section]
\newtheorem{prop}[theorem]{Proposition}
\newtheorem{cor}[theorem]{Corollary}
\newtheorem{lem}[theorem]{Lemma}
\newtheorem{defn}[theorem]{Definition}

\def\F{\mathcal{F}}

\newlength{\boxwidth}
\newcommand{\RV}{\textcolor{black}}

\title{Multi-species McKean-Vlasov dynamics in non-convex landscapes}

\author{Manh Hong Duong\thanks{School of Mathematics, University of Birmingham, Birmingham B15 2TT, UK. Corresponding author (\texttt{hduong@bham.ac.uk})}
\and Grigorios A. Pavliotis \thanks{Department of Mathematics, Imperial College London, UK (\texttt{g.pavliotis@imperial.ac.uk}}
\and Julian Tugaut\thanks{Universit\'{e} Jean Monnet, Institut Camille Jordan, 23, rue du
docteur Paul Michelon, CS 82301, 42023 Saint-\'{E}tienne Cedex 2,France (\texttt{julian.tugaut@univ-st-etienne.fr})}}
\begin{document}
\maketitle
\begin{abstract}
In this paper, we study multi-species stochastic interacting particle systems and their mean-field McKean-Vlasov partial differential equations (PDEs) in non-convex landscapes. \RV{Under general assumptions on non-convex confining and interaction potentials with polynomial growth}, we establish the well-posedness of the multi-species SDE system, prove propagation of chaos, deriving the corresponding coupled McKean–Vlasov PDE system in the mean-field limit. Our focus is on the long-time and asymptotic behaviour of the mean-field PDEs. \RV{For quadratic interaction potentials and under an appropriate structural assumption, which implies that the generator of each species is multiple of a common generator}, we show the existence and (non-) uniqueness of stationary solutions, study their linear stability \RV{and prove the existence of a phase transition at low noise strengths. For quadratic and symmetric interaction potentials (but no structural assumption),} we construct a free-energy functional that plays the role of a Lyapunov function for the mean-field PDE system.  Furthermore, we establish the convergence of solutions to the mean-field PDEs (and of their free energy) to a stationary state (and the corresponding free energy).
\end{abstract}
\Keywords multi-species interacting particle systems, McKean-Vlasov dynamics, propagation of chaos, invariant measures. Phase transitions, Free-energy functional
\subjclass 60H10; 35K55; 60J60; 60G10.

\tableofcontents

\section{Introduction}
\label{sec:intro}
\subsection{Multi-species interacting particle systems and the mean-field PDEs}
Interacting particle systems play a crucial role in the modeling of complex phenomena in physics, biology, finance, and engineering. In many real-world scenarios, interacting populations are heterogeneous, consisting of multiple species or types of agents with distinct characteristics and interaction rules. When different species influence one another through weak intergroup and intragroup interactions (i.e. when the strengths of the interactions are inversely proportional to the numbers of particles) and subject to thermal stochastically noises, they can be modeled using \textit{ multi-species McKean-Vlasov stochastic differential equations}. Furthermore, for large complex systems, where it is impossible (and not necessary) to trace all individual particles, it is more efficient to study their statistically collective behavior via their probability distributions, which describe the probability of finding the species at a given time and location.  These probability distributions can be obtained as the mean-field limits for the associated empirical measures when the number of particles tends to infinity. The mean-field dynamics, in turn, is described by a system of nonlinear, nonlocal Fokker-Planck-type partial differential equations (PDEs) (also called McKean-Vlasov equations). The connections from the (microscopic) interacting particle systems to the (macroscopic) PDEs and their mathematical analysis give rise to many challenging problems in statistical physics, stochastic analysis, and the theory of PDEs, and constitute an exciting, active with long history, research direction, of which the present work pertains. We refer the reader to the important survey articles \cite{golse2016dynamics, jabin2014review,jabin2017mean, serfaty2018systems, Diez_2022a, Diez_2022b} for additional information and bibliographical remarks; more references on specific topics will be given in the literature review in Section \ref{sec: literature review}. We now describe in detail the models that will be studied in this paper.

In this paper, we will consider the following McKean-Vlasov multi-species interacting SDEs and their mean-field limits PDEs:
\begin{subequations}
\label{eq: many SDEs}
\begin{align}
&\dd X_t^{i,k} =-\nabla V_k(X^{i,k}_t)\,\ddt-\frac{1}{N_n}\sum_{\ell=1}^M\sum_{j=1}^{N_n^\ell}\nabla F_{k\ell}(X^{i,k}_t-X^{j,\ell}_t)\,\ddt+\sigma_k \dd W^{i,k}_t,
\\& X^{i,k}(0)=  X^{i,k}_0,
\\& \text{for}~~ i=1,\ldots, N_n^k, \quad k=1,\ldots,M. \notag
\end{align}
\end{subequations}
We have used the following notation:  
\begin{itemize}
\item $M>0$, a positive integer, is the number of species; $\{N_n^k\}_{k=1,\ldots, M}$ are $M$ sequences of positive integers, which are the populations of the species. $N_n=\sum_{k=1}^M N_n^k$ is the total population.
\item $\{X^{i,k}\}_{i=1}^{N_n^k} \in 
(\mathbb{R}^d)^{N_n^k}$ are copies of species $k$, $k=1,\ldots, M$.      
\item $\{V_k\}_{k=1}^M$ are external potentials, $V_k:\mathbb{R}^d\mapsto \mathbb{R}$.
\item $\{F_{kk}\}_{k=1}^M$ are self-interacting potentials that describe interactions between individuals of the same species (intraspecies interactions), $F_{kk}:\mathbb{R}^d\mapsto \mathbb{R}$.
\item $\{F_{k\ell}\}_{k\neq \ell=1}^M$ are cross-interacting potentials representing the interactions between individuals belonging to different species $k$ and $\ell$ (interspecies interactions), $F_{k\ell}:\mathbb{R}^d\mapsto \mathbb{R}$.
\item $\{\sigma_k\}_{k=1}^M$ are the noise strengths, $\sigma_k>0$.
\item $\{W^{i,k}_t: t\geq 0\}_{i=1,\ldots,N_n^k,k=1,\ldots, M}$ are $d$-dimensional independent Wiener processes.
\item $X_0^{i,k}$ are the initial data, $X_0^{i,k}\in\mathbb{R}^d$. We assume chaotic initial data, i.e. that $X_0^{i,k} \sim \mu_0^k, \; i=1, \dots N_n^k, \; k=1, \dots M$ iid.
\end{itemize}
Thus, SDEs \eqref{eq: many SDEs} describe the motions of an interacting multi-species population under the influence of external forces, intra-species and inter-species interacting forces, and stochastic noises. We are also interested in the mean-field limit of \eqref{eq: many SDEs} when the population sizes $N^k_n$, which are parameterized by $n$, go to infinity as $n \rightarrow \infty$ . Suppose that
\begin{equation}
\label{eq: coefficients a}
    a_i:=\lim_{n\rightarrow\infty}\frac{N^i_n}{N_n}
\end{equation} 
exist for all $i=1,\ldots M$. The parameters $\{a_i\}_{i=1}^M$, which satisfy $0\leq a_i\leq 1, \sum_{i=1}^M a_i=1$, describe the proportion of the $i$-th species in the population. Then in the limit as the number of particles of each species goes to infinity, the interacting particle system $\{X^{i,k}_t\}$ converges, in a suitable sense, to  $\{\overline{X}^{i,k}_t\}$, the solution of the coupled (nonlinear in the sense of) McKean SDE (see Section \ref{subsec:well_posed} below)
\begin{subequations}
\label{eq: mean-field SDE}
\begin{eqnarray}
\dd\overline{X}^{i{\color{black},k}}_t &=&-\nabla V_k(\overline{X}^{i{\color{black},k}}_t)\,\dd t-\sum_{\ell=1}^M (a_\ell\nabla F_{k\ell}\ast \mu^{\ell}_t)(\overline{X}^{i{\color{black},k}}_t) \,\ddt+\sigma_{{\color{black}k}} \dd W^{i{\color{black},k}}_t,
 \\ \overline{X}^{i{\color{black},k}}_{t=0} & =& X^{i{\color{black},k}}_0 \sim \mu_0^{{\color{black}k}}, 
\end{eqnarray}
\end{subequations}
for ${\color{black}k}=1, \dots M$. In the above equation, $\mu_t^\ell$ is the marginal law at time $t$ of $\overline{X}^{{\color{black}i,}\ell}$ and $\mu_0^{\color{black}k}$ denotes the initial distribution (we assume that we have chaotic initial conditions). The mean-field limit can also be described using the associated empirical measures. More specifically, the associated empirical measures of the ${\color{black}k}$-th species is defined by
\begin{equation}
\label{eq: empirical measures}
\mu^{{\color{black}k}}_t:=\frac{1}{N_n^{\color{black}k}}\sum_{j=1}^{N_n^{\color{black}k}}\delta_{X^{j{\color{black},k}}_t},\quad {\color{black}k}=1,\ldots, M.
\end{equation}
Then, as $n\rightarrow \infty$, the empirical measures $\mu_t^{{\color{black}k,}n}$ will weakly converge to a probability measure $\mu_t^{{\color{black}k}}, \, {\color{black}k}=1,\dots {\color{black}M}$,  which are the solution of the coupled Mckean-Vlasov PDE system (this statement is a direct consequence of the so-called propagation of chaos property \cite{mckean1967propagation,sznitman1991topics}, see Section~\ref{subsec:well_posed} below):
\begin{subequations}
\label{eq: multi-species PDE}
\begin{eqnarray}
  \partial_t \mu_t^i & = & \div\Big[\Big(\nabla V_i+\sum_{j=1}^M a_j\nabla F_{ij}\ast\mu^j_t\Big)\mu^i_t \Big]+\frac{\sigma_i^2}{2}\Delta\mu^i_t, \quad i = 1, \dots M,
  \\ \mu_{t=0}^i &=& \mu_0^i, \quad i = 1, \dots M. 
\end{eqnarray}
\end{subequations}
This system of nonlinear, nonlocal partial different equations describes the evolution of the probability distributions of the species, in which the evolution of each distribution depends on the others via convolution terms.

In this paper, we are interested in the mathematical analysis of the many-particle interacting system~\eqref{eq: many SDEs} and their mean-field limit~\eqref{eq: mean-field SDE} or \eqref{eq: multi-species PDE} from various aspects and for \textit{non-convex confining potentials}. We discuss the well-posedness of \eqref{eq: many SDEs} and the convergence to the mean-field limits. More specifically, we will focus on the long-time and asymptotic behavior of the PDEs \eqref{eq: multi-species PDE}, including more specially, the existence and (non-) uniqueness of stationary solutions, their linear stability, the existence of a phase transition phenomena when the noise intensities are varied, as well as {\color{black} the description of the limiting (in long-time) set with possibly a convergence towards an equilibrium state if the set of invariant probability measures is discrete.}

Before stating our main results in the next section, we provide a review of the literature on related work on the interacting particle systems~\eqref{eq: many SDEs} and their mean field limit~\eqref{eq: mean-field SDE} and \eqref{eq: multi-species PDE}.
%
%
\subsection{Literature Review}
\label{sec: literature review}
As we have already mentioned, stochastic interacting particle systems and their mean field limit naturally arise in a variety of applications in statistical physics, biology, ecology, population dynamics and even the social sciences. There is now a huge body of literature on this topic; we comment below of a few papers that are most closely related to the present work.

\textit{On well-posedness and long-time behavior of the one-species and two-species mean-field PDEs.} The main motivation of this paper is to generalize some of the existing results on the long-time behavior for the one-species McKean-Vlasov PDEs in \cite{Tamura1984,yozo1987free,Dawson83,HT1,AOP,SPA,PT,JOTP,HT2,HT3,carrillo2020long,delgadino2021diffusive} and two species in \cite{DuongTugaut2020} to multi-species ($M>2$) settings. 

\textit{On well-posedness and long-time behavior of closely related multi-species mean-field PDEs.}
In \cite{potts2019spatial, giunta2022local,giunta2022detecting}, the authors study the following system
\begin{equation}
\label{eq: Potts model}
\partial_t\rho_i=D_i\Delta\rho_i +\div\Big(\rho_i \sum_{j=1}^M \gamma_{ij} (\nabla K\ast\rho_j)\Big), \quad i=1,\ldots, M,
\end{equation}
on the $d$-dimensional torus, with periodic boundary conditions. More precisely, in the one-dimensional case, for the top-hat kernels (that is, $K(x)=1/(2\delta)$ for $x\in (-\delta,\delta)$ and $K(x)=0$ otherwise) and under the assumption of  no self attraction or repulsion (i.e.,$i\neq j$), \cite{potts2019spatial} analyses pattern formation properties, fully analytically characterizing for $N=2$ and numerically for $N=3$. In \cite{giunta2022local} the assumption that $j\neq i$ is removed, allowing for self-attraction or repulsion, but the kernel $K$ is instead twice diﬀerentiable. Under these assumptions, the authors prove the global existence of a unique, positive solution in one spatial dimension and the local existence in arbitrary dimensions; they also propose an efficient scheme for numerically solving the equations. In \cite{giunta2022detecting}, under the symmetry assumption that $\gamma_{ij}=\gamma_{ji}$, the authors construct an energy functional for \eqref{eq: Potts model}, showing that for any initial data $u_0\in L_1(\mathbb{T}^M)$ the solution evolves over time towards a configuration that is a local minimizer of the energy functional.

In \cite{jungel2022nonlocal} the authors consider the system \eqref{eq: Potts model} with general pair-wise kernels of the form
\begin{equation}
\label{eq: Jungel model}
\partial_t\rho_i=D_i\Delta\rho_i +\div\Big(\rho_i \sum_{j=1}^M \nabla K_{ij}\ast\rho_j\Big), \quad i=1,\ldots, M.
\end{equation}
Under the assumption that these kernels are positive definite and in detailed balance, based on the construction of a suitable Lyapunov functional and entropy estimates, the authors prove the global existence of weak solutions, the weak-strong uniqueness (i.e., if a weak solution and a strong solution have the same initial data and appropriate regularities, then they coincide for all time), and the localization limit (the existence and uniqueness theorems hold for nondiﬀerentiable kernels). In \cite{carlier2016remarks}, the authors consider \eqref{eq: Jungel model} with general vector fields, which are not necessarily gradients of potentials.  Under suitable Lipschitz regularity conditions for vector fields, they show the existence and uniqueness of solutions \eqref{eq: Jungel model} using the framework of (perturbations of) Wasserstein gradient flows; we remark that this paper also allows for nonlinear diffusion terms. 

\textit{On the rigorous justification of the mean-field limit (propagation of chaos) of related multi-species stochastic interacting particle systems} 
The mean-field limits and propagation of chaos for systems of the form \eqref{eq: Jungel model} for various assumptions on the kernels have been studied by many authors; see, for instance, \cite{chen2019rigorous,carrillo2025propagation}.

\textit{On other related multi-species models in statistical physics and biology.} It is also worth mentioning other closely related multi-species models in statistical physics and biology employing various types of intra- and inter-species interaction potentials, see for instance \cite{chen2021rigorous, hecht2023multispecies} for multi-species cross-diffusion systems, \cite{barra2015multi,panchenko2015free} for multi-species spin-glass systems, \cite{wolansky2002multi,he2021multispecies,karmakar2020critical,Lin_2024, Lin2024b} for multi-species Patlak-Keller-Segel chemotaxis systems, as well as \cite{CAUVINVILA2024} for multi-species diffusion-Cahn–Hilliard systems.

\subsection{Assumptions and contributions}
\label{sec:contrib}
\RV{We present assumptions on the confining potentials $V_k$ ($k=1,\ldots, M)$, interaction potentials $F_{ij}$ ($1\leq i,j\leq M$), and strength of the noises $\sigma_k$ ($k=1,\ldots, M)$ that will be used in the subsequent analysis. In Table \ref{tab: conditions} we specify which assumptions are used for which result.}
\subsubsection{Assumptions}
\label{subsec:assump}

In this section, we present the assumptions that we make in this paper. We start with our assumption on the confining potentials.

\begin{assumption}\label{asp: assumption}
We make the following assumptions on the confining potentials.
\begin{description}

\item[(H1)] The functions $V_i$ are twice continuously differentiable for $i\in\{1,\ldots, M\}$.  
\item[(H2)] The potentials $V_i$, $i=1,\ldots, M$ are convex at infinity: $\displaystyle\lim_{|x|\to+\infty}y^T\nabla^2V_i(x)y=+\infty$ for all $y\in\R^d$ with $\|y\|=1$. 
\item[(H3)] There exists an integer $p\geq 2$ and a constant $C>0$ such that the following inequality holds: $\sum_{i=1}^M|\nabla V_i(x)|\leq C\left(1+|x|^{2p-1}\right)$.
\end{description}
\end{assumption}

Assumptions~\ref{asp: assumption} are taken to ensure the existence of a strong solution to the system~\eqref{eq: mean-field SDE}. Assumption (H1) is crucial since we will make repeated use of It\^{o}'s formula. This immediately implies that the functions $\nabla V_i$ are locally Lipschitz. Assumption (H2) is a standard dissipativity assumption. We point out that it implies the existence of $\theta_i>0$ such that 
\begin{equation}
\label{eq: V1}
(\nabla V_i(x)-\nabla V_i(y))(x-y)\geq -\theta_i|x-y|^2\quad\forall x,y \in \mathbb{R}^d.
\end{equation}
Finally,  (H3) is used to prove the existence of a solution to the mean-field SDE system \eqref{eq: mean-field SDE}. Indeed, to do so, we use a fixed point approach in a space with finite moments of appropriate orders.

\begin{assumption}
\label{severine}
For any $k\in\{1,\ldots,M\}$, the confining potential $V_k$ is an even polynomial function of the form $V_k(x)=-\frac{\theta_k^{(2)}}{2}x^2+\sum_{\ell=2}^{p_k}\frac{\theta_k^{(2\ell)}}{2\ell}x^{2\ell}$ with $\theta_k^{(2)}>0$, $\theta_k^{(2p_k)}>0$ and all the coefficients $\theta_k^{(2i)}\geq0$ for any $2\leq i\leq p_k-1$. Furthermore, $p_k\geq2$. 
\end{assumption}
This will allow us to employ the results obtained in~\cite{PT}. 
We now give the main assumptions about the interaction potentials.
\begin{assumption} (polynomial interactions)
\label{asp: interaction2}
\begin{itemize}
\item For $i=1,\ldots, M$, $\nabla F_{ii}$ is odd and increasing with polynomial growth functions.
\item For $1\leq i\neq j\leq M$, $\nabla F_{ij}$ is Lipschitz.
 \end{itemize}
\end{assumption}
\begin{assumption}(quadratic interactions)
\label{asp: assumptionbis}
\begin{itemize} 
\item For any $1\leq i,j\leq M$, there exists $\alpha_{ij}\in\bRb$ such that $F_{ij}(x)=\frac{\alpha_{ij}}{2}|x|^2$.
\item For $i=1,\ldots, M$, $\alpha_{ii}>0$.
\end{itemize}
\end{assumption}

\begin{assumption}
\label{severine2}
For all $k,\ell\in\{1,\ldots,M\}$, we have $\alpha_{k\ell} \geq 0$. 
\end{assumption}
For quadratic interaction potentials, the system of coupled McKean-SDEs that we consider in this paper is an extension of the Desai--Zwanzig model~\cite{Dawson83}. Some of our results can be extended to the case of nonlinear interactions; we will comment on when such extensions are possible in later sections. It is also possible to relax the assumption that the diagonal elements $\alpha_{ii}$'s are positive, which is analogous to the assumption that the interaction potential is ferromagnetic in the single-species case~\cite{shiino1987}. However, we will not consider this in this paper. We emphasize that the assumption that the interaction potentials are quadratic is used in the study of phase transitions. Well-posedness of the mean-field PDEs and uniform propagation of chaos hold for a broader class of interaction potentials.

For some of our results, we will need to assume that the interaction potentials are symmetric:

\begin{assumption}
\label{asp: assumptionter}
For any $1\leq i,j\leq M$, $F_{ij}=F_{ji}.$
\end{assumption}
This symmetry assumption implies the existence of a free-energy functional and, consequently, of a gradient structure, and will be used in the {\color{black} description of the limiting (in long-time) set with possibly a convergence towards an equilibrium state if the set of invariant probability measures is discrete}; see Section~\ref{sec:converg}.

An alternative to Assumption~\ref{asp: assumptionter} is the following; we will refer to it as the \textit{structural assumption}.

\begin{assumption}
\label{asp: assumptionterbis}
For any $1\leq i,j\leq M$, $\frac{\alpha_{ij}}{\sigma_i^2}=\frac{\alpha_{1j}}{\sigma_1^2}$ and $\frac{V_i}{\sigma_i^2}=\frac{V_1}{\sigma_1^2}$. 
\end{assumption}
In particular, this assumption implies that the effect on the individual $i$ of the individual $j$ depends only on the species $j$. We immediately deduce the existence of coefficients $\alpha_j, j=1, \dots, M$ such that
\[
\frac{\alpha_{ij}}{\sigma_i^2}=\frac{\alpha_j}{\sigma^2}\,,
\]
where $\sigma^2>0$. Similarly, for the confining potentials, we deduce that there exists a potential $\overline{V}$ such that
\[
\frac{V_i}{\sigma_i^2}=\frac{\overline{V}}{\sigma^2}\,.
\]

Under the structural assumptions and for quadratic interaction potentials, the mean-field PDE system takes a particularly simple form; see the remark below. 
\begin{remark}[Multi-species dynamics under the structural assumption]
Under the structural assumption, the multi-species McKean-Vlasov coupled PDE system \eqref{eq: multi-species PDE} takes the form:
\begin{align}
\nonumber \partial_t \mu_t^i&=\div\Big[\Big(\nabla V_i+\sum_{j=1}^M a_j \alpha_{ij} \left(x-\int y \mu^j_t \, dy \right)\Big)\mu^i_t \Big]+\frac{\sigma_i^2}{2}\Delta\mu^i_t
\\&=\sigma_i^2\left\{\div\Big[\Big(\frac{\nabla V_i}{\sigma_i^2}+\sum_{j=1}^M a_j \frac{\alpha_{ij}}{\sigma_i^2} \left(x-\int y \mu^j_t \, dy \right)\Big)\mu^i_t \Big]+\frac{1}{2}\Delta\mu^i_t\right\}  \nonumber
\\&=\sigma_i^2\left\{\div\Big[\Big(\frac{\nabla\bar{V}}{\sigma^2}+\sum_{j=1}^Ma_j\frac{\alpha_j}{\sigma^2}(x-\int y \mu^j_t \, dy)\Big)\mu^i_t\Big]+\frac{1}{2}\Delta \mu^i_t\right\}  \nonumber
  \\&=: \frac{\sigma_i^2}{\sigma^2}\mathcal{L}^* \mu^i_t,  \label{e:structural}
\end{align}
where
\begin{equation} \label{e:struct_generator}
\mathcal{L}^* \mu=\div\Big[\Big(\nabla\bar{V}+\sum_{j=1}^Ma_j\alpha_j(x-m_j)\Big)\mu\Big]+\frac{\sigma^2}{2}\Delta \mu.
\end{equation}
Here, we have used the notation $m_j = \int y \mu^j_t \, dy$. Thus, the equations for the $\mu_t^i$'s differ from each other by a time scaling, which depends on each particular species. Note that $\mathcal{L}$--the $L^2-$adjoint of the operator in~\eqref{e:struct_generator} is precisely the generator of the one-species Desai--Zwanzig model with an effective potential $\bar V$. The above formula implies that all species have the same stationary distribution, which is the stationary state of the one-species Desai--Zwanzig model with confining potential $\bar V$ that was studied in~\cite{Dawson83}.
\end{remark}

Finally, the following \emph{synchronization} assumption ensures that, in the zero-noise 
limit, the stationary measures are of Dirac form. 
\begin{assumption}
\label{asp: assumption:synchronization}
For any $k\in\{1,\cdots,M\}$, $\theta_k<\sum_{\ell=1}^M\RV{a_{\ell}}\alpha_{k\ell}$, where the $\theta_k$'s are defined in~\eqref{eq: V1}.
\end{assumption}
The terminology ``synchronization'' has been systematically used by the third author and originates from the literature on nearest-neighbour systems, see, for example, \cite{BFG1,BFG2}. This term is used to describe the fact that in the mean-field limit and for the single-species case, the only minimizer of the potential acting on the whole system of particles seen as a unique diffusion in $(\bRb^d)^N$ are of the form $(z,\cdots,z)$. In particular, when the zero-noise system reaches a stable equilibrium, all the particles are in the same position. 

\RV{Finally, for the results of Section~\ref{sec:converg} that do not require the interaction potentials to be quadratic, we will use instead the following assumption.}
\begin{assumption}[\RV{radial polynomial interactions}]
\label{asp: interaction3}
\RV{For any $1\leq i,j\leq M$, we have $F_{ij}(x)=G_{ij}(|x|)$, where each $G_{ij}$ is a polynomial even function of order less than $q\in\bNb^*$.}
\end{assumption}

\subsubsection{Main results and novelties}
As already mentioned in the first bullet point of Section \ref{sec: literature review}, the main motivation of this paper is to generalize some of the existing results on the long-time behaviour for the one-species and two-species McKean-Vlasov PDEs to arbitrary multi-species ($M>2$) settings (note that the two-species case studied in \cite{duong2016stationary} provides only a preliminary result on the number stationary solutions under special conditions). This will significantly broaden the applicability of the models since 
many applications in biology, population dynamics and socio-economical dynamics necessarily contain more than two interacting species \cite{clifford1973model,garnier2017consensus,potts2019spatial}. As will also be shown in the subsequent sections, from a mathematical point of view, the generalization from one or two-species to an arbitrary number of species is highly nontrivial due to the more complex inter-species interactions between different species. Compared to the existing aforementioned results discussed in Section~\ref{sec: literature review} for the models \eqref{eq: Potts model} and \eqref{eq: Jungel model}, the key novelty of the present work is to take into account the non-convex potentials, revealing the non-trivial influence of the interplay between these non-convex potentials and the intra- and inter- species interaction potential on the long-time behaviour of the PDEs system \eqref{eq: multi-species PDE}. The main results of the papers \RV{can be grouped} as follows. \RV{In Table \ref{tab: conditions}, we clarify which conditions are used for each result.
}

\medskip

\RV{\textbf{Results that hold for general multiple dimensional multi-species systems:}}
\begin{itemize}
    \item Theorem \ref{thm: well-posedness} shows the well-posedness of the {\color{black}McKean mean-field limit SDEs} \eqref{eq: mean-field SDE}.
    \item Theorem \ref{theo: PoC} establishes the propagation of chaos for the particle systems, thus the convergence of the associated empirical measures to the mean-field limits.
    \item Proposition \ref{jugnot} provides a characterization for the stationary solutions as a fixed point of a nonlinear map, which reduces to a system of equations (self-consistency equations) for the mean values in the case of quadratic interaction potentials.
\end{itemize}
\RV{\textbf{Results that are valid in one dimension, require quadratic interaction potentials and assumptions on the noise strength:}}
\begin{itemize}
    \item Proposition \ref{thunder} shows the existence of a stationary state close to each so-called ``candidate'' when the diffusion densities (temperatures) are small enough.
    \item Theorem \ref{lola} proves the existence and uniqueness of a stationary solution in the large-noise (temperature) regime.
\end{itemize}
\RV{\textbf{Results that are valid in one dimension, require quadratic interaction potentials and the Structural Assumption:}}
\begin{itemize}    
\item Theorem \ref{negan} reveals the non-uniqueness of stationary states and a phase transition phenomenon, that is, there is a unique (respectively, non-unique) stationary solution when the temperature is greater or equal (resp., smaller) to a critical temperature. The equation determining the critical temperature is also given.
\item In Proposition~\ref{prop:null_space_multi} we look into linear stability analysis and we study the null space of the linearized McKean-Vlasov operator.
\end{itemize}
\RV{\textbf{Results that are valid in one dimension and require quadratic and symmetric interactions:}}
\begin{itemize}
\item Theorem \ref{chichi} describes the topology of the set of limiting values of the solutions of \eqref{eq: multi-species PDE}.
\item Proposition \ref{thm:fr:subconv} shows the convergence of {\color{black}some sequences of } solutions of \eqref{eq: multi-species PDE} to a stationary measure.
\end{itemize}
\RV{\textbf{Results that are valid in one dimension and symmetric interactions:}}

Proposition \ref{jigsaw} proves the convergence of the free-energy  of the solutions of \eqref{eq: multi-species PDE} to that of the  stationary measure. 
\medskip
\RV{
\begin{table}[ht!]
\begin{tabular}{ *{7}{|c}|} 
\hline
&content&confining&interaction &dimension& noise\\ \hline
  Th. \ref{thm: well-posedness} &well-posedness of {\color{black}McKean} SDE & A. \ref{asp: assumption}&A. \ref{asp: interaction2} &arbitrary  &arbitrary\\
  \hline
  Th. \ref{theo: PoC} & propagation of chaos &A. \ref{asp: assumption} &  A. 
  \ref{asp: interaction2} &arbitrary &arbitrary\\
  \hline
Prop. \ref{jugnot} &self-consistency eqn.  & A. \ref{asp: assumption}&A. \ref{asp: interaction2} &arbitrary &arbitrary\\
  \hline
  Prop. \ref{thunder}&existence of SS  & A. \ref{asp: assumption}&A. \ref{asp: assumptionbis}  &one &small\\
  \hline
  Th. \ref{lola}
& existence and uniqueness of SS  & A. \ref{asp: assumption},\ref{severine},\ref{severine2}& A. \ref{asp: assumptionbis} & one &large\\
  \hline
  Th. \ref{negan}
&phase transition  & A. \ref{asp: assumption},\ref{severine},\ref{severine2},\ref{asp: assumptionterbis}& A. \ref{asp: assumptionbis} &one  &critical\\
\hline
Prop.~\ref{prop:null_space_multi} 
&linear stability  &A. \ref{asp: assumption},\ref{severine},\ref{severine2},\ref{asp: assumptionterbis} &A. \ref{asp: assumptionbis}  &one & arbitrary\\
\hline
Th. \ref{chichi}
&limiting sets  &A. \ref{asp: assumption}  &A. \ref{asp: assumptionbis}  ,\ref{asp: assumptionter}&one &arbitrary\\
\hline
Prop. \ref{thm:fr:subconv}&conv. to a SS &A. \ref{asp: assumption}  &A. \ref{asp: assumptionbis}  ,\ref{asp: assumptionter}&one &arbitrary\\
\hline
Prop. \ref{jigsaw}
& conv. of the free energy  & A. \ref{asp: assumption}  &A. \ref{asp: assumptionter}, \ref{asp: interaction3} &one &arbitrary\\
\hline
\end{tabular}
\caption{Specification of the conditions used in each result. We use the following shorthands: Th.: Theorem, Prop.: Proposition, McKean SDE: McKean-Vlasov multi-species stochastic differential equations, SS: stationary state, A.: assumption(s), confining: confining potential, interaction: interaction potential, noise: strength of noise.}
\label{tab: conditions}
 \end{table}
}
\medskip

\RV{The assumptions of quadratic and symmetric interaction potentials are commonly adopted in biological models; see, for instance, the Desai--Zwanzig model~\cite{Dawson83} and the cross-diffusion models in \cite{chen2021rigorous,hecht2023multispecies}.
Although the structural assumption used in Theorem 3.2 and Proposition 4.1 is rather restrictive, as it effectively reduces the multi-species model to a single-species setting, it enables us to derive the phase transition result for the former from the corresponding analysis of the latter. All other results are applicable to the multi-species setting whose proofs require technical improvements from the two-species analysis in \cite{DuongTugaut2020}.}
%
%
\subsection{Outline of the paper}
The rest of the paper is organized as follows. 
In Section~\ref{sec:prelim}, we present some preliminary results on the multi-species interacting particle system and its mean-field limit, including the well-posedness of the multi-species McKean-Vlasov SDE, propagation of chaos, characterization of the stationary states and the associated free-energy functional. In Section~\ref{sec:inv_meas} we study the invariant probability measures (the existence in the small-noise case, the existence and uniqueness when the diffusion coefficients are large, the phase transition under the structural assumptions). In Section~\ref{sec:lin_stab} we study the linear stability analysis of stationary states. In Section~\ref{sec:converg} we show that, under appropriate assumptions, solutions of the mean-field coupled PDE system converge to a stationary state. Finally, in the appendices, we collect some technical results: in Appendix~\ref{app:pogacar} we study the algebraic equations characterizing the stationary states in the limit of vanishing noise. 
In Section~\ref{app:quadr}, we prove that the quadratic confining potentials may lead to a non-uniqueness of the steady state. 

Finally, in Appendix~\ref{app:converg} a technical result which is central in {\color{black}describing the limiting set (in the long time)} is proved.
%
%
\section{Preliminaries}
\label{sec:prelim}

In this section, we present some basic properties of the multi-species interacting particle system~\eqref{eq: many SDEs}, of the mean-field McKean-Vlasov PDEs~\eqref{eq: multi-species PDE} and of the McKean SDEs~\eqref{eq: mean-field SDE} that will be needed later on. In Section~\ref{subsec:well_posed} we present a basic well-posedness result for the mean-field SDEs; in Section~\ref{subsec:prop_chaos} we present a basic propagation of chaos theorem. In Section~\ref{subsec:stat_states} we characterize all stationary states of the mean-field dynamics. Finally, in Section~\ref{subsec:free_energy} we derive the free-energy of the mean-field system, under the symmetry Assumption~\ref{asp: assumption}, and we  show that it is non-increasing and lower-bounded.

\subsection{Well-posedness of the multi-species SDEs}
\label{subsec:well_posed}

We start by stating a basic well-posedness result for the multi-species McKean{\color{black}-Vlasov} SDE system.

\begin{theorem}
\label{thm: well-posedness}
Suppose that Assumptions \ref{asp: assumption} and \RV{\ref{asp: interaction2}} hold and that the initial conditions $\{X_i(0)\}$ $i=1,\ldots M$ satisfy $\E(|X_i(0)|^{8p^2})<\infty$ where $p\geq 2$ is defined in (H3) of Assumption \ref{asp: assumption}. Then the system \eqref{eq: mean-field SDE} admits a unique strong solution {\color{black}on the whole time interval $\bRb_+$}. Furthermore, the $8 p^2$ moments are uniformly bounded in time for any index $k$, with the upper bound depending only on the moments of the initial conditions.
\end{theorem}
\begin{proof}
\RV{The theorem is proved for $M=2$ in~\cite[Theorem 2.1]{DuongTugaut2020} which employs the method developed earlier in \cite{benachour1998nonlinear} and relies on three main ingredients: (1) transforming the system \eqref{eq: many SDEs} to a fixed point problem associated to a transformation in an appropriate functional space, (2) showing that the transformation is a continuous contraction to obtain the existence and uniqueness of a strong solution locally in time and (3) expanding the local solution to a global one. }

\RV{The proof for the case $M>2$ proceeds analogously. We therefore provide only a sketch, highlighting the aspects that require extension.} 

\RV{\textbf{Step 1 (construction of the transformation and functional space)}. Following \cite{DuongTugaut2020} we consider the following space of functions
\begin{equation*}
\Lambda_T:=\Lambda_T^1\cap\Lambda_T^2\cap\Lambda_T^3\,,
\end{equation*}
where the three spaces of functions $\Lambda_T^1$, $\Lambda_T^2$ and $\Lambda_T^3$ are defined by
\begin{equation*}
\Lambda_T^1:=\left\{b:[0;T]\times\bRb\longrightarrow\bRb^d\,\,\left|\right.\,\,x\mapsto b(s,x)\mbox{ is locally Lipschitz uniformly in }s\right\}\,,
\end{equation*}
where the parameter of Lipschitz may depend on $b$;
\begin{equation*}
\Lambda_T^2:=\left\{b:[0;T]\times\bRb\longrightarrow\bRb^d\,\,\left|\right.\,\,x\mapsto b(s,x)\mbox{ is increasing and }b(s,x)-b(s,y)\geq\xi_1(x-y)+\xi_0\right\}\,,
\end{equation*}
where $\xi_1>0$, $\xi_0\in\bRb$ and $x\geq y$; and
\begin{equation*}
\Lambda_T^3:=\left\{b:[0;T]\times\bRb\longrightarrow\bRb^d\,\,\left|\right.\,\,\left|\left|b\right|\right|_T<\infty\right\}\,,
\end{equation*}
where for a function $b:\bRb_+\times\bRb\rightarrow\bRb^d$, we define
\begin{equation*}
\left|\left|b\right|\right|_T:=\sup_{0\leq s\leq T}\sup_{x\in\bRb}\left(\frac{\left|b(s,x)\right|}{1+|x|^{2q}}\right).
\end{equation*}
The space $\Lambda_T$ is equipped with the norm $||.||_T$. 
Now we define
$F_T:=(\Lambda_T)^{M^2}$ equipped with the norm 
\begin{equation*}
\left|\left|b\right|\right|_T^F:=\sum_{i,j=1}^{M}\left|\left|b_{ij}\right|\right|_T\,,
\end{equation*}
where $b:=(b_{ij})_{i,j=1}^{M}$. This will be the functional space for the multi-species model. We next define a transformation $\Gamma$ from $F_T$ to $F_T$ by
\begin{equation}
\label{Gammab}
[\Gamma b]_{k\ell}(x):=a_k\mathbb{E}\Big[\nabla F_{k\ell}(x-X^\ell_t)\Big],    
\end{equation}
where $\{X^k\}$ is a strong solution to the following SDE:
\begin{equation*}
dX^k_t=-\nabla V_k (X^k_t)\,dt-\sum_{\ell=1}^M b_{k\ell}(X^k_t)\,dt+\sigma_k dW^k_t.   
\end{equation*}
The existence of solutions to the above SDE is a consequence of \cite[Theorem 10.2.2]{stroock2007multidimensional}. The key observation is that a solution to the mean-field SDE system \eqref{eq: mean-field SDE} is a fixed point of the transformation $\Gamma$.}

\RV{\textbf{Step 2 (existence and uniqueness of a local solution)}. To show the existence and uniqueness of a local solution, we show that the transformation $\Gamma$ is a continuous contraction on $F_T$. The fact that $F_T$ is a transformation from $F_T$ to itself follows from Assumptions \ref{asp: assumption} and \ref{asp: interaction2} which ensure that $
[\Gamma b]_{k\ell} \in \Lambda_T$ for all $1\leq k,\ell\leq M$. The continuity of $\Gamma$ follows from estimates of the form 
\begin{equation}
\label{Gammabc}
\|[\Gamma b]_{k\ell}-[\Gamma c]_{k\ell}\|_T\leq C\sqrt{T}\,\|b_{k\ell}-c_{k\ell}\|_T,  \end{equation}
for $b,c\in F_T$, $1\leq k,\ell\leq M$, where $C$ is a constant depending the $8p^2$-moments of the initial data. The upper bound can be deduced from \eqref{Gammabc} and Assumptions \ref{asp: assumption} and \ref{asp: interaction2}. Since the aforementioned estimates are  coordinate-wise estimates, their proofs are similar to $M=2$ case. We refer to \cite[Lemma 2.8 \& Lemma 2.9]{DuongTugaut2020} for the detailed calculations. The contraction of $\Gamma$ and local-in-time existence of a fixed point (thus a solution of \eqref{eq: mean-field SDE}) is then deduced from estimate \eqref{Gammabc} by choosing $T<T_0$ sufficiently small.
}

\RV{\textbf{Step 3 (from local- to global-in-time solution)} The local-in-time solution obtained in Step 2 is then extended to a global one by controlling its moments. We refer to [Theorem 2.1]\cite{DuongTugaut2020} for the detailed argument.
}
\end{proof}
\subsection{Propagation of chaos}
\label{subsec:prop_chaos}

The well-posedness, i.e. existence and uniqueness of a strong solution, together with moment bounds, of the multi-species interacting particle system follows from standard results on SDEs, e.g.\cite[Thm. 10.2.2]{stroock2007multidimensional} or~\cite[Thm. 3.1.1]{Rockner2007}. We note that Eqn~\eqref{eq: many SDEs} is an Itô diffusion on $\bRb^{dN_n}$, with appropriate growth conditions on the drift and, therefore, standard SDE theory applies.

We now state a standard propagation of chaos result that provides us with the link between the interacting particle system~\eqref{eq: many SDEs} and the multi-species McKean SDE~\eqref{eq: mean-field SDE}.

\begin{theorem}[Propagation of chaos]
    \label{theo: PoC}
Under the same assumption as in Theorem \ref{thm: well-posedness}, if moreover the two SDEs are driven by the same Brownian motion, then for $T<\infty$, we have
\begin{equation}
\label{eq: E-sup}
\lim\limits_{n\to\infty}\E\Big[\sup\limits_{t\in[0,T]}\big(X_t^{i,k}-\overline{X}_t^{i,k}\big)^2\Big]=0\quad \text{for}\quad k=1,\ldots, M\quad\text{and all}\quad i\geq 1.
\end{equation}
\end{theorem}
\begin{proof}
\RV{The proof of this theorem extends that of \cite[Theorem 1.2]{DuongTugaut2020} for the case $M=2$. We therefore only sketch the main steps.}

\RV{\textbf{Step 1 (second and fourth moments estimates)}. Let $X^{i,k}_t$ and $\overline{X}^{i,k}_t$ be solutions to \eqref{eq: many SDEs} and \eqref{eq: mean-field SDE} respectively. If $\{N^k_n\}$ tend to $+\infty$ simultaneously, as $n\rightarrow\infty$, but $\frac{N^k_n}{N_n}$ is a constant, then for any $1\leq k\leq M$, it holds that
\begin{equation*}
\sup\limits_{t\in [0,T]}\E\bigg[\Big(X_t^{i,k}-\overline{X}_t^{i,k}\Big)^2\bigg]\leq\frac{C}{N^k_n}\quad\text{and}\quad\sup\limits_{t\in [0,T]}\E\bigg[\Big(X_t^{i,k}-\overline{X}_t^{i,k}\Big)^4\bigg]\leq\frac{C}{(N^k_n)^2},
\end{equation*}
where $C>0$ is a constant.
}

\RV{\textbf{Step 2 (Estimate of $\big(X_t^{i,k}-\overline{X}_t^{i,k}\big)^2$)}
\begin{align*}
\big(X_t^{i,k}-\overline{X}_t^{i,k}\big)^2
&=-2\int_0^t (\nabla V_k(X_s^{i,k})-\nabla V_k(\overline{X}_s^{i,k}))\cdot (X_s^{i,k}-\overline{X}_s^{i,k}) \,ds\notag
\\&\quad-\frac{2}{N_n}\int_0^t \sum_{j=1}^{N_n}\nabla F_{kk}(X_s^{i,k}-X_s^{j,k})-a_k(\nabla F_{kk}\ast\mu^k_s)(\overline{X}_s^{i,k})\cdot (X_s^{i,k}-\overline{X}_s^{i,k}) \,ds\notag
\\&\quad-\frac{2}{N_n}\int_0^t \sum_{\ell\neq k}\Big(\sum_{j=1}^{N_n^\ell}\nabla F_{k
\ell}(X_s^{i,k}-X_s^{j,\ell})-a_\ell(\nabla F_{k\ell}\ast\mu^\ell_s)(\overline{X}_s^{i})\Big)\cdot (X_s^{i}-\overline{X}_s^{i,k}) \,ds
\\&:=L_1+L_2+L_3.
\end{align*}
Using assumptions on $V_k$ and the estimates obtained in Step 1, the three terms $L_1, L_2$ and $L_3$ can be bounded from above by
\begin{align*}
L_1\leq 2\theta_k\int_0^t \Big|X_s^{i,k}-\overline{X}_s^{i,k}\Big|^2\, ds,\quad \mathbb{E}[L_2+L_3]\leq \frac{C}{N_n^k}.
\end{align*}
The claimed statement of the theorem then follows by applying Grönwall's lemma.
}
\end{proof}
\RV{It would be an interesting problem for future work} to obtain quantitative propagation of chaos estimates, along the lines of what is done in~\cite{cattiaux2008,lacker2023sharp}. This will be presented elsewhere.

\subsection{Characterization of stationary states}
\label{subsec:stat_states}

An equilibrium state (stationary state) of the multi-species McKean-Vlasov PDE system \eqref{eq: multi-species PDE} is a collection of probability measures $(\mu^k(\ddx)=\mu^k(x)\,\ddx)_{k=1}^M$ that solve the stationary McKean-Vlasov PDE system:
\begin{equation}
\label{eq: equilibrium}
\div\left[\Big(\nabla V_k+\sum_{\ell=1}^M a_\ell\nabla F_{k\ell}\ast\mu^\ell\Big)\mu^k \right]+\frac{\sigma_k^2}{2}\Delta\mu^k=0,\quad k=1,\ldots M,
\end{equation}
where the coefficients $\{a_i\}_{i=1}^M$ are defined in \eqref{eq: coefficients a}. The following proposition provides a characterization of stationary states. 
\begin{proposition}[Characterization of stationary states]
\label{jugnot}
Under Assumptions \ref{asp: assumption} and \ref{asp: assumptionbis}, a collection of probability measures $(\mu^k(\ddx)=\mu^k(x)\,\ddx)_{k=1}^M \in (\mathcal{P}
(\mathbb{R}^d)\times H^1(\mathbb{R}^d))^M$ is an equilibrium state if and only if
\begin{equation}
\label{eq: muk}
\mu^k(x)=\frac{1}{Z_k}\exp\Big[-\frac{2}{\sigma_k^2}\Big(V_k(x)+\sum_{\ell=1}^M a_\ell F_{k\ell}\ast\mu^\ell(x)\Big)\Big],    
\end{equation}
where the constants $\{ a_\ell \}_{\ell=1}^M$ are defined in~\eqref{eq: coefficients a} and where $Z_k$ is the normalization constant given by
\begin{equation}
\label{eq: Zk}
Z_k:=\int \exp\Big[-\frac{2}{\sigma_k^2}\Big(V_k(x)+\sum_{\ell=1}^M a_\ell F_{k\ell}\ast\mu^\ell(x)\Big)\Big]\,\ddx.
\end{equation}   
\end{proposition}

\begin{proof}
We follow the proof for the single-species system from~\cite{Tamura1984,dressler1987stationary}. We consider the following linearized stationary equation
\begin{equation}
 \label{eq: linearised stationary eqn}
\L_i[\rho_1,\ldots, \rho_M](\mu)=0, 
\end{equation}
where $\rho_1,\ldots \rho_M \in (L^1(\bRb^d))^M$ are given and 
\begin{align*}
\L_i[\rho_1,\ldots, \rho_M](\mu)&=\div\bigg[\Big(\nabla V_i+\sum_{j=1}^M a_j\nabla F_{ij}\ast\rho_j\Big)\mu \bigg]+\frac{\sigma_i^2}{2}\Delta\mu\\
&=\div\bigg[\Big(\nabla V_i+\sum_{j=1}^M a_j\nabla F_{ij}\ast\rho_j +\frac{\sigma_i^2}{2}\frac{\nabla \mu}{\mu}\Big)\mu\bigg],\quad i=1,\ldots M.
\end{align*}
Note that $\L_i$ is a linear operator with respect to $\mu$. 
We define
\[
v_i(x):=\frac{1}{Z_i}\exp\Big[-\frac{2}{\sigma_i^2}\Big(V_i(x)+\sum_{j=1}^M a_j F_{ij}\ast\rho^j(x)\Big)\Big],\quad i=1,\ldots, M,  
\]
where the $Z_i$'s are the normalization constants so that $\|v_i\|_{L^1}=1$. We will show that $v_i$ is the unique solution of \eqref{eq: linearised stationary eqn} in $\mathcal{P}(\mathbb{R}^d)\times H^1(\mathbb{R}^d)$.
In fact, due to the growth conditions of $V_i$ in Assumption~\ref{asp: assumption}, direct computations show that $v_i \in \mathcal{P}(\mathbb{R}^d)\times H^1((\mathbb{R}^d))$ (in fact $v_i$ is smooth, see \cite{Tamura1984}). In addition, we have
\[
\nabla v_i=-\frac{2}{\sigma_i^2}\Big(\nabla V_i+\sum_{j=1}^M a_j \nabla F_{ij}\ast \rho^j\Big) v_i.
\]
Thus
\[
\frac{\sigma_i^2}{2}\frac{\nabla v_i}{v_i}+\Big(\nabla V_i+\sum_{j=1}^M a_j \nabla F_{ij}\ast \rho^j\Big)=0.
\]
It follows that $v_i$ solves the linearized stationary equation \eqref{eq: linearised stationary eqn}. Now suppose this equation has another solution $w_i \in A_i$, that is $\L_i(w_i)=0$. Let $f_i:=w_i v_i^{-1/2}$. Define
\[
\Q_if:=-v_i^{-1/2}\L_i(f v_i^{1/2}).
\]
Thus $\Q_i f_i=-v_i^{-1/2}\L_i(f_i v_i^{1/2})=-v_i^{-1/2}\L_i(w_i)=0$. We have
\begin{align*}
\div\Big( v_i^{-1/2}f_i\nabla v_i\Big)&=-\frac{2}{\sigma_i^2}\div \bigg[ v_i^{-1/2}f_i \Big(\nabla V_i+\sum_{j=1}^M a_j \nabla F_{ij}\ast \rho^j\Big) v_i\bigg]
\\&=-\frac{2}{\sigma_i^2}\div \bigg[ v_i^{1/2}f_i \Big(\nabla V_i+\sum_{j=1}^M a_j \nabla F_{ij}\ast \rho^j\Big)\bigg],
\end{align*}
that is
\[
\frac{\sigma_i^2}{2}\div\Big( v_i^{-1/2}f_i\nabla v_i\Big)+\div \bigg[ v_i^{1/2}f_i \Big(\nabla V_i+\sum_{j=1}^M a_j \nabla F_{ij}\ast \rho^j\Big)\bigg]=0.
\]
Therefore,
\begin{align*}
    \L_i (w_i)&=\div\bigg[\Big(\nabla V_i+\sum_{j=1}^M a_j\nabla F_{ij}\ast\rho_j\Big)w_i+\frac{\sigma_i^2}{2}\nabla w_i \bigg]
    \\&=\div\bigg[\Big(\nabla V_i+\sum_{j=1}^M a_j\nabla F_{ij}\ast\rho_j\Big)f_i v_i^{1/2}+\frac{\sigma_i^2}{2}\nabla ( v_i v_i^{-1/2} f_i ) \bigg]
    \\&=\div\bigg[\Big(\nabla V_i+\sum_{j=1}^M a_j\nabla F_{ij}\ast\rho_j\Big)f_i v_i^{1/2}\bigg]+\frac{\sigma_i^2}{2}\bigg\{\div\Big[v_i\nabla (v_i^{-1/2} f_i)\Big]+\div\Big[(v_i^{-1/2} f_i)\nabla v_i  \Big]\bigg\}
    \\&=\frac{\sigma_i^2}{2}\div\Big[v_i\nabla (v_i^{-1/2} f_i) \Big].
\end{align*}
Hence,
\[
\Q_i f_i=-v_i^{-1/2}\L_i(w_i)=-\frac{\sigma_i^2}{2}v_i^{-1/2}\div\Big[v_i\nabla (v_i^{-1/2} f_i)\Big].
\]
Thus we obtain
\begin{align*}
 0&=\langle \Q_i f_i, f_i\rangle_{L^2}=-\frac{\sigma_i^2}{2}\int_{\bRb^d}v_i^{-1/2}\div\Big[v_i\nabla (v_i^{-1/2} f_i)\Big]\, f_i
 \\&=\frac{\sigma_i^2}{2}\int_{\bRb^d}v_i\Big|\nabla (v_i^{-1/2} f_i)\Big|^2.
\end{align*}
It follows that $\nabla (v_i^{-1/2} f_i)=0$. Hence $ v_i^{-1/2}f_i=C=\mathrm{const}$, so $w_i=v_i^{1/2} f_i=C\, v_i$.  Since $\|w_i\|=\|v_i\|=1$, it implies that $C=1$. That is, $ w_i=v_i$. 

Now suppose that $(\mu^k(\ddx)=\mu^k(x)\ddx)_{k=1}^M$ is a stationary solution of \eqref{eq: multi-species PDE}. Then we have
\[
\mathcal{L}_k[\mu_1,\ldots, \mu_M](\mu_i)=0,\quad{for}\quad k=1,\ldots, M.
\]
According the above statement, $\mu_i$ must satisfy \eqref{eq: muk}. This completes the proof.
\end{proof}

Proposition~\ref{jugnot} provides us with the essential structure of the invariant measures for the multi-species McKean-Vlasov PDE. This Gibbs-like structure, Eqn.~\eqref{eq: muk}, will play a crucial role in the following sections: to prove the existence of an equilibrium state, it is necessary and sufficient to solve the system of equations~\eqref{eq: muk}-\eqref{eq: Zk}.

\subsection{The free-energy functional}
\label{subsec:free_energy}

In this subsection, we {\color{black}first assume only the symmetry Assumption~\ref{asp: assumptionter}, i.e. Newton's third law, is satisfied}. In particular, we have $F_{k\ell}=F_{\ell k}$ for any $1\leq k,\ell\leq M$. This assumption implies that the McKean-Vlasov PDE system \eqref{eq: multi-species PDE} is of gradient flow form. Therefore, we can define the free energy functional of this gradient flow. This functional has been used in \cite{BCCP,CMV} to prove convergence towards a steady state for the single-species model when the confining potential is convex. This was later extended to the non-convex case in \cite{AOP,SPA}. We emphasize the fact that the formal calculations that follow are valid for arbitrary interaction potentials, and that they are not restricted to the quadratic case. In particular, in {\color{black} the first part of this section, we are not imposing Assumption~\ref{asp: assumptionbis}. However, to obtain the lower-bound of the free-energy functionnal, we will need to restrict ourselves to the aforementionned assumption that means the interaction potentials are quadratic}.


For the single-species McKean-Vlasov PDE, the free-energy functional is the following:
\begin{equation}\label{e:free_energy_single_species}
\Upsilon_\sigma(\mu):=\frac{\sigma^2}{2}\int_{\bRb^d}\mu(x)\log(\mu(x))\ddx+\int_{\bRb^d}V(x)\mu(x)\ddx+\frac{1}{2}\iint_{\bRb^d\times\bRb^d}F(x-y)\mu(x)\mu(y)\ddx\ddy\,,
\end{equation}
where $V$ is the confining potential, $F$ is the interacting intra-species potential and $\sigma$ is the diffusion coefficient of the McKean SDE. As is well known, the free energy functional~\eqref{e:free_energy_single_species} provides us with a Lyapunov function for the McKean-Vlasov evolution PDE. We point out that the latter definition requires that $\mu$ is absolutely continuous with respect to the Lebesgue measure (and its density is denoted as $\mu$ also for simplicity). If this hypothesis is not respected, then $\Upsilon_\sigma(\mu):=+\infty$.

We will mimic the strategy in \cite{BCCP,AOP} to {\color{black}describe the limiting set and possibly to obtain the} convergence of solutions to the multi-species mean field PDE system to a stationary state. For this, we need to show that the free energy is nonincreasing and that it is lower bounded. Moreover, we will show that the time derivative of the free energy (i.e. the entropy dissipation functional) evaluated at the solution of the mean-field PDE system $\mu_t$ vanishes if and only if $\mu_t = \mu$ is a stationary state.


For any $\mu=(\mu^1,\ldots,\mu^M)$ where $\mu^\ell$ is a measure in $\bRb^d$ admitting a density with respect to the Lebesgue measure (which we continue to denote as $\mu^\ell$) and for any $\sigma=(\sigma_1,\ldots,\sigma_M)\in\left(\bRb_+^*\right)^M$, we introduce the following quantity that will play the role of a free energy functional for the McKean-Vlasov PDE system \eqref{eq: multi-species PDE}.

\begin{align}
\label{eq: free energy}
\Upsilon_\sigma(\mu)=\Upsilon_\sigma(\mu^1,\ldots,\mu^M):=&\sum_{k=1}^M a_k\left\{\frac{\sigma_k^2}{2}\int_{\bRb^d}\mu^k(x)\log(\mu^k(x))\ddx+\int_{\bRb^d}V_k(x)\mu^k(x)\ddx\right\}\\
&+\frac{1}{2}\sum_{k=1}^M\sum_{\ell=1}^M a_ka_\ell\iint_{\bRb^d\times\bRb^d}F_{k\ell}(x-y)\mu^k(x)\mu^\ell(y)\ddx\ddy\,. \nonumber  
\end{align}
It is worth noting that the free-energy functional can be written in the form
\[
\Upsilon_\sigma(\mu)=\sum_{k=1}^Ma_k\F_k(\mu^k)\,,
\]
where, compare with~\eqref{e:free_energy_single_species},
\[
\F_k(\mu^k):=\frac{\sigma_k^2}{2}\int_{\bRb^d}\mu^k(x)\log(\mu^k(x))\ddx+\int_{\bRb^d}V_k(x)\mu^k(x)\ddx+\frac{1}{2}\int_{\bRb^d}W_k^\mu(x)\mu^k(x)\ddx\,.
\]
Here, the potential $W_k^\mu$ contains all the interactions from the species:
\[
W_k^\mu(x):=\sum_{\ell=1}^Ma_\ell F_{k\ell}\ast\mu^\ell(x)\,.
\]
Thus, the total free-energy of the multi-species mean field system can be expressed as the weighted, with respect to the weights $a_1,\ldots,a_M$, average of the free-energy for each species.

We calculate the functional derivative of the free energy:
\[
\frac{\delta\Upsilon_\sigma}{\delta \mu^k}=a_k\Big[\frac{\sigma_k^2}{2}\left(\log\left(\mu^k\right)+1\right)+V_k+\sum_{\ell=1}^M a_\ell F_{k\ell}\ast \mu^\ell\Big]\,,
\]
which leads us to
\begin{equation*}
\nabla\frac{\delta\Upsilon_\sigma}{\delta \mu^k}= a_k\Bigg[\frac{\sigma_k^2}{2}\frac{\nabla\mu^k}{\mu^k}+\nabla V_k+\sum_{\ell=1}^M a_\ell\nabla F_{k\ell}\ast \mu^\ell\Bigg]\,.
\end{equation*}
In particular, we (formally) deduce that the coupled McKean-Vlasov PDE system~\eqref{eq: multi-species PDE} can be written as a Wasserstein gradient flow 
\[
\partial_t\mu^k=\div\Big[a_k^{-1}\mu^k\nabla\frac{\delta\Upsilon_\sigma}{\delta \mu^k}\Big]\,,
\]
for any $k=1,\ldots,M$.

We now compute the time derivative of the free-energy functional \eqref{eq: free energy} along the McKean-Vlasov PDE~\eqref{eq: multi-species PDE} flow:
\begin{align*}
\frac{\dd}{\ddt}\Upsilon_\sigma(\mu_t^1,\ldots, \mu_t^M)&=\sum_{k=1}^M \int_{\bRb^d}\partial_t\mu_t^k\frac{\delta\Upsilon_\sigma}{\delta \mu_t^k}\,\ddx\\
&=\sum_{k=1}^M\int_{\bRb^d}\left\{\div\Big[\Big(\nabla V_k+\sum_{\ell=1}^M a_\ell\nabla F_{k\ell}\ast\mu_t^\ell\Big)\mu_t^k\Big]+\frac{\sigma_k^2}{2}\Delta\mu_t^k\right\}\frac{\delta\Upsilon_\sigma}{\delta \mu_t^k}\,\ddx\,.
\end{align*}
We perform an integration by parts to obtain:
\begin{align*}
&\frac{\dd}{\ddt}\Upsilon_\sigma(\mu_t^1,\ldots, \mu_t^M)\\
&=-\sum_{k=1}^M\int_{\bRb^d}\Bigg[\Big(\nabla V_k+\sum_{\ell=1}^M a_\ell\nabla F_{k\ell}\ast\mu_t^\ell\Big)\mu_t^k+\frac{\sigma_k^2}{2}\nabla \mu_t^k\Bigg]\cdot\nabla\frac{\delta\Upsilon_\sigma}{\delta \mu_t^k}\,\ddx\\
&=-\sum_{k=1}^M\int_{\bRb^d}a_k\Bigg[\Big(\nabla V_k+\sum_{\ell=1}^M a_\ell\nabla F_{k\ell}\ast\mu_t^\ell\Big)\mu_t^k+\frac{\sigma_k^2}{2}\nabla \mu_t^k\Bigg]\cdot\Bigg[\frac{\sigma_k^2}{2}\frac{\nabla\mu_t^k}{\mu_t^k}+\nabla V_k+\sum_{\ell=1}^M a_\ell\nabla F_{k\ell}\ast \mu_t^\ell\Bigg]\,\ddx\\
&=-\sum_{k=1}^M\int_{\bRb^d}a_k\Big|\frac{\sigma_k^2}{2}\frac{\nabla\mu_t^k}{\mu_t^k}+\nabla V_k+\sum_{\ell=1}^M a_\ell \nabla F_{k\ell}\ast \mu_t^\ell\Big|^2\mu_t^k\,\ddx.
 \end{align*}
Thus, $\Upsilon_\sigma$ is always non-increasing, and 
\begin{equation*}
\frac{\dd}{\ddt}\Upsilon_\sigma(\mu^1,\ldots,\mu^M)=0
\end{equation*}
if and only if for any $1\leq k\leq M$, we have:
\[
\frac{\sigma_k^2}{2}\frac{\nabla\mu^k}{\mu^k}+\nabla V_k+\sum_{\ell=1}^M a_\ell \nabla F_{k\ell}\ast\mu^\ell=0.
\]
This immediately implies that $(\mu^1,\ldots,\mu^M)$ satisfies
\[
\nabla\frac{\delta\Upsilon_\sigma}{\delta\mu^k}=0, \quad 1\leq k\leq M.
\]
We conclude that the time-derivative of the free-energy is equal to $0$ if and only if $(\mu^1,\ldots,\mu^M)$ is a stationary state.

We now show that the free-energy functional is lower-bounded. We restrict ourselves to the symmetric quadratic case; in particular, Assumptions~\ref{asp: assumptionbis} and~\ref{asp: assumptionter} are satisfied. The idea is similar to the one in \cite{BCCP,AOP}. First, we introduce the following quantities:
\[
\beta:={\color{black}\min}\left\{\alpha_{k\ell}\,\,:\,\,1\leq k,\ell\leq M\right\}\,,
\]
and
{\color{black}
\[
V_k^0(x):=V_k(x)+\frac{\min\{\beta;0\}}{2}|x|^2\,.
\]
On the one hand, if $\beta\geq0$, the quantity $\int_{\bRb^d}F_{k\ell}\ast\mu^\ell(x)\mu^k(x)\ddx$ is nonnegative due to the positivity of $\alpha_{k\ell}$. As a consequence,
\begin{align*}
&\sum_{k=1}^Ma_k\int_{\bRb^d}\left(V_k(x)+\frac{1}{2}\sum_{\ell=1}^Ma_\ell F_{k\ell}\ast\mu^\ell(x)\right)\mu^k(x)\ddx&\\
&\geq\sum_{k=1}^Ma_k\int_{\bRb^d}V_k(x)\mu^k(x)\ddx\\
&\geq\sum_{k=1}^Ma_k\int_{\bRb^d}V_k^0(x)\mu^k(x)\ddx\,.
\end{align*}
On the other hand, if $\beta<0$, we put $\rho_{k\ell}:=\alpha_{k\ell}-\beta\geq0$ and $G_{k\ell}(x):=\frac{\rho_{k\ell}}{2}|x|$. Thus, the quantity $\int_{\bRb^d}G_{k\ell}\ast\mu^\ell(x)\mu^k(x)\ddx$ is nonnegative. As a consequence,
\[
\sum_{k=1}^Ma_k\int_{\bRb^d}\frac{1}{2}\sum_{\ell=1}^Ma_\ell F_{k\ell}\ast\mu^\ell(x)\mu^k(x)\ddx\geq\sum_{k=1}^Ma_k\int_{\bRb^d}\frac{\beta}{4}\sum_{\ell=1}^Ma_\ell\int_{\bRb^d}|x-y|^2\mu^\ell(y)\ddy\mu^k(x)\ddx\,.
\]
By expanding then rearranging the terms, we obtain by using the positivity of $-\beta$:
\begin{align*}
&\sum_{k=1}^Ma_k\int_{\bRb^d}\left(V_k(x)+\frac{\beta}{4}\sum_{\ell=1}^Ma_\ell\int_{\bRb^d}|x-y|^2\mu^\ell(y)\ddy\right)\mu^k(x)\ddx\\
&=\sum_{k=1}^Ma_k\int_{\bRb^d}V_k(x)\ddx+\frac{\beta}{2}\sum_{k=1}^Ma_k\int_{\bRb^d}|x|^2\mu^k(x)\ddx-\frac{\beta}{2}\left|\sum_{k=1}^Ma_k\int_{\bRb^d}x\mu^k(x)\ddx\right|^2\\
&\geq\sum_{k=1}^Ma_k\int_{\bRb^d}\left(V_k(x)+\frac{\beta}{2}|x|^2\right)\mu^k(x)\ddx\\
&\geq\sum_{k=1}^Ma_k\int_{\bRb^d}V_k^0(x)\mu^k(x)\ddx\,.
\end{align*}
In any case, the potential part of the free-energy functional may be lower-bounded in the following way:
\[
\sum_{k=1}^Ma_k\int_{\bRb^d}\left(V_k(x)+\frac{1}{2}\sum_{\ell=1}^Ma_\ell F_{k\ell}\ast\mu^\ell(x)\right)\mu^k(x)\ddx\geq\sum_{k=1}^Ma_k\int_{\bRb^d}V_k^0(x)\mu^k(x)\ddx\,.
\]
}
It remains to deal with the entropy part. We have
\begin{align*}
&\int_{\bRb^d}\mu_t^k(x)\log\left(\mu_t^k(x)\right)\ddx=\int_{\bRb^d}\mu_t^k(x)\log\left(\mu_t^k(x)\right)\mathds{1}_{\left\{\mu_t^k(x)\geq1\right\}}\ddx\\
&+\int_{\bRb^d}\mu_t^k(x)\log\left(\mu_t^k(x)\right)\mathds{1}_{\left\{\eee^{-|x|}\leq\mu_t^k(x)\leq1\right\}}\ddx+\int_{\bRb^d}\mu_t^k(x)\log\left(\mu_t^k(x)\right)\mathds{1}_{\left\{\mu_t^k(x)\leq\eee^{-|x|}\right\}}\ddx\\
&\geq-\int_{\bRb^d}|x|\mu_t^k(x)\mathds{1}_{\left\{\eee^{-|x|}\leq\mu_t^k(x)\leq1\right\}}\ddx+\int_{\bRb^d}\sqrt{\mu_t^k(x)}\left[\sqrt{\mu_t^k(x)}\log\left(\mu_t^k(x)\right)\right]\mathds{1}_{\left\{\mu_t^k(x)\leq \eee^{-|x|}\right\}}\ddx\\
&\geq-\frac{1}{2}\int_{\bRb^d}(1+|x|^2)\mu_t^k(x)\ddx-2{\rm e}^{-1}\int_{\bRb^d}\sqrt{\mu_t^k(x)}\mathds{1}_{\left\{\mu_t^k(x)\leq\eee^{-|x|}\right\}}\ddx\,.
\end{align*}
Note that to obtain the last inequality, we have used the fact that, for any $u\in(0;1)$, $\sqrt{u}\log(u)>-2\eee^{-1}$. This immediately implies that
\begin{align}
\mathcal{F}_k(\mu^k)&=\int_{\bRb^d}\left(V_k(x)+\frac{1}{2}\sum_{\ell=1}^Ma_\ell F_{k\ell}\ast\mu^\ell(x)\right)\mu^k(x)\ddx+\frac{\sigma_k^2}{2}\int_{\bRb^d}\mu_t^k(x)\log\left(\mu_t^k(x)\right)\ddx \notag\\
&\geq-{\color{black}\sigma_k^2}\eee^{-1}\int_{\bRb^d}\eee^{-\frac{|x|}{2}}\ddx-\frac{\sigma_k^2}{{\color{black}4}}+\int_{\bRb^d}\left(V_k(x)+\frac{{\color{black}2}\beta-\sigma_k^2}{{\color{black}4}}|x|^2\right)\mu^k(x)\ddx\,.
\label{lowerbound Fk}
\end{align}
We then introduce
\[
\lambda_-:=\inf\left\{V_k(x)+\frac{\beta-\sigma_k^2}{2}|x|^2\,\,:\,\,x\in\bRb^d, 1\leq k\leq M\right\}\,,
\]
which, by multiplying \eqref{lowerbound Fk} by $a_k$ and then adding over $k$ from $1$ to $M$, noting that $\sum_{k=1}^M a_k=1$, leads us to the estimate
\[
\Upsilon_\sigma(\mu^1,\cdots,\mu^M)\geq C+\lambda_-\,.
\]
{\color{black}We stress that $\lambda_-$ is not equal to $-\infty$ due to Assumption~\ref{asp: assumption} and more precisely from (H2) part which implies that for any $\theta>0$ $x\mapsto V_k(x)-\theta |x|^2$ is convex at infinity.}

This completes the proof that the free-energy functional is lower-bounded.

%
%
\section{Invariant probability measures}
\label{sec:inv_meas}

In this section, we study the invariant measure(s) of the mean-field coupled SDEs~\eqref{eq: mean-field SDE}. {\color{black}We recall the following:}

{\color{black}
\[
(F_{k\ell}\ast\mu^\ell)(x)=\alpha_{k\ell}\Big[\frac{1}{2}|x|^2 -x^T\int_{\bRb^d} y \mu^\ell(y)\,\ddy+\frac{1}{2}\int_{\bRb^d} |y|^2 \mu^\ell(y)\,\ddy\Big].    
\]}

We define the mean vectors/first moments of $\mu^\ell$ by
\begin{equation}
 \label{eq: mean value}
 m_\ell := \int_{\bRb^d} x\mu^\ell(x)\,\ddx, \quad \ell=1,\ldots, M.
\end{equation}

Now, we will use the method developed in \cite{Dawson83,HT1,PT}, see also~\cite{PZ2024, PhysRevResearch.5.013078}, to transform the infinite-dimensional fixed point problem \eqref{eq: muk}-\eqref{eq: Zk} into a finite-dimensional one. That is, we will show that the set of stationary states $(\mu^1,\cdots,\mu^k)$ is in bijection with the set of first moments (``magnetizations") $(m_1,\cdots,m_M)\in(\bRb^d)^M$. To this end, substituting the above calculations into \eqref{eq: muk} and \eqref{eq: Zk}, noting that the term $\int_{\bRb^d} |y|^2 \mu^\ell(y)\,dy$ is canceled out, we obtain
\begin{align*}
  \mu^k(x)=\frac{\exp\Big[-\frac{2}{\sigma_k^2}\Big(V_k(x)+\sum_{\ell=1}^M a_\ell \alpha_{k\ell}(\frac{1}{2}|x|^2-x^T m_\ell\Big)\Big]}{\int \exp\Big[-\frac{2}{\sigma_k^2}\Big(V_k(y)+\sum_{\ell=1}^M a_\ell \alpha_{k\ell}(\frac{1}{2}|y|^2-y^T m_\ell\Big)\Big]\,\ddy} . 
\end{align*}
It follows that $\{ m_k, k=1,\ldots,M \}$ satisfy the following system of self-consistency equations
\begin{equation}
 \label{eq: eqns mk}
 m_k=\frac{\int x\exp\Big[-\frac{2}{\sigma_k^2}\Big(V_k(x)+\frac{1}{2}\sum_{\ell=1}^M a_\ell \alpha_{k\ell}|x|^2-x^T\sum_{\ell=1}^M a_\ell \alpha_{k\ell} m_\ell\Big)\Big]\,\ddx}{\int \exp\Big[-\frac{2}{\sigma_k^2}\Big(V_k(x)+\frac{1}{2}\sum_{\ell=1}^M a_\ell \alpha_{k\ell}|x|^2-x^T\sum_{\ell=1}^M a_\ell \alpha_{k\ell} m_\ell\Big)\Big]\,\ddx}\,, \quad k=1,\ldots, M.
\end{equation}
In view of this system of equations, we deduce that, to study the number and structure of invariant measures, it is sufficient to study a fixed-point problem in a space of dimension $dM$. If the particles take values in $\bRb$, we have a fixed-point problem in $\bRb^M$. This is a generalization to the multi-species setting of the classical self-consistency equation for the single species problem that was studied extensively in, e.g.~\cite{Dawson83, shiino1987}.
%
%
\subsection{Existence of Stationary States in the small-noise limit}
\label{smallnoise}

In this section, we focus on the one-dimensional case. \RV{Extending the analysis to arbitrary dimensions is nontrivial and will be investigated in future work.}

Our analysis is based on the self-consistency equations \eqref{eq: eqns mk}, which in the one-dimensional setting are simplified to
\begin{equation}
\label{eq: eqns mk:dim1}
 m_k=\frac{\int x\exp\Big[-\frac{2}{\sigma_k^2}\Big(V_k(x)+\frac{1}{2}\sum_{\ell=1}^M a_\ell \alpha_{k\ell}x^2-x\sum_{\ell=1}^M a_\ell \alpha_{k\ell} m_\ell\Big)\Big]\,\ddx}{\int \exp\Big[-\frac{2}{\sigma_k^2}\Big(V_k(x)+\frac{1}{2}\sum_{\ell=1}^M a_\ell \alpha_{k\ell}x^2-x\sum_{\ell=1}^M a_\ell \alpha_{k\ell} m_\ell\Big)\Big]\,\ddx}\,.
\end{equation}
The key idea is to rely on the small-noise limit, as was done in \cite{HT2} for the single species model. We briefly review this case. Assume that the system admits a stationary measure $\mu^\sigma$, parameterized by the noise strength $\sigma>0$. By studying the family $\{\mu^\sigma\,\,:\,\,\sigma>0\}$, one can show that, under some reasonable hypotheses--the ones that we also make in this paper-- all the possible limits as $\sigma \rightarrow 0$ of this family of stationary measures are of the form $\mu^0:=\sum_{j=1}^q\lambda_j\delta_{\xi_j}$ where $q\geq1$, $\lambda_j>0$ for any $j=1,\ldots,q$ and $\lambda_1+\ldots+\lambda_q=1$, and $\xi_j\in\bRb$ for any $j\in\{1,\cdots,q\}$. 
Moreover, these limiting measures satisfy the following necessary (albeit not sufficient) condition:
\[
\left(V'+F'\ast\mu^0\right)\mu^0=0\,.
\]
Consequently, for any $1\leq j\leq q$, we obtain the following equation
\begin{equation}\label{e:algebraic_1d}
V'(\xi_j)+\sum_{\ell=1}^q\lambda_\ell F'(\xi_j-\xi_\ell)=0\,.
\end{equation}
In this section, we implement this approach to the multi-species problem. We first need to study solutions to some appropriate algebraic equations, which are multi-species counterparts of Equation~\eqref{e:algebraic_1d}; the study of these algebraic equations is postponed to Section~\ref{app:pogacar}. Then, we will show that for some of these solutions, which we will refer to as ``candidates'', provided that the appropriate assumptions are satisfied, we can associate an invariant probability measure for sufficiently small diffusion coefficients.

The general multidimensional problem leads to very complicated algebraic equations for the candidate solutions. This is why we need to restrict our study to the case $q=1$. 

We consider candidate solutions of the form $\mu^{0,j}=\delta_{m_j}$ where $m_j\in\bRb$ and where $\mu^{0,j}$ denotes the limit of a stationary measure related to the species $j$. We introduce the following effective potentials.

\begin{defn}
For any $m=(m_1,\ldots,m_M)$, and for any $1\leq k\leq M$, we set

\[
W_{(m),k}(x):=V_k(x)+\sum_{\ell=1}^Ma_\ell F_{k\ell}(x-m_\ell)\,.
\]

\end{defn}
We will now show that for any $m^0$ such that
\begin{equation}
\label{eq: m-systems}
W'_{(m^0),k}(m^0_k)=0\,,
\end{equation}
for any $1\leq k\leq M$, there is an invariant probability measure arbitrarily close to the measure $\delta_{m^0}$  for sufficiently small noise strengths, and under additional hypotheses--see Assumptions~\ref{yasmeen}, \ref{yasmeen2} below. As we already mentioned, we will employ a strategy similar to the one in \cite{HT1}, adapted to the multi-species case. First, we introduce the following quantities.

\begin{defn}
Consider the following norm in $\bRb^M$: \[ ||m||_\sigma:={\color{black}\max}_{1\leq k\leq M}\frac{|m_k|}{\sigma_k}\,, \]
for any $m=(m_1,\ldots,m_M)$.
\end{defn}

\begin{defn}
For any $\lambda>0$, for any $m=(m_1,\cdots,m_M)\in\bRb^M$ and for any $\sigma=(\sigma_1,\ldots,\sigma_M)\in\left(\bRb_+^*\right)^M$, we introduce the following parallelepiped:
\[
\mathcal{C}^\sigma(m,\lambda):=\prod_{k=1}^M\left[m_k-\lambda\sigma_k,m_k+\lambda\sigma_k\right]\,.
\]

\end{defn}
We note that $\mathcal{C}^\sigma(m,\lambda)$ is the ball of center $m$ and radius $\lambda$ for the norm $||\cdot||_\sigma$. We also note that the parallelepiped $\mathcal{C}^\sigma(m,\lambda)$ reduces to the point $\{m\}$ when 
$\max_{1\leq k\leq M}\sigma_k$ vanishes. Since this quantity is used several times, we introduce the notation
\begin{equation}
\label{sandra}
f(\sigma):={\color{black}\max}_{1\leq k\leq M}\sigma_k\,,
\end{equation}
for any $\sigma=(\sigma_1,\ldots,\sigma_M)\in\left(\bRb_+^*\right)^M$

We now prove that if $f(\sigma)$ is small enough, then the parallelepiped $\mathcal{C}^\sigma(m^0,\lambda)$ contains a solution to the self-consistency equations~\eqref{eq: eqns mk:dim1}, provided that $m^0$ verifies Equation~\eqref{eq: m-systems}. We will need to make additional assumptions.
\begin{assumption}
\label{yasmeen}
We assume that for any $1\leq k\leq M$, $m_k^0$ is the unique global minimizer of the effective potential $W_{(m^0),k}$.
\end{assumption}

We emphasize the fact that this assumption is necessary. To show this, we can proceed with \emph{reductio ad absurdum}. We assume that there is $k_0\in\{1,\ldots,M\}$ such that $m_{k_0}^0$ is not the unique global minimizer of $W_{(m^0),k_0}$. And we assume that despite this fact, there is a solution to \eqref{eq: eqns mk:dim1} in the parallelepiped for any $\sigma$ such that $f(\sigma)$ is small enough. Then, it is immediate that the limiting invariant probability measure for species $k_0$ is $\delta_{m_{k_0}^0}$. Substituting this into \eqref{eq: eqns mk:dim1} and applying the Laplace method (see \cite[Annex]{HT1}), it is easy to see that the limiting probability measure is not $\delta_{m_{k_0}^0}$. In fact, if $m_{k_0}^0$ is not the unique global minimizer, the Gibbs measure $Z^{-1}\exp\left\{-\frac{2}{\sigma_{k_0}^2}W_{(m^0),k_0}\right\}$ cannot converge to $\delta_{m_{k_0}^0}$ when $\sigma_{k_0}$ vanishes.

We will also need to make the following assumption.
\begin{assumption}
\label{yasmeen2}
We assume that for any $1\leq k\leq M$, $m_k^0$ is such that

\begin{equation}
\label{linegalite}
W_{(m^0),k}''(m_k^0)>\sum_{\ell=1}^Ma_\ell\left|\alpha_{k\ell}\right|>0\,.
\end{equation}
\end{assumption}
We find again that $m_k^0$ is a minimizer of $W_{(m^0),k}$ as necessarily the second derivative of the effective potential for the species $k$ is positive in $m_k^0$. However, this technical condition is probably not necessary and is due to our method of proof, rather than to the structure of the multi-species McKean-Vlasov dynamics. Indeed, Equation~\eqref{linegalite} is equivalent to
\begin{equation}
\label{linegalite:csq}
V_k''(m_k^0)>\sum_{\ell=1}^Ma_\ell\left(\alpha_{k\ell}\right)_-\,,
\end{equation}
where $y_-:=\max\{0;-y\}$ for any $y\in\bRb$.

We emphasize the fact that in this work we are not assuming the positivity of the coefficients $\alpha_{k\ell}, \, k, \ell=1, \dots M$, except when $k=\ell$. The positivity of the coefficients, together with Assumption~\ref{yasmeen2} immediately implies that $V_k$ is locally convex around $m^0_k$. In particular, the method we are using cannot deal with the case where $m_k^0$ is an unstable stationary point of $V_k$. We believe that Assumption~\eqref{yasmeen2} can be relaxed.
\begin{prop}
\label{thunder}
Assume that {\color{black}Assumptions}~\ref{asp: assumption} and~\ref{asp: assumptionbis} hold. Let $m^0:=(m_1^0,\cdots,m_M^0)\in\bRb^M$ be a solution of the self-consistency equations~\eqref{eq: m-systems} such that {\color{black}Assumptions}~\ref{yasmeen} and \ref{yasmeen2} are satisfied. Then, for any $\lambda>0$, there exists $\sigma_0(\lambda)>0$ such that for any $\sigma=(\sigma_1,\ldots,\sigma_M)\in\left(\bRb_+^*\right)^M$ satisfying the inequality $f(\sigma)<\sigma_0(\lambda)$, there exists $m=(m_1,\cdots,m_M)\in\mathcal{C}^\sigma(m^0,\lambda)$ satisfying System~\eqref{eq: eqns mk:dim1}.
\end{prop}

\begin{proof}

Let $m$ be an element of $\mathcal{C}^\sigma(m^0,\lambda)$. Then, there exist coordinates $(r_k)_{1\leq k\leq M}$ such that 

\[
m_k=m_k^0+\lambda r_k\sigma_k\,,
\]

with $|r_k|\leq1$. For any $1\leq k\leq M$ and for any $\sigma=(\sigma_1,\cdots,\sigma_M)$ (with $\sigma_l>0$ for any $1\leq l\leq M$), we introduce the function $\varphi_k^{(\sigma)}  : \bRb^M \mapsto \bRb$:

\[
\varphi_k^{(\sigma)}(m):=\frac{\int_\bRb x\exp\left[-\frac{2}{\sigma_k^2}\left(V_k(x)+\frac{1}{2}\sum_{\ell=1}^Ma_\ell\alpha_{k\ell}x^2-x\sum_{\ell=1}^Ma_l\alpha_{k\ell}m_\ell\right)\right]\ddx}{\int_\bRb\exp\left[-\frac{2}{\sigma_k^2}\left(V_k(x)+\frac{1}{2}\sum_{\ell=1}^Ma_\ell\alpha_{k\ell}x^2-x\sum_{\ell=1}^Ma_\ell\alpha_{k\ell}m_\ell\right)\right]\ddx}\,.
\]
We then define  $\Phi^{\sigma}  : \, \bRb^M  \mapsto \bRb^M$ as 
\begin{equation}
\label{littletwon}
\Phi^{\sigma}(m):=\left(\varphi_1^{(\sigma)}(m),\cdots,\varphi_M^{(\sigma)}(m)\right)\,.
\end{equation}
For any $1\leq k\leq M$, applying \cite[Lemma A.4.]{HT1} and \cite[Remark A.5]{HT1} to the functions $U(x):=V_k(x)+\frac{1}{2}\sum_{\ell=1}^Ma_\ell\alpha_{k\ell}x^2-x\sum_{\ell=1}^Ma_\ell\alpha_{k\ell}m_\ell^0$, $f(x):=x$  and $G_k(x):=x$ with the parameters $\mu:=-\lambda\sum_{\ell=1}^Ma_\ell\alpha_{k\ell}r_\ell$, $\epsilon:=\sigma_k^2$ and $\eta_\epsilon:=\sigma_k$ yields the following estimate, as $\sigma_k$ approaches $0$:

\[
\varphi_k^{(\sigma)}(m)=m_k^0-\frac{\sum_{\ell=1}^Ma_\ell\alpha_{k\ell}r_\ell\lambda}{V_k''(m_k^0)+\sum_{\ell=1}^Ma_\ell\alpha_{k\ell}m_\ell^0}\sigma_k+o(\sigma_k)\,.
\]

We point out that we can apply this result, based on Laplace's method, since the assumptions in \cite[Lemma A.4.]{HT1} are satisfied: in particular, $m_k^0$ is the unique global minimizer of $x\mapsto U(x)$ due to Assumption~\ref{yasmeen} and, moreover, the second derivative of $U$ in this point $m_k^0$ is positive according to the Inequality~\eqref{linegalite:csq}. Consequently, for any $\lambda>0$, we find

\[
\left|\varphi_k^{(\sigma)}(m)-m_k^0\right|=\lambda\sigma_k\left|\frac{\sum_{\ell=1}^Ma_\ell\alpha_{k\ell}r_\ell}{V_k''(m_k^0)+\sum_{\ell=1}^Ma_\ell\alpha_{k\ell}m_\ell^0}\right|+o(\sigma_k)\,.
\]

However, due to Assumption~\ref{yasmeen2}, we have

\[
\left|\frac{\sum_{\ell=1}^Ma_\ell\alpha_{k\ell}r_\ell}{V_k''(m_k^0)+\sum_{\ell=1}^Ma_\ell\alpha_{k\ell}m_\ell^0}\right|\leq\frac{\sum_{\ell=1}^Ma_\ell|\alpha_{k\ell}||r_\ell|}{W''_{(m^0),k}(m_k^0)}<1 \, .
\]

Consequently, if $\sigma_k$ is small enough, we have:

\[
\left|\varphi_k^{(\sigma)}(m)-m_k^0\right|<\lambda\sigma_k\,.
\]

The previous inequality is true for any $1\leq k\leq M$ provided that $f(\sigma)$ is small enough. That is, we have proved that for any $\lambda>0$, there exists $\sigma_0(\lambda)>0$ such that for any $\sigma\in\left(\bRb_+^*\right)^M$ satisfying $0<\sigma_k<\sigma_0(\lambda)$ for any $1\leq k\leq M$ the following inclusion holds:

\[
\Phi^{\sigma}\left(\mathcal{C}^\sigma(m^0,\lambda)\right)\subset\mathcal{C}^\sigma(m^0,\lambda)\,,
\]

We point out that $\mathcal{C}^\sigma(m,\lambda)$ is a compact and convex subset of $\bRb^M$. Then, we can apply Schauder's fixed point theorem, see, e.g. \cite[Proposition 4.1.]{HT1} and we obtain the existence of a fixed point to $\Phi^{\sigma}$ in the domain $\mathcal{C}^\sigma(m,\lambda)$; this completes the proof.

\end{proof}

We immediately deduce the following existence result.

\begin{cor}
Under the assumptions of Proposition~\ref{thunder}, and for sufficiently small noise strengths $\sigma_k, \, k=1, \dots M$, there exists an equilibrium point $(\mu^{1,\sigma},\cdots,\mu^{M,\sigma})$ of the mean-field multi-species system, such that the vector
\[
\left(\int_{\bRb}x\mu^{1,\sigma}(x)\ddx,\cdots,\int_{\bRb}x\mu^{M,\sigma}(x)\ddx\right)\in\mathcal{C}^\sigma(m^0,\lambda)\,;
\]
that is, for any $1\leq k\leq M$, we have:
\[
\int_{\bRb}x\mu^{k,\sigma}(x)\ddx\in\left[m_k^0-\lambda\sigma_k,m_k^0+\lambda\sigma_k\right]\,.
\]
In particular,
\[
\lim_{f(\sigma)\to0}\mu^{k,\sigma}=\delta_{m_k^0}\,.
\]

\end{cor}

As we remarked in the previous section, the set of invariant probability measures is in bijection with a subset of $\bRb^M$ (in the one-dimensional case). Each element of this subset corresponds to the first moment of an invariant probability measure. More precisely, finding stationary states $(\mu^{1,\sigma},\cdots,\mu^{M,\sigma})$ is equivalent to finding solutions to the self-consistency equations~\eqref{eq: eqns mk:dim1}. On the other hand, in Proposition ~\ref{thunder} we proved that if the noise strengths are small enough, it is sufficient to consider the zero-noise case. In particular, solving equations~\eqref{eq: eqns mk:dim1} provides invariant probability measures in the small-noise case, provided that Assumptions~\ref{yasmeen} and \ref{yasmeen2} are satisfied.

%
\subsection{Existence and uniqueness of stationary states in the large-noise regime}
\label{bignoise}
In this section, we consider the large noise strength/high "temperature" regime, in which we can prove the uniqueness of invariant measures for the multi-species mean-field system. 
We first introduce two additional assumptions that will be used in the proof of this result.

In order to be able to employ the techniques developed in~\cite{PT}, the convexity of all of the interaction potentials is necessary.

We now present the main result of this section.

\begin{thm}
\label{lola}
Under Assumptions~\ref{asp: assumption},~\ref{asp: assumptionbis},~\ref{severine} and~\ref{severine2}, there exists a $\sigma_0>0$ such that for any $\sigma\in\left(\bRb_+^*\right)^M$, if $\inf_{1\leq k\leq M}\sigma_k\geq\sigma_0$, the unique solution to the self-consistency system~\eqref{eq: eqns mk:dim1} is $m_k=0$ for any $1\leq k\leq M$. Consequently, the multi-species McKean-Vlasov diffusion admits a unique stationary state with zero `magnetization' vector.
\end{thm}

\begin{proof}
For any $1\leq k\leq M$, we put
\[
\overline{\alpha}_k:=\sum_{\ell=1}^Ma_\ell\alpha_{k\ell} \quad  \mbox{and} \quad
A_k:=\sum_{\ell=1}^Ma_\ell\alpha_{k\ell}m_\ell\,,
\]
where $m=(m_1,\ldots,m_M)\in\bRb^M$. We notice that \eqref{eq: eqns mk:dim1} can be written  as:
\[
m_k=\frac{\int_\bRb x\exp\left\{-\frac{2}{\sigma_k^2}\left[V_k(x)+\frac{\overline{\alpha_k}}{2}x^2-A_kx\right]\right\}\ddx}{\int_\bRb\exp\left\{-\frac{2}{\sigma_k^2}\left[V_k(x)+\frac{\overline{\alpha_k}}{2}x^2-A_kx\right]\right\}\ddx}\,.
\]
We immediately deduce that $m_k$ and $A_k$ have the same sign if $m$ is a solution to \eqref{eq: eqns mk:dim1}.

Consider now $k_0\in\{1,\ldots,M\}$ such that for any $1\leq\ell\leq M$, we have $|m_\ell|\leq|m_{k_0}|$. Due to the positivity of $\alpha_{k\ell}$ for any $1\leq k,\ell\leq M$ (thanks to Assumption~\ref{severine2}), we obtain the following estimate, provided that $|m_{k_0}|>0$:
\[
\frac{|m_{k_0}|}{|A_{k_0}|}\geq\frac{|m_{k_0}|}{\sum_{\ell=1}^Ma_\ell\alpha_{k_0\ell}|m_\ell|}\geq\frac{1}{\overline{\alpha_{k_0}}}\,.
\]
In particular, we deduce
\[
\frac{|A_{k_0}|}{\overline{\alpha}_{k_0}}\leq\frac{\int_\bRb x\exp\left\{-\frac{2}{\sigma_{k_0}^2}\left[V_{k_0}(x)+\frac{\overline{\alpha}_{k_0}}{2}x^2-A_{k_0}x\right]\right\}\ddx}{\int_\bRb\exp\left\{-\frac{2}{\sigma_{k_0}^2}\left[V_{k_0}(x)+\frac{\overline{\alpha}_{k_0}}{2}x^2-A_k{k_0}\right]\right\}\ddx}\,.
\]
Without any loss of generality, we assume that $k_0=1$ and that $A_{1}>0$. Then, the latter inequality becomes
\begin{equation}
\label{thea}
\int_\bRb\left(x-\frac{A_1}{\overline{\alpha}_1}\right)\exp\left\{-\frac{2}{\sigma_1^2}\left[V_1(x)+\frac{\overline{\alpha}_1}{2}x^2-A_1x\right]\right\}\ddx\geq0\,.
\end{equation}

By $U(x)$, we denote the quantity $V_1(x)+\frac{\overline{\alpha}_1}{2}x^2$. We point out that the potential $U$ is even. We set $\sigma:=\sigma_1$. We can rewrite the above inequality as follows:
\[
\int_\bRb\left(x-\frac{A_1}{\overline{\alpha}_1}\right)\exp\left\{-\frac{2}{\sigma^2}U(x)\right\}\eee^{\frac{2A_1}{\sigma^2}x}\ddx\geq0\,.
\]
We now follow a strategy similar to the one in \cite{PT}. In particular, we consider a series expansion of $x\mapsto\eee^{\frac{2A_1}{\sigma^2}x}$. From this we deduce that
\[
\sum_{k=0}^\infty\frac{1}{k!}\left(\frac{2A_1}{\sigma_1^2}\right)^k\left(I_{k+1}(\sigma)-\frac{A_1}{\overline{\alpha}_1}I_k(\sigma)\right)\geq0\,,
\]
where $I_\ell(\sigma):=\int_\bRb x^\ell\exp\left[-\frac{2}{\sigma^2}U(x)\right]\ddx$ for any $\ell\in\bNb$. Since $U$ is even, we deduce that $I_{2\ell+1}(\sigma)=0$ for any $\ell\in\bNb$. The inequality becomes
\[
\sigma^2\sum_{\ell=0}^\infty\frac{I_{2\ell}(\sigma)}{(2\ell)!}\left(\frac{2A_1}{\sigma^2}\right)^{2\ell+1}\left(\frac{I_{2\ell+2}(\sigma)}{(2\ell+1)\sigma^2I_{2\ell}(\sigma)}-\frac{1}{\overline{\alpha}_1}\right)\geq0\,.
\]
In \cite{PT}, it was shown that the sequence $\ell\mapsto\frac{I_{2\ell+2}(\sigma)}{(2\ell+1)\sigma^2I_{2\ell}(\sigma)}$ does not increase. We will now show that for sufficiently large $\sigma$, the first term of the sequence is negative. We consider the following function
\[
\zeta(\sigma):=\frac{I_2(\sigma)}{\sigma^2I_{0}(\sigma)} = \frac{\int_\bRb x^2\eee^{-\frac{2U(x)}{\sigma^2}}\ddx}{\sigma^2\int_\bRb\eee^{-\frac{2U(x)}{\sigma^2}}\ddx}\,.
\]
Then, by making the change of variables $x:=\sigma y$, we find
\[
\zeta(\sigma)=\frac{\int_\bRb y^2\exp\left(-2\frac{U(\sigma y)}{\sigma^2}\right)\ddy}{\int_\bRb\exp\left(-2\frac{U(\sigma y)}{\sigma^2}\right)\ddy}\,.
\]
Using the form of $U$, we obtain
\[
\zeta(\sigma)=\frac{\int_\bRb y^2\exp\left\{-2\left(\sum_{\ell=2}^{p_1}\frac{\theta_1^{(2\ell)}}{2\ell}\sigma^{2\ell-2}y^{2\ell}+\frac{\overline{\alpha}_1-\theta_1^{(2)}}{2}y^2\right)\right\}\ddy}{\int_\bRb\exp\left\{-2\left(\sum_{\ell=2}^{p_1}\frac{\theta_1^{(2\ell)}}{2\ell}\sigma^{2\ell-2}y^{2\ell}+\frac{\overline{\alpha}_1-\theta_1^{(2)}}{2}y^2\right)\right\}\ddy}\,.
\]
By a straightforward application of Laplace's method (see \cite{HT1,HT2,HT3}), we deduce that
\[
\lim_{\sigma\to+\infty}\zeta(\sigma)={\color{black}0}\,.
\]
In particular, since the function $\zeta$ is continuous, there exists $\sigma_0$ such that for any $\sigma\geq\sigma_0$, we have
\[
\zeta(\sigma){\color{black}<}\frac{1}{\overline{\alpha}_1}\,,
\]
and we conclude that $\frac{I_{2\ell+2}(\sigma)}{(2\ell+1)\sigma^2I_{2\ell}(\sigma)}{\color{black}-\frac{1}{\overline{\alpha}_1}}$ is negative for any $\ell\in\bNb$. This implies that inequality~\eqref{thea} is not valid and that, in particular, assumption $|m_{k_0}|>0$ is wrong. We deduce that the only possible solution is $m_k=0$ for any $1\leq k\leq M$. We can easily verify that it corresponds to a solution since $V_k$ is even for any $k$ and since then $A_k=0$ for any $k$.

\end{proof}

\subsection{Phase transitions and Non-Uniqueness of stationary sates}
\label{subsec:pha_tran_non_uniq}

In Section~\ref{smallnoise}, it was shown that, under appropriate assumptions on the confining and interaction potentials, each solution to the self-consistency equations~\eqref{eq: eqns mk:dim1} provided us with a stationary state, provided that the noise strengths are small enough. In particular, more than one stationary state might exist. In addition, in Section~\ref{bignoise}, we proved--under additional assumptions--the existence and uniqueness of a stationary state for sufficiently large diffusion coefficients. Therefore, a natural question is whether a phase transition occurs. This is the result that was proved in, e.g. ~\cite{Dawson83,PT, shiino1987}, for the single species case under hypotheses similar to Assumption~\ref{severine} and Assumption~\ref{severine2}. The purpose of the present section is to extend these results to the multi-species case. A general result for the problem of $M$ species appears to be beyond reach at present, at least using the techniques of~\cite{PT}. We will restrict ourselves to the case where the structural condition in Assumption~\ref{asp: assumptionterbis} is satisfied.
\begin{thm}
\label{negan}
Under Assumptions~\ref{asp: assumption},~\ref{asp: assumptionbis},~\ref{asp: assumptionterbis},~\ref{severine} and~\ref{severine2}, a phase transition occurs. More precisely, each stationary state is of the form $\mu_\sigma\otimes\ldots\otimes\mu_\sigma$ where $\mu_\sigma$ is a given probability measure on $\bRb$ such that
\[
\mu_\sigma(\ddx)=\frac{\exp\left\{-\frac{2}{\sigma^2}\left(\overline{V}^0(x)-Ax\right)\right\}}{\int_\bRb\exp\left\{-\frac{2}{\sigma^2}\left(\overline{V}^0(y)-Ay\right)\right\}\ddy}\,\ddx\,,
\]
where $\overline{V}^0(x)=\overline{V}(x)+\frac{\overline{\alpha}_1}{2}x^2$ with $\overline{\alpha}_1:=\sum_{\ell=1}^Ma_\ell\alpha_{1\ell}$ and $A$ is satisfying
\[
\frac{A}{\overline{\alpha}_1}=\frac{\int_\bRb x\exp\left\{-\frac{2}{\sigma^2}\left(\overline{V}^0(x)-Ax\right)\right\}\ddx}{\int_\bRb\exp\left\{-\frac{2}{\sigma^2}\left(\overline{V}^0(x)-Ax\right)\right\}\ddx}\,.
\]
Moreover, there exists a $\sigma_c>0$ such that if $\sigma \geq\sigma_c$, the system admits a unique stationary state ($A=0$) and if $\sigma<\sigma_c$, there are exactly three stationary states associated to $A=0$, $A>0$ and $A<0$). Furthermore, the critical value $\sigma_c$ is the unique positive real number which satisfies
\begin{equation}\label{e:critical_temp}
\frac{\int_\bRb x^2\exp\left\{-\frac{2}{\sigma_c^2} \overline{V}^0(x)\right\}\ddx}{\int_\bRb\exp\left\{-\frac{2}{\sigma_c^2} \overline{V}^0(x)\right\}\ddx}=\frac{\sigma_c^2}{2\overline{\alpha}_1}\,.
\end{equation}
\end{thm}

\begin{proof}
From Assumption~\ref{asp: assumptionterbis}, it follows that the invariant probability measure associated with species $k$ is the following Gibbs measure:
\[
Z_k^{-1}\exp\left\{-\frac{2}{\sigma_k^2}\left(V_k^0(x)-x\sum_{\ell=1}^Ma_\ell\alpha_{k\ell}m_\ell\right)\right\}\,\ddx\,,
\]
where $m=(m_1,\ldots,m_M)$ is a solution to the system~\eqref{eq: eqns mk:dim1} and $V_k^0(x)=V_k(x)+\frac{\overline{\alpha_k}}{2}x^2$ with $\overline{\alpha}_k:=\sum_{\ell=1}^Ma_\ell\alpha_{k\ell}$. We note that, under the structural assumption \ref{asp: assumptionterbis}, $\frac{V_k^0}{\sigma_k^2}=\frac{\overline{V}^0}{\sigma^2}$ and $\frac{\sum_{\ell=1}^Ma_\ell\alpha_{k\ell}m_\ell}{\sigma_k^2}=\frac{\sum_{\ell=1}^Ma_\ell\alpha_{1\ell}m_\ell}{\sigma^2}$. Consequently, if $m$ is a solution of the self-consistency equations~\eqref{eq: eqns mk:dim1}, we can write
\[
m_k=\frac{\int_\bRb x\exp\left\{-\frac{2}{\sigma^2}\left(\overline{V}^0(x)-Ax\right)\right\}\ddx}{\int_\bRb\exp\left\{-\frac{2}{\sigma^2}\left(\overline{V}^0(x)-Ax\right)\right\}\ddx}\,,
\]
where $A=\sum_{\ell=1}^Ma_\ell\alpha_{1\ell}m_\ell$. In particular, we have $m_k=m_1=\frac{A}{\overline{\alpha}_1}$. The self-consistency equation becomes
\begin{equation}
\label{marion}
\frac{A}{\overline{\alpha}_1}=\frac{\int_\bRb x\exp\left\{-\frac{2}{\sigma^2}\left(\overline{V}^0(x)-Ax\right)\right\}\ddx}{\int_\bRb\exp\left\{-\frac{2}{\sigma^2}\left(\overline{V}^0(x)-Ax\right)\right\}\ddx}\,.
\end{equation}

This completes the proof of the first part of the theorem. To prove the existence of a phase transition, we apply the techniques developed in the proof of \cite[Theorem 2.1.]{PT}. We use the series expansion of $x\mapsto \eee^{\frac{2Ax}{\sigma^2}}$ introduced earlier; by carefully studying the different terms in the expansion, we can show that the function 
$$
\psi(A):=\int_\bRb\left(x-\frac{A}{\overline{\alpha}_1}\right)\exp\left\{-\frac{2}{\sigma^2}\left(\overline{V}^0(x)-Ax\right)\right\}\ddx
$$
is either odd and decreases on $\bRb_+$ or increases and then decreases toward $-\infty$, as explained in Step 3 of the proof of  \cite[Theorem 2.1.]{PT}.  Since $\psi(A)=0$, we always have a solution with $A=0$. By following carefully the steps described in~\cite{PT}, it is straightforward to prove that the function $\psi$ is decreasing if $\sigma$ is greater than some critical value $\sigma_c$. Furthermore, it increases and then decreases if $\sigma$ is strictly less than this value $\sigma_c$. This completes the second part of the proof. The proof of the last part (the characterization of the critical value) proceeds as in~\cite{PT}.

\end{proof}
We emphasize that we do not really use Assumption~\ref{severine2} in Theorem~\ref{negan}. In particular, we could allow the coefficients $\alpha_{k\ell}$ to be negative. The only coefficient that enters into the proof of the theorem is $\overline{\alpha}_1$. Three cases may occur:
\begin{itemize}
 \item Case I: $\overline{\alpha}_1>0$. Then, the statement of Theorem~\ref{negan} holds.
 \item Case II: $\overline{\alpha}_1<0$. Then, the self-consistency equation~\eqref{marion} does not have a solution for any $A\neq0$. In fact, since $A$ and the expression $\frac{\int_\bRb x\exp\left\{-\frac{2}{\sigma^2}\left(\overline{V}^0(x)-Ax\right)\right\}\ddx}{\int_\bRb\exp\left\{-\frac{2}{\sigma^2}\left(\overline{V}^0(x)-Ax\right)\right\}\ddx}$ have the same sign, the quantity $\frac{A\sigma_c^2}{2\overline{\alpha}_1}$ cannot be equal to $\frac{\int_\bRb x\exp\left\{-\frac{2}{\sigma^2}\left(\overline{V}^0(x)-Ax\right)\right\}\ddx}{\int_\bRb\exp\left\{-\frac{2}{\sigma^2}\left(\overline{V}^0(x)-Ax\right)\right\}\ddx}$ provided that $\overline{\alpha}_1<0$. Consequently, for all values of $\sigma>0$, there is always a unique stationary state that is the symmetric one corresponding to $A = 0$.
 \item Case III: $\overline{\alpha}_1=0$. Since all $m_k$'s are equal, we immediately deduce $A=\overline{\alpha}_1 m_1=0$ so $A$ is necessary equal to $0$; consequently, there is again a unique symmetric stationary state for all values of the noise strength.
\end{itemize}
To end this section, we present an example with an ``immigration'' scenario 
where several species with small (mutant/immigrant) populations interact with a species with a much larger (resident) population. We will assume that the interaction between species with small populations is different from the dominant one.
\begin{example}[Immigration scenario]
{\color{black}
Consider the case where $a_1=\ldots=a_{M-1}=\varepsilon$, $a_M=1-\varepsilon (M-1)$, $\alpha_1=\ldots=\alpha_{M-1}=\alpha$ and $\alpha_M\neq\alpha$, where $\varepsilon \in (0,1)$ sufficiently small. In this case, $\overline{\alpha}_1=(M-1)\varepsilon\alpha+\alpha_M(1-\varepsilon(M-1))=\alpha_M+\varepsilon(M-1)(\alpha-\alpha_M)$. We will focus on the critical value for the phase transition with and without immigration. To do so, we first remark that the function
\[
\sigma\mapsto\frac{\int_\bRb x^2\exp\left\{-\frac{2}{\sigma^2}\overline{V}^0(x)\right\}\ddx}{\int_\bRb\exp\left\{-\frac{2}{\sigma^2}\overline{V}^0(x)\right\}\ddx}
\]
is increasing. If $\alpha<\alpha_M$ (that is $\overline{\alpha}_1<\alpha_M$ and so $\frac{1}{\overline{\alpha}_1}>\frac{1}{\alpha_M}$), then immigration raises the critical value and favors multiple invariant probability measures, indicating greater dynamical complexity. Conversely, if $\alpha>\alpha_M$, it lowers the critical value and favors uniqueness of the stationary state. From a biological point of view, a unique stationary state indicates that immigration drives the system toward a single robust long-term community structure. In contrast, multiple stationary states suggest that the system can sustain alternative long-term community structures, with the eventual outcome depending on the initial population composition and other factors. This coexistence of several stable configurations, with the one eventually realised selected by the system's history, is precisely the notion of \emph{alternative stable states} in ecology~\cite{beisner2003alternative,scheffer2001catastrophic}. We refer the reader to \cite{may1972will,allesina2015stability} for further discussion on the stability-complexity relationship in ecology.





}

\end{example}

\section{Linearization of the multi-species McKean-Vlasov System}
\label{sec:lin_stab}

In this section, we characterize the null space of the linearized multi-species McKean-Vlasov operator and perform the linear stability analysis for the nonlinear nonlocal McKean-Vlasov PDE systems \eqref{eq: multi-species PDE} around the equilibrium distributions.  We will occasionally use the notation $\beta^{-1} = \frac{\sigma^2}{2}.$

We start by linearizing the general system~\eqref{eq: multi-species PDE} around (one of) the stationary states described by Proposition~\ref{jugnot}, see Equation~\eqref{eq: muk}. Upon setting $\mu_t^i = \mu_{\infty}^i + \epsilon \mu_t^{L,i}$, using the equations for the invariant measures and collecting terms of $O(\epsilon)$) the linearized system takes the form 
\begin{equation}
\partial_t \mu_t^{L,i} = \frac{\sigma_i^2}{2} \Delta \mu_t^{L,i} +  {\color{black}{\rm div} }\left(\mu_t^{L,i} \nabla V_i\right) 
+  {\color{black}{\rm div} }\left( \left(\sum_{j=1}^M a_j \nabla F_{ij} {\color{black}\ast} \mu_t^{L,{\color{black}j}} \right) \mu_{\infty}^{i} + \left(\sum_{j=1}^M a_j \nabla F_{ij} {\color{black}\ast} \mu_{\infty}^{i} \right) \mu_{t}^{L,i} \right),  \label{e:lin_gen}
\end{equation}
for $i=1,\dots M$. The stability of stationary states is determined by the spectral gap of the matrix-valued integrodifferential operator defined above. It is possible to show that, under the assumption of uniqueness of a stationary state, the nonlinear, nonlocal McKean-Vlasov PDE is exponentially close in time (e.g., in relative entropy) to the linearized system~\eqref{e:lin_gen}. In particular, we can extend the recent "stability" result from~\cite{pavliotis2025linearization} to the multi-species case.\footnote{Note, however, that the linearization in~\cite{pavliotis2025linearization} is done at the level of the McKean SDE, rather than at the level of the McKean-Vlasov PDE. The linearization of the PDE leads to an additional term, namely the third term on the righthand side of~\eqref{e:lin_gen}. The details will be presented elsewhere.} Furthermore, it should be possible to extend the local convergence results obtained in~\cite{monmarché2024localconvergencerateswasserstein} to the multidimensional case--this is certainly clear for the multi-species Desai--Zwanzig model, in particular under the structural assumption. The analysis will be presented in a future publication.

In the rest of this section, we will consider the multi-species Desai--Zwanzig model in $d=1$, under the structural assumption, Assumption~\ref{asp: assumptionterbis}. In particular, we place ourselves in the context of Theorem~\ref{negan}. The linearized McKean-Vlasov equation in this case, see Equations~\eqref{e:structural} and~\eqref{e:struct_generator}, become
\begin{equation} \label{e:linearized_struct}
\partial_t \mu_t^{L,i}=\div\Big[\Big(\nabla\bar{V}+\sum_{j=1}^Ma_j\alpha_j(x-m_j^{\infty})\Big) \mu_t^{L,i} - \sum_{j=1}^Ma_j\alpha_j m_j(t) \mu_{\infty}^i \Big]+\frac{\sigma^2}{2}\Delta \mu_t^{L,i},
\end{equation}
where $m_{\infty}^i = \int x \mu_{\infty}^i \, dx$.

The main result of this section is the following.
\begin{prop}\label{prop:null_space_multi}
Consider the multi-species system in $d=1$ and assume that Assumptions~\ref{asp: assumption}, \ref{asp: assumptionbis},~\ref{asp: assumptionterbis},~\ref{severine} and~\ref{severine2} hold. Then, for the bistable potential $\bar{V} = \frac{x^4}{4} - \frac{x^2}{2}$:

(a) For $\sigma > \sigma_c$, the linearized system~\eqref{e:linearized_struct} is stable, that is, perturbations from the unique stationary distribution decay to zero exponentially fast in $L^2(\R , \rho_{\infty})$.

(b) At the critical interaction strength $\sigma = \sigma_c$ given by~\eqref{e:critical_temp}, the null space of the linearized McKean-Vlasov operator is two-dimensional and consists of $\{ \mu_\sigma, \ldots ,\mu_\sigma \}$ and $\{ x \mu_\sigma, \ldots, x \mu_\sigma \}$.
\end{prop}
In view of the structural assumption, the analysis of this problem follows from~\cite[Sec. 3.4]{Dawson83} for the single species case; we review this case next.

\paragraph{Review of the single-species case from~\cite[Sec. 3.4]{Dawson83}}
 
The single-species model with quadratic interaction potential and bistable confining potential is referred to in the literature as the  Desai--Zwanzig model. We consider the case of the mean-field interaction  $f(x) = \frac{1}{2} x^{2}$. The invariant density of the McKean-Vlasov-Fokker-Planck equation can be written as
\begin{equation}\label{e:stat-DZ}
\rho_{s}(x ; \beta , m)  = \frac{1}{Z(\beta ; m)} \eee^{-\beta \left( V(x) + \theta \left( \frac{x^{2}}{2} - x m \right) \right)}, \quad Z(\beta ; m) = \int \eee^{-\beta \left(V(x) + \theta \left(\frac{x^{2}}{2} - x m \right) \right)} \, \ddx,
\end{equation}
where $m$ denotes the first stationary moment
$
m = \int x \rho_{s}(x ; \beta , m) \, \ddx.
$ 
We will use the notation $\langle \cdot \rangle_{s} := \int \cdot  \rho_{s}(x ; \beta , m) \, \ddx$.
Combining the above two equations, we obtain the self-consistency equation
\begin{equation}\label{e:self-consist}
m = R(m), \quad R(m) := \frac{1}{Z(\beta ; m)} \int x \eee^{-\beta \left(V(x) + \theta \left(\frac{x^{2}}{2} - x m \right) \right)} \, \ddx.
\end{equation}
We differentiate the self-consistency equation with respect to $m$ to obtain an equation for the critical temperature $\beta_{c}^{-1}$ at which the bifurcation occurs. We first calculate the derivative of the partition function with respect to the order parameter $m$: $
\frac{\partial Z(\beta ; m)}{\partial m} = \beta \theta \langle  x\rangle_{s}.
$

Now we have:
\begin{eqnarray*}
1 & = & \int x \left(- \frac{1}{Z^{2}(\beta ; m)} \frac{\partial Z(\beta ; m)}{\partial m} \eee^{-\beta \left( V(x) + \theta \left( \frac{x^{2}}{2} - x m \right) \right)} \right) \, \ddx \\ &&+ \frac{1}{Z(\beta ; m)} \beta  \theta \int x^{2} \eee^{-\beta \left(V(x) + \theta \left(\frac{x^{2}}{2} - x m \right) \right)} \, \ddx
\\ & = &  \beta \theta \left( -\langle x \rangle^{2}_{s} + \langle x^{2} \rangle_{s}  \right).
\end{eqnarray*}
Consequently, the equation for the critical temperature is
\begin{equation}\label{e:critic-temper}
\mbox{Var}_{s}(x) = \beta^{-1} \theta^{-1},
\end{equation}
where $\mbox{Var}_{s}$ denotes the stationary variance at $m=0$. This is the analog of Equation~\eqref{e:critical_temp} for the multiscpecies case, under the structural assumption. The expectation is calculated at $m=0$. We will sometimes use the notation $ \langle \cdot \rangle = \mathbb{E}_{\sigma, \theta, m=0} \cdot$. 

We linearize the Mckean-Vlasov operator around the stationary state $\rho_{\infty}^{\beta, \theta, m=0}$, i.e. below the phase transition:
\begin{equation}\label{e:FP_lin}
\mathcal{L}^*_{lin} p = \beta^{-1} p'' + (V_{\theta}' p)' - \theta \Big(\int y p(y)\,dy\Big) \partial_x \rho_{\infty}^{\beta, \theta, m=0}\,.
\end{equation}
where $V_\theta=V+\theta x^2/2$. We can now characterize the null space of the linearized McKean-Vlasov operator at the critical temperature.
\begin{lemma}
We have that 
$$
\mathcal{L}^*_{lin}  \rho_{\infty}^{\beta, \theta, m=0} = 0 .
$$
\end{lemma}

\begin{proof}
It follows from direct calculation:
$$
\mathcal{L}^*_{lin}  \rho_{\infty}^{\beta, \theta, m=0} = \beta^{-1} \partial_x^2 \rho_{\infty}^{\beta, \theta, m=0} + \partial_x (V_{\theta}' \rho_{\infty}^{\beta, \theta, m=0}) - \theta \langle  x \rangle \partial_x \rho_{\infty}^{\beta, \theta, m=0} = 0,  
$$
since the sum of the first two terms gives zero and $\langle x \rangle = 0$. 
\end{proof}

The null space of the linearized McKean-Vlasov operator is two-dimensional 
at the critical temperature.

\begin{lemma}\label{lem:null_space}
We have that 
$$
\mathcal{L}^*_{lin}  (x \rho_{\infty}^{\beta_c, \theta_c, m=0}) \vert_{(\beta \theta) = \beta^c \theta^c} = 0 .
$$
\end{lemma}
\begin{proof}
We calculate:
\begin{eqnarray*}
&&\mathcal{L}^*_{lin}  (x \rho_{\infty}^{\beta_c, \theta_c, m=0}) \vert_{(\beta \theta) = \beta^c \theta^c} 
\\
& = & 
\beta_c^{-1} (x \rho_{\infty}^{\beta_c, \theta_c, m=0}) '' + (V'_{\theta_c} x  \rho_{\infty}^{\beta_c, \theta_c, m=0} )' - \theta_c  \left( \int y^2 \rho_{\infty}^{\beta_c, \theta_c, m=0} \, dy \right) \partial_x \rho_{\infty}^{\beta_c, \theta_c, m=0}
\\
& = & 
\beta_c^{-1} \left( 2 \partial_x \rho_{\infty}^{\beta_c, \theta_c, m=0}  + x \partial_x^2 \rho_{\infty}^{\beta_c, \theta_c, m=0} \right) + x (V'_{\theta_c} \rho_{\infty}^{\beta_c, \theta_c, m=0})' 
\\ && + 
V'_{\theta_c} \rho_{\infty}^{\beta_c, \theta_c, m=0} - \theta_c \langle y^2 \rangle_{\rho_{\infty}^{\beta_c, \theta_c, m=0}}  \partial_x \rho_{\infty}^{\beta_c, \theta_c, m=0}
\\ & =& 
x \underbrace{\left(\beta_c^{-1}  \partial^2_x \rho_{\infty}^{\beta_c, \theta_c, m=0}  +  \partial_x (V'_{\theta_c} \rho_{\infty}^{\beta_c, \theta_c, m=0}) \right)}_{=0} + \underbrace{\left(\beta^{-1}_c\partial_x\rho_{\infty}^{\beta_c, \theta_c, m=0}+V'_{\theta_c}\rho_{\infty}^{\beta_c, \theta_c, m=0}\right)}_{=0}\\
&&
+\left(\beta_c^{-1} \partial_x \rho_{\infty}^{\beta_c, \theta_c, m=0} - \theta_c \langle y^2 \rangle_{\rho_{\infty}^{\beta_c, \theta_c, m=0}}  \partial_x \rho_{\infty}^{\beta_c, \theta_c, m=0}\right)
\\ & = &
\beta_c^{-1} \partial_x \rho_{\infty}^{\beta_c, \theta_c, m=0} 
- \theta_c \langle y^2 \rangle_{\rho_{\infty}^{\beta_c, \theta_c, m=0}}  \partial_x \rho_{\infty}^{\beta_c, \theta_c, m=0}
\\ & = & 
0,
\end{eqnarray*}
in view of~\eqref{e:critic-temper}. 

\end{proof}

\begin{remark}\label{rem:spectral_gap}
From the argument in~\cite[p. 57]{Dawson83}, it follows that $\{\rho_{\infty}^{\beta_c, \theta_c, m=0}, x \rho_{\infty}^{\beta_c, \theta_c, m=0} \}$ are the only two linearly independent solutions of equation $\mathcal{L}^*_{lin} f = 0$.
\footnote{Alternatively, we can look for solutions to $\mathcal{L}^*_{lin} f = 0$ of the form $f = h \rho_{\infty}^{\beta, \theta, m=0}$ to obtain $\mathcal{L}_{lin} h := \beta^{-1} h'' - V'_{\theta} \big(h' - \beta \theta \int y h \rho_{\infty}^{\beta, \theta, m=0} \, dy \big) = 0 $. We immediately confirm that $\mathcal{L}_{lin} 1 = 0$ for all values of $\beta, \theta$ and that $\mathcal{L}_{lin} x = 0$ for $\beta \theta = \beta_c \theta_c$. In addition, we readily check that $\mathcal{L}_{lin} x^n \neq 0$ for all $n =2,3 \dots$, and for $V_{\theta}(x) = \frac{x^4}{4} - (1- \theta) \frac{x^2}{2}$.} 
Furthermore, we note that the spectral gap of the linearized McKean-Vlasov operator~\eqref{e:FP_lin}, when restricted to the subspace of $L^2(\R ; \rho_{\infty}^{\beta_c, \theta_c, m=0})$ that is orthogonal to $\{\rho_{\infty}^{\beta_c, \theta_c, m=0}, x \rho_{\infty}^{\beta_c, \theta_c, m=0} \}$ is larger than the positive first nonzero eigenvalue of the operator considered on $L^2(\R ; \rho_{\infty}^{\beta_c, \theta_c, m=0})$. This follows immediately from the variational characterization of the spectral gap--see~\cite[Lemma 3.4.1]{Dawson83}. 
\end{remark}

\smallskip

\noindent {\it Proof of Proposition~\ref{prop:null_space_multi}} The proof of Part (a) follows from the structural assumption and Remark~\ref{rem:spectral_gap}. The proof of Part (b) follows from the structural assumption and Lemma~\ref{lem:null_space}.
\qed

\smallskip

We remark that, for the case of the multi-species Desai--Zwanzig model, and, in particular, under the structural assumption, the coupled McKean-Vlasov PDE is equivalent to the infinite-dimensional ODE system of the moments. Since the calculations are very similar to the ones presented in~\cite{Dawson83}, we do not present the details.

%
%
\section{Convergence \RV{of the solutions and the energy functionals}}
\label{sec:converg}
In this section, we study {\color{black}the limiting set in long-time for the solutions of the coupled mean-field PDE system. If moreover the set of invariant probability measures (that is a set of $M$-tuples) is discrete, then, we obtain the} convergence to a stationary state. \RV{We also study the convergence of the associated energy functionals.} 

Throughout the section, we assume that the initial conditions for the coupled PDEs are absolutely continuous with respect to the Lebesgue measure and with finite entropy, which implies that the free-energy is finite. We deduce that the probability measures $\mu_t^\ell$ are all absolutely continuous with respect to the Lebesgue measure, for any $t\geq0$. Consequently, the weak convergence of the measures $\mu_t^\ell$ is equivalent to convergence of the corresponding densities with respect to Lebesgue measure in {\color{black}the following function space:

\[
\left\{f:\bRb^d\longrightarrow\bRb\,\vert\,\mathcal{N}(f)<+\infty\right\}\,,
\]
equipped with the norm $\mathcal{N}$ defined as $\mathcal{N}(f):=\int_{\bRb^d}\left(1+|x|^{8p^2}\right)\left|f(x)\right|\ddx$. Let us stress that in fact, we consider the subset of functions $f$ such that $f(x)\geq0$ and $\int_{\bRb^d}f(x)\ddx=1$.} Since the moments of $\mu_t^\ell$ are uniformly upper-bounded, see~\cite{DuongTugaut2020}, we immediately deduce that the families $\left\{\mu_t^\ell\,;\,t\geq0\right\}$ are tight for any $1\leq\ell\leq M$.

\begin{defn}
\label{sangoku}
By $\mathcal{A}_\sigma$ (resp. $\mathcal{S}_\sigma$), we denote the set of the limiting values of the family $\left\{\mu_t^\ell\,;\,t\geq0\right\}$ (resp. the set of the invariant probabilities of the mean field SDE~\eqref{eq: mean-field SDE}){\color{black}; for the aforementioned topology.}
\end{defn}

\begin{defn}
\label{sangohan}
We say that a set $\mathcal{D}$ of measures in $\bRb^M$ is discrete if for any $\nu\in\mathcal{D}$, there exists a neighbourhood $\mathcal{V}$ of $\nu$ for the topology of weak convergence such that $\mathcal{D}\bigcap\mathcal{V}=\left\{\nu\right\}$. Similarly, we say that $\mathcal{D}$ is path-connected if it is path-connected for the topology {\color{black}mentioned at the beginning of the current section.}
\end{defn}

{\color{black}We stress that from now on, weak convergence will always refer to the topology associated to the norm $\mathcal{N}$, mentioned at the beginning of the section.}

We now give the three main results of this section. The first one concerns the topology of the set $\mathcal{A}_\sigma$.

\begin{thm}
\label{chichi}
Let Assumptions~\ref{asp: assumption}, \ref{asp: assumptionbis} and \ref{asp: assumptionter} hold. Then, the set $\mathcal{A}_\sigma$ contains either a single element $\mu^\sigma\in\mathcal{S}_\sigma$ or a path-connected subset of $\mathcal{S}_\sigma$. Moreover, for all $\mu\in\mathcal{A}_\sigma$, we have $\displaystyle\Upsilon_\sigma(\mu)=L_\sigma:=\lim_{t\to+\infty}\Upsilon_\sigma(\mu_t)$.
\end{thm}

From this theorem, we deduce the following statements.

\begin{cor}
\label{piccolo}
Let assumptions of Theorem~\ref{chichi} hold and assume, in addition, that the set $\mathcal{S}_\sigma\bigcap\Upsilon_\sigma^{-1}\left(\left\{\lambda\right\}\right)$ is discrete for any $\lambda\in\bRb$. Then  the probability measure $\mu_t$ weakly converges to an invariant probability $\mu^\sigma\in\mathcal{S}_\sigma$, as $t$ goes to infinity.
\end{cor}

\begin{cor}
\label{piccolo2}
Let assumptions of Theorem~\ref{chichi} hold, together with the synchronization assumption~\ref{asp: assumption:synchronization}. 
Then, the following limit holds for any $1\leq k\leq M$:
\begin{equation*}
\limsup_{\sigma\to0}\inf_{a\in\mathcal{G}_k}\limsup_{t\to+\infty}\int_{\bRb^d}\gaga x-a\drdr^2\mu_t^k(x)\ddx=0\,.
\end{equation*}
where $\mathcal{G}_k:=\left\{x\in\bRb^d\,\,\left|\right.\,\,\nabla V_k(x)=0\right\}$ denotes the set of critical points of $V_k$.

\end{cor}

\begin{defn}
\label{fightclub}
We define the function space $\mathcal{M}$ as the set of positive $L^1$ functions $f$ with bounded moments of order $8 p^2$, where $p$ is defined in Assumption~\ref{asp: assumption}:\\[2pt]
{\bf 1.} For all $x\in\bRb^M$, $f(x)>0$.\\
{\bf 2.} We have $\int_{\bRb^M} f = 1$. \\
{\bf 3.} There exists an $M_0 >0$ such that  $\displaystyle\max_{1\leq\ell\leq8p^2}\int_{\bRb^d}\left| x\right|^\ell f(x)\ddx\leq M_0$.
\end{defn}

Set 
\[
\xi(t):=\Upsilon_\sigma\left(\mu_t^1,\ldots,\mu_t^M\right)\, ,
\]
where $\Upsilon_\sigma$ denotes the free energy~\eqref{eq: free energy}.
Due to the lower-bound of the free-energy functional $\Upsilon_\sigma$ and the fact that $\xi$ is non-increasing (see Section \ref{subsec:free_energy}), we obtain the next lemma.
\begin{lem}
\label{lem:fr:conv}
Under Assumption~\ref{asp: assumption}, Assumption~\ref{asp: assumptionbis} and Assumption~\ref{asp: assumptionter}, there exists a constant $L_\sigma\in\bRb$ such that $\Upsilon_\sigma(\mu_t^1,\ldots,\mu_t^M)$ converges to $L_\sigma$ in the limit as $t \rightarrow +\infty$.
\end{lem}
The second main result of this section shows the convergence of {\color{black} some sequence of the} solutions of the mean-field PDE \eqref{eq: multi-species PDE} to a stationary state.
\begin{prop}
\label{thm:fr:subconv}
There exists a stationary state $\left(\mu_\infty^1,\ldots,\mu_\infty^M\right)$ and a sequence $\{ t_k \}_{k\in\mathbb{N}}$ that converges to infinity such that, for any $1\leq\ell\leq M$, $\mu_{t_k}^\ell$ converges weakly to $\mu_\infty^\ell$, in the limit as $t_k \rightarrow +\infty$.
\end{prop}

\begin{proof}

{\bf Plan:} we adapt the proof of \cite[Proposition 2.1]{JOTP}. We will use dots to denote differentiation with respect to time. First, we use the convergence of $\int_t^\infty \dot{\xi}(s) \, ds$ towards $0$ when $t$ tends to infinity  and we deduce the existence of a sequence $(t_k)_k$ such that $\dot{\xi} \left(t_k\right)$ tends to $0$ when $k$ goes to infinity. Then, the tightness of the sequences $\left\{\mu_{t_k}^\ell\,;\,k\in\mathbb{N}\right\}$ allows us to extract a subsequence of $\{ t_k \}_k$ - we will continue to write it $\{ t_k \}_k$ - such that for any $1\leq\ell\leq M$, $\mu_{t_k}^\ell$ weakly converges towards a limiting value of the family $\left\{\mu_{t}^\ell\,;\,t\in\mathbb{R}_+\right\}$. By using an appropriate test function and Weyl's lemma, we prove that this adherence value satisfies the system of equations~\eqref{eq: muk}. We deduce that $(\mu_\infty^1,\ldots,\mu_\infty^M)$ is a stationary state. Let us give now the details of the proof.

\noindent{}{\bf Step 1:} Lemma~\ref{lem:fr:conv} implies that the quantity $\zeta(t):=\int_t^\infty \dot{\xi}(s)ds$ converges to $0$ as $t \rightarrow +\infty$. 
Since the free-energy functional is lower-bounded on $\mathcal{M}$, $\xi(t)$ is monotonous so we deduce the existence of an increasing sequence $\{ t_k\})_{k\in\mathbb{N}}$ with $t_k \rightarrow +\infty$ such that $\dot{\xi}(t_k) \longrightarrow 0$.

\noindent{}{\bf Step 2:} The uniform boundedness with respect to time of the second moment allows us to use Prohorov's theorem: we can extract a subsequence, still denoted by $\{t_k \}_{k}$, such that $\mu_{t_k}^\ell$ weakly converges to a probability measure $\mu_\infty^\ell$, for all $1\leq\ell\leq M$.

\noindent{}{\bf Step 3:} We now consider a test function $\varphi^\ell\in\mathcal{C}^\infty\left(\mathbb{R}^d,\mathbb{R}^d\right)\bigcap L^2\left(\mathbb{R}^d, \, \mu_\infty^\ell\right)$ with compact support and we estimate the following quantity:
\[
\left|\int_{\mathbb{R}^d}\left\langle\varphi^\ell(x)\, , \,\eta_{t_k}^\ell(x)\right\rangle\ddx\right|\,,
\]
where $\langle \cdot , \cdot \rangle$ denotes the Euclidean inner product and
\begin{equation}
\label{lanvin}
\eta_t^\ell(x):=\frac{\sigma_\ell^2}{2}\nabla\mu_t^\ell(x)+\left(\nabla V_\ell(x)+\sum_{j=1}^Ma_j\nabla F_{\ell j}\ast\mu_t^j(x)\right)\mu_t^\ell(x)\,.
\end{equation}
We calculate the time derivative of the free energy, as was done in Section~\ref{subsec:free_energy}, to derive the bound
\[
\dot{\xi}(t)\leq-\sum_{\ell=1}^M\int_{\bRb^d} a_\ell\frac{\left|\eta_t^\ell(x)\right|^2}{\mu_t^\ell(x)}\ddx\,.
\]
We use the Cauchy-Schwarz inequality to deduce that
\begin{align*}
\left|\int_{\mathbb{R}^d}\left\langle\varphi^\ell(x)\, , \,\eta_{t_k}^\ell(x)\right\rangle \ddx\right|&=\left|\int_{\mathbb{R}^d}\left\langle\varphi^\ell(x)\sqrt{\mu_{t_k}^\ell(x)}\, , \,\frac{\eta_{t_k}^\ell(x)}{\sqrt{\mu_{t_k}^\ell(x)}}\right\rangle\ddx\right|\\
&\leq\sqrt{\int_{\mathbb{R}^d}\left|\varphi^\ell(x)\right|^2\mu_{t_k}^\ell(x)\ddx}\times\sqrt{\int_{\mathbb{R}^d}\frac{1}{\mu_{t_k}^\ell(x)}\left|\eta_{t_k}^\ell(x)\right|^2\ddx}
\end{align*}
We combine the above two estimates to obtain:

\[
\left|\int_{\mathbb{R}^d}\left\langle\varphi^\ell(x)\,, \,\eta_{t_k}^\ell(x)\right\rangle \ddx\right|\leq\sqrt{-\frac{\dot{\xi}(t_k)}{a_\ell}}\sqrt{\int_{\mathbb{R}^d}\left|\varphi^\ell(x)\right|^2\mu_{t_k}^\ell(x)\ddx}\longrightarrow0
\]
as $k \rightarrow +\infty$. We use our assumption that $\varphi$ has compact support, together with an integration by parts, to deduce that

\begin{align*}
&\int_{\mathbb{R}^d}\left\langle\varphi^\ell(x)\, , \,\frac{\sigma_\ell^2}{2}\nabla\mu_{t_k}^\ell(x)+\mu_{t_k}^\ell(x)\left[\nabla V_\ell(x)+\sum_{j=1}^Ma_j\nabla F_{\ell j}\ast \mu_{t_k}^j(x)\right]\right\rangle\ddx\\
=&\int_{\mathbb{R}^d}\left\langle\varphi^\ell(x)\, , \,\nabla V_\ell(x)+\sum_{j=1}^Ma_j\nabla F_{\ell j}\ast \mu_{t_k}^j(x)\right\rangle \mu_{t_k}^\ell(x)\ddx-\int_{\mathbb{R}^d}\frac{\sigma_\ell^2}{2}{\rm div }\left(\varphi^\ell(x)\right)\mu_{t_k}^\ell(x)\ddx.
\end{align*}
The weak convergence of $\mu_{t_k}^j$ towards $\mu_\infty^j$ for all $1\leq j\leq M$ implies that this term converges towards 

\[
\int_{\mathbb{R}^d}\left\langle\varphi^\ell(x)\, , \,\nabla V_\ell(x)+\sum_{j=1}^Ma_j\nabla F_{\ell j}\ast \mu_\infty^j(x)\right\rangle \mu_\infty^\ell(\ddx)-\int_{\mathbb{R}^d}\frac{\sigma_\ell^2}{2}{\rm div }(\varphi^\ell(x)) \mu_\infty^\ell(\ddx)
\]

On the other hand, it was proved previously that $\int_{\mathbb{R}^d}\left\langle\varphi^\ell(x)\, , \,\eta_{t_k}^\ell(x)\right\rangle\ddx$ converges to $0$ as $k  \rightarrow  \infty$. We deduce that, for any test function $\varphi^\ell\in\mathcal{C}^\infty\left(\mathbb{R}^d,\mathbb{R}^d\right)\bigcap L^2\left(\mu_\infty^\ell\right)$ with compact support, we have : 
\begin{eqnarray}
\label{eq:fr:weaksol}
\int_{\mathbb{R}^d}\left\langle\varphi^\ell(x)\, , \,\nabla V_\ell(x)+\sum_{j=1}^Ma_j\nabla F_{\ell j}\ast \mu_\infty^j(x)\right\rangle \mu_\infty^\ell(\ddx)-\int_{\mathbb{R}^d}\frac{\sigma_\ell^2}{2}{\rm div }\left(\varphi^\ell(x)\right)\mu_\infty^\ell(\ddx)=0\,.
\end{eqnarray}

\noindent{}{\bf Step 4:} {\color{black}We remark that the nonlinearity does not imply any difficulty here. Indeed, $\mu_\infty^1,\cdots,\mu_\infty^M$ are given measures satisfying Equation~\eqref{eq:fr:weaksol} so that by putting $W^\ell:=V_\ell+\sum_{j=1}^Ma_jF_{\ell j}\ast\mu_\infty^j(x)$, Equation~\eqref{eq:fr:weaksol} becomes

\[
\int_{\mathbb{R}^d}\left\langle\varphi^\ell(x)\, , \,\nabla W^\ell(x)\right\rangle\mu_\infty^\ell(\ddx)-\int_{\mathbb{R}^d}\frac{\sigma_\ell^2}{2}{\rm div }\left(\varphi^\ell(x)\right)\mu_\infty^\ell(\ddx)=0\,.
\]

This means that $\mu_\infty^\ell$ is a weak solution of the equation
\begin{eqnarray*}
\frac{\sigma_\ell^2}{2}\nabla\rho+\nabla W^\ell\rho=0\,,
\end{eqnarray*}

for any $1\leq\ell\leq M$. By applying Weyl's lemma, we deduce that the measure 
\begin{eqnarray*}
\exp\left[\frac{2}{\sigma_\ell^2}W^\ell(x)\right]\mu_\infty^\ell(\ddx)
\end{eqnarray*}
admits a $\mathcal{C}^\infty$ density with respect to the Lebesgue measure, still denoted by $\mu_\infty^\ell$. Furthermore, its gradient is zero and so the density $\mu_\infty^\ell$ satisfies 

\[
\mu^\ell(x)=\frac{1}{Z_\ell}\exp\Big[-\frac{2}{\sigma_\ell^2}W^\ell(x)\Big]\,,
\]

with $Z_\ell:=\int_\bRb^d\exp\Big[-\frac{2}{\sigma_\ell^2}W^\ell(y)\Big]\ddy$. This exactly corresponds to Equation~\eqref{eq: muk}.}

Consequently, the collection of measures $\left(\mu_\infty^1,\ldots,\mu_\infty^M\right)$ is a stationary state, according to Proposition~\ref{jugnot}.

\end{proof}

\RV{In the remainder of this section, namely in Lemma~\ref{johnkramer}, Proposition~\ref{jigsaw} and Proposition~\ref{edeka}, we no longer restrict ourselves to quadratic interaction potentials: we work there under Assumption~\ref{asp: interaction3} instead. All the other results of this section are stated under Assumption~\ref{asp: assumptionbis}, as recorded in Table~\ref{tab: conditions}.}

\begin{lem}
\label{johnkramer}
Let Assumptions~\ref{asp: assumption} \RV{and~\ref{asp: interaction3}} hold. Then, for any stationary state $\left(\mu^1,\cdots,\mu^M\right)$ and for any $1\leq\ell\leq M$, the measure $\mu^\ell$ is uniquely determined by its moments. In particular, if a measure $\nu$ has the same moments as $\mu^\ell$, we deduce $\nu=\mu^\ell$.
\end{lem}
\begin{proof}
Let $(\mu^1,\cdots,\mu^M)$ be a stationary state and let $\ell$ be such that $1\leq\ell\leq M$. By Proposition~\ref{jugnot}, it is absolutely continuous with respect to Lebesgue measure and its density $\mu^\ell$ satisfies the equation
\begin{equation*}
\mu^\ell(x)=\frac{\exp\left[-\frac{2}{\sigma_\ell^2}\left(V_\ell(x)+\sum_{j=1}^Ma_jF_{\ell j}\ast \mu^j(x)\right)\right]}{\int_{\bRb^d}\exp\left[-\frac{2}{\sigma_\ell^2}\left(V_\ell(y)+\sum_{j=1}^Ma_jF_{\ell j}\ast \mu^j(y)\right)\right]\ddy}\,.
\end{equation*}
Consequently, for all $r>0$, we have
\begin{equation*}
\int_{\bRb^d}{\rm e}^{\frac{2}{\sigma_\ell^2}r|x|}\mu^\ell(x)\ddx=\frac{\int_{\bRb^d}\exp\left[-\frac{2}{\sigma_\ell^2}\left(V_\ell(x)+\sum_{j=1}^Ma_jF_{\ell j}\ast \mu^j(x)-r|x|\right)\right]\ddx}{\int_{\bRb^d}\exp\left[-\frac{2}{\sigma_\ell^2}\left(V_\ell(x)+\sum_{j=1}^Ma_jF_{\ell j}\ast \mu^j(x)\right)\right]\ddx}\,.
\end{equation*}
Since $F$ is convex and since $\nabla^2V_\ell(x)>0$ if $|x|$ is sufficiently large, we deduce that 

\[
\int_{\bRb^d}\eee^{\frac{2}{\sigma_\ell^2}r|x|}\mu^\ell(x)\ddx<+\infty\,.
\]

Consequently, the series $\sum_{k=0}^\infty\frac{\nu_k}{k!}r^k$, with $\nu_k:=\int_{\bRb}|x|^k\mu^\ell(x)\ddx$, has a positive convergence radius. After applying Fourier's criteria, the statement follows.
\end{proof}

\begin{defn}
\label{tototo}
For any $d$-index of integers $q$ with $q_1+\cdots+q_d\geq1$, by $\Lambda_q$, we denote the mapping from $\mathcal{M}^M$ to $\bRb^{dM}$ defined by
\begin{equation*}
\Lambda_q\left(\mu^1,\ldots,\mu^M\right):=\left(\int_\bRb|x|^q\mu^1(x)\ddx,\ldots,\int_\bRb|x|^q\mu^M(x)\ddx\right)\,.
\end{equation*}
\emph{Namely, the linear function $\Lambda_q$ corresponds to the vector of moments of order $q$.}
\end{defn}

In \cite[Proposition 2.5]{AOP}, it has been shown that the free energy of any limiting value of the stationary state is less than the limit of the free energy in the trajectories. We now prove that equality holds under some additional hypotheses. The following proposition is the third main result of this section.

\begin{prop}
\label{jigsaw}
\RV{Let Assumptions \ref{asp: assumption}, \ref{asp: assumptionter} and \ref{asp: interaction3} hold}. Let $\{ t_k\}_{k\in\mathbb{N}}$ be an increasing sequence such that $t_k \rightarrow +\infty$, $\dot{\xi}(t_k) \rightarrow 0$ and $\left(\mu_{t_k}^1,\ldots,\mu_{t_k}^M\right)$ converges to $\left(\mu_\infty^1,\ldots,\mu_\infty^M\right)\in\mathcal{A}_\sigma$ as $k \rightarrow +\infty$ . Then, we have :
\[
\Upsilon_\sigma(\mu_\infty^1,\ldots,\mu_\infty^M)=L_\sigma:=\lim_{t\to+\infty}\Upsilon_\sigma(\mu_t^1,\ldots,\mu_t^M)\,.
\]
\end{prop}

The proof is similar to the one of \cite[Proposition 1.8]{AOP}.

\begin{proof}
From now on, by $\mu_{t_k}^\ell$ (resp. $\mu_\infty^\ell$), we denote the density with respect to the Lebesgue measure of the probability measure $\mu_{t_k}^\ell$ (resp. $\mu_\infty^\sigma$). The convergence from the term

\[
\int_{\bRb^d}V(x)\mu_{t_k}^\ell(x)\ddx+\frac{1}{2}\sum_{j=1}^Ma_j\int_{\bRb^d}\left(F_{\ell j}\ast \mu_{t_k}^j(x)\right)\mu_{t_k}^\ell(x)\ddx
\]

towards 

\[
\int_{\bRb^d} V_\ell(x)\mu_\infty^\ell(x)\ddx+\frac{1}{2}\sum_{j=1}^Ma_j\int_{\bRb^d}\left(F_{\ell j}\ast \mu_\infty^j(x)\right)\mu_\infty^\ell(x)\ddx
\]

is implied by the convergence hypothesis of $\mu_{t_k}^j$ to $\mu_\infty^j$ for any $1\leq j\leq M$. Hence, we focus on the entropy term.

\noindent{}{\bf Step 1.} \emph{Here, we aim to prove the existence of a constant $C>0$ such that $\mu_{t_k}^\ell(x)\leq C$ for all $k\in\mathbb{N}^*$, for all $x\in\bRb^d$ and for any $1\leq\ell\leq M$.}

To show this, we bound the integral of $\left|\nabla \mu_{t_k}^\ell\right|$ on $\mathbb{R}^d$. Using the triangle inequality we obtain
\begin{align*}
\int_{\bRb^d}\left|\nabla \mu_{t_k}^\ell(x)\right|\ddx\leq&\frac{2}{\sigma_\ell^2}\int_{\bRb^d}\left|\eta_{t_k}^\ell(x)\right|\mu_{t_k}^\ell(x)\ddx\\
&+\frac{2}{\sigma_\ell^2}\int_{\bRb^d}\left|\nabla V_\ell(x)+\sum_{j=1}^Ma_j\nabla F_{\ell j}\ast \mu_{t_k}^j(x)\right|\mu_{t_k}^\ell(x)\ddx\,,
\end{align*}
where $t\mapsto\eta_t$ has been introduced in~\eqref{lanvin}. The growth property on $\nabla V$ and the particular form of the interaction potentials yield
\begin{align*}
\int_{\bRb^d}\left|\nabla V_\ell(x)+\sum_{j=1}^Ma_j\nabla F_{\ell j}\ast \mu_{t_k}^j(x)\right| \mu_{t_k}^\ell(x)\ddx\leq C_1(\ell)\int_{\bRb^d}\left(1+|x|^{2p}\right)\mu_{t_k}^\ell(x)\ddx\leq C_2(\ell)\,,
\end{align*}
where $C_2(\ell)>0$ is a constant. By using Cauchy-Schwarz's inequality and the entropy dissipation, we obtain
\begin{equation*}
\int_{\bRb^d}\left|\eta_{t_k}^\ell(x)\right|\mu_{t_k}^\ell(x)\ddx\leq\sqrt{\int_{\bRb^d}\left|\eta_{t_k}^\ell(x)\right|^2\mu_{t_k}^\ell(x)\ddx}\leq\sqrt{-\frac{1}{a_\ell} \dot{\xi}(t_k)}\,.
\end{equation*}
The quantity $\sqrt{-\dot{\xi}(t_k)}$ tends to $0$ as $k \rightarrow +\infty$; therefore, it is bounded. We conclude that there  exists  a positive constant $C_3$ independent from $k\in\mathbb{N}^*$ and from $1\leq\ell\leq M$ such that
\begin{equation*}
\int_{\bRb^d}\left|\nabla \mu_{t_k}^\ell(x)\right|\ddx\leq C_3\,.
\end{equation*}
Consequently, $\mu_{t_k}^\ell(x)\leq \mu_{t_k}^\ell(0)+C$ for all $x\in\mathbb{R}$, for all $k\in\mathbb{N}^*$ and for any $1\leq\ell\leq M$. However, the sequences $\left\{\mu_{t_k}^\ell(0)\right\}_{k\in\mathbb{N}}$ converge so that they are bounded. Hence, there exists a constant $C_4$ such that $\mu_{t_k}^\ell(x)\leq C_4$ for all $k\in\mathbb{N}$, for any $1\leq\ell\leq M$ and $x\in\mathbb{R}^d$.

\noindent{}{\bf Step 2.} The application $x\mapsto \mu_{t_k}^\ell(x)\log\left(\mu_{t_k}^\ell(x)\right)$ is uniformly lower-bounded with respect to $k$ and $\ell$. Then, we can apply Lebesgue's dominated convergence theorem which can be used to show the convergence as, $k \rightarrow +\infty$  of the integral term
\begin{equation*}
\int_{\bRb^d}\mu_{t_k}^\ell(x)\log\left(\mu_{t_k}^\ell(x)\right)\mathds{1}_{\{|x|\leq R\}}\ddx
\end{equation*}

to $\int_{\bRb^d}\mu_\infty^\ell(x)\log\left(\mu_\infty^\ell(x)\right)\mathds{1}_{\{|x|\leq R\}}\ddx$ for any $R\geq0$.

\noindent{}{\bf Step 3.} The remainder of the entropy integral, i.e. for $|x| > R$,  is split into three terms.

\noindent{}{\bf Step 3.1.} The first term is
\begin{equation*}
I_1(k):=\int_{\bRb^d}\mu_{t_k}(x)\log\left(\mu_{t_k}(x)\right)\mathds{1}_{\{|x|>R\,\,;\,\,\mu_{t_k}^\ell(x)\geq1\}}\ddx\,.
\end{equation*}
Due to the upper-bound from Step 1, the uniform boundedness of the moments and Markov's inequality, we obtain
\begin{align*}
I_1(k)\leq\log(C_4)\mu_{t_k}^\ell\left(\mathbb{B}\left(0\,;\,R\right)^c\right)\leq\frac{\log(C_4)\,M_0}{R^2}\,.
\end{align*}
{\bf Step 3.2.} The second term is defined as
\begin{equation*}
I_2(k):=\int_{\bRb^d}\mu_{t_k}^\ell(x)\log\left(\mu_{t_k}^\ell(x)\right)\mathds{1}_{\{|x|>R\,\,;\,\,\mu_{t_k}^\ell(x)<1\,\,;\,\,\mu_{t_k}^\ell(x)\geq \eee^{-|x|}\}}\ddx<0\,.
\end{equation*}
We bound it in the following way:
\begin{align*}
\left|I_2(k)\right|=-I_2(k)\leq\int_{\mathbb{B}\left(0\,;\,R\right)^c}|x|\mu_{t_k}^\ell(x)\ddx\leq\frac{M_0}{R^2}\,.
\end{align*}

\noindent{}{\bf Step 3.3.} We now consider the third term,
\begin{equation*}
I_3(k):=\int_{\bRb^d}\mu_{t_k}^\ell(x)\log\left(\mu_{t_k}^\ell(x)\right)\mathds{1}_{\{|x|>R\,\,;\,\,\mu_{t_k}^\ell(x)<1\,\,;\,\,\mu_{t_k}^\ell(x)<\eee^{-|x|}\}}\ddx<0\,.
\end{equation*}
We introduce the function $\gamma(x):=\sqrt{x}\log(x)\mathds{1}_{\{x<1\}}$. We have $-C'\leq\gamma(x)\leq0$ for all $x\in\bRb$, where $C'$ is a positive constant. This provides
\begin{align*}
\left|I_3(k)\right|=-I_3(k)\leq-\int_{\mathbb{B}\left(0\,;\,R\right)^c}\gamma\left(\mu_{t_k}(x)\right)\eee^{-\frac{1}{2}|x|}\ddx\leq\Theta(R)\,,
\end{align*}
$\Theta$ being a decreasing function from $\bRb_+$ to itself such that $\displaystyle\lim_{R\to+\infty}\Theta(R)=0$.

\noindent{}{\bf Step 4.} Let $\epsilon$ be a positive real, arbitrarily small. By taking $R$ sufficiently large, we have the upper-bound
\begin{equation*}
{\color{black}\max}_{1\leq\ell\leq M}\sup_{k\in\mathbb{N}^*}\max\left\{I_1(k)\,;\,I_2(k)\,;\,I_3(k)\right\}<\frac{\epsilon}{9}\,.
\end{equation*}
Moreover, for $R$ large enough, we have the inequality
\begin{equation*}
\int_{\bRb^d}\mu_\infty^\ell(x)\log\left(\mu_\infty^\ell(x)\right)\mathds{1}_{\{|x|\geq R\}}\ddx<\frac{\epsilon}{3}\,.
\end{equation*}
The two previous inequalities do not depend on $k$. Then, by taking $k$ sufficiently large, the following upper-bound holds:
\begin{equation*}
\left|\int_{\bRb^d}\mu_{t_k}^\ell(x)\log\left(\mu_{t_k}^\ell(x)\right)\mathds{1}_{\{|x|\leq R\}}\ddx-\int_{\bRb^d}\mu_\infty^\ell(x)\log\left(\mu_\infty^\ell(x)\right)\mathds{1}_{\{|x|\leq R\}}\ddx\right|<\frac{\epsilon}{3}\,.
\end{equation*}
This implies the inequality $\left|\Upsilon_\sigma\left(\mu_{t_k}^1,\ldots,\mu_{t_k}^M\right)-\Upsilon_\sigma(\mu_\infty^1,\ldots,\mu_\infty^M)\right|<\epsilon$. Namely, the free-energy $\Upsilon_\sigma\left(\mu_{t_k}^1,\ldots,\mu_{t_k}^M\right)$ converges to $\Upsilon_\sigma(\mu_\infty^1,\ldots,\mu_\infty^M)$. Due to the monotonicity of the free-energy along the trajectories, $\Upsilon_\sigma\left(\mu_t^1,\ldots,\mu_t^M\right)$ converges to $L_\sigma$ which implies $\Upsilon_\sigma(\mu_\infty^1,\ldots,\mu_\infty^M)=L_\sigma$.
\end{proof}

\begin{prop}
\label{edeka}
Let Assumption~\ref{asp: assumption} 
 \RV{and~\ref{asp: interaction3}} hold. Let $(\nu^1,\ldots,\nu^M)$ be $M$ measures with densities $(\nu^1,\ldots,\nu^M)\in\mathcal{M}^M$. We also assume $(\nu^1,\ldots,\nu^M)\notin\mathcal{S}_\sigma$. Thus, there exist $q\in\mathbb{N}^*$ and $\rho>0$ such that
\begin{equation*}
\inf_{(\mu^1,\ldots,\mu^M)\in\mathcal{S}_\sigma}\max_{1\leq\ell\leq M}\left|\int_{\bRb^d}x^q\mu^\ell(x)\ddx-\int_{\bRb^d}x^q\nu^\ell(x)\ddx\right|\geq\rho\,.
\end{equation*}
\end{prop}

\begin{proof}
By arguing as in \cite{AOP}, we can easily show that all the moments of $\mu_t^\ell$ are bounded for any $1\leq\ell\leq M$ and for any $t\geq1$. Moreover, for any $N_0\in\mathbb{N}^*$, we have
\[
\sup_{t\geq1}\int_{\bRb^d}\left|x\right|^{2N_0}\mu_t^\ell(x)\ddx<+\infty\,.
\]
Let us proceed a \emph{reductio ad absurdum} by assuming that there exists $n\in\mathbb{N}^*$ such that for all $k\geq n$, there exists $\mu_k^\ell\in\mathcal{S}_\sigma$ satisfying
\begin{equation*}
\left|\int_{\bRb^d}x^q\mu_k^\ell(x)\ddx-\int_{\bRb^d}x^q\nu^\ell(x)\ddx\right|\leq\frac{1}{k}\,.
\end{equation*}
We deduce that the sequence $\left(\int_{\bRb^d}x^q\mu_k^\ell(x)\ddx\right)_{k\in\mathbb{N}^*}$ is bounded. Consequently, we can extract a subsequence $\left\{\phi(k)\right\}_{k\in\mathbb{N}^*}$ such that the sequence $\left\{\int_{\bRb^d}x^q\mu_{\phi(k)}^\ell(x)\ddx\right\}_{k\geq n}$ converges to $m_\infty^\ell\in\bRb$. Since the family $\left\{\mu_t^\ell\,;\,t\geq0\right\}$ is tight, Prokhorov's theorem allows us to assume, without any loss of generality, that the measure $\mu_{\phi(k)}^\ell$ converges weakly to a measure $\mu_\infty^\ell$ as $k \rightarrow +\infty$  (along a subsequence of $\mu_{\phi(k)}^\ell$).

Furthermore, $\Lambda_q\left(\mathcal{S}_\sigma\right)$ (see Definition~\ref{tototo} for the introduction of the function $\Lambda_q$) satisfies an algebraic equation in $\bRb^M$ and so it is closed. Consequently, $\left(\mu_\infty^1,\ldots,\mu_\infty^M\right)$ is a stationary state.

In addition, we have that
\begin{equation*}
\left|\int_{\bRb^d}x^q\mu_\infty^\ell(x)\ddx-\int_{\bRb^d}x^q\nu^\ell(x)\ddx\right|=0
\end{equation*}
for any $1\leq\ell\leq M$ and $p\geq 1$. In other words, the stationary state $(\mu_\infty^1,\ldots,\mu_\infty^M)$ has the same moments as $\nu$. Since the stationary state is uniquely determined by its moments, we deduce that $\nu^\ell=\mu_\infty^\ell$ so $(\nu^1,\ldots,\nu^M)$ is a stationary state. We have arrived at a contradiction.
\end{proof}

\subsection{Proof of Theorem~\ref{chichi}}
In this subsection, we present the proof of the main theorem, Theorem \ref{chichi}, of the section.

\subsubsection{Outline of the proof}
First, we provide the main ideas of the proof. The details will be presented in the following subsections. By Proposition \ref{thm:fr:subconv}, we know that there exists a stationary state $\left(\mu_\infty^1,\ldots,\mu_\infty^M\right)$ such that
\begin{itemize}
 \item the $M$-uple of measures $\left(\mu_\infty^1,\ldots,\mu_\infty^M\right)$ is in $\mathcal{A}_\sigma$.
 \item it is a stationary state.
 \item The free energy of $\left(\mu_\infty^1,\ldots,\mu_\infty^M\right)$ is equal to $L_\sigma:=\lim_{t\to0}\Upsilon_\sigma(\mu_t^1,\ldots,\mu_t^M)$.
\end{itemize}
If $\mathcal{A}_\sigma=\left\{\left(\mu_\infty^1,\ldots,\mu_\infty^M\right)\right\}$, the proof is achieved. From now on, we assume that $\#\mathcal{A}_\sigma>1$. The proof now consists in establishing that
\begin{enumerate}
 \item The set $\mathcal{A}_\sigma$ is included in $\mathcal{S}_\sigma$.
 \item The set $\mathcal{A}_\sigma$ is connected to the path.
 \item The free energy is constant on $\mathcal{A}_\sigma$.
\end{enumerate}

\noindent{}{\bf Step 1.} \emph{We proceed a \emph{reductio ad absurdum} in order to prove the first statement. The details are given in Subsection~\ref{tomhanks}}

We assume the existence of $(\nu^1,\ldots,\nu^M)\in\mathcal{A}_\sigma$ such that $(\nu^1,\ldots,\nu^M)\notin\mathcal{S}_\sigma$. According to Proposition \ref{edeka}, there exists a closed set with nonempty interior $\mathcal{H}$ that contains $(\nu^1,\ldots,\nu^M)$ and which has an empty intersection with $\mathcal{S}_\sigma$.

Since $(\nu^1,\ldots,\nu^M)$ is an adherence value, we can prove that for some $1\leq\ell\leq M$, there exists a smooth function with compact support $\varphi^\ell$, a constant $\rho>0$ and two increasing sequences $\left(r_k\right)_{k\in\mathbb{N}}$ and $\left(s_k\right)_{k\in\mathbb{N}}$ such that $r_k<s_k$ and for all $t\in[r_k;s_k]$, we have
\begin{equation*}
\int_\bRb\varphi^\ell(x)\nu^\ell(\ddx)=0<\rho=\int_{\bRb}\varphi^\ell(x)\mu_{r_k}^\ell(\ddx)\leq\int_{\bRb}\varphi^\ell(x)\mu_t^\ell(\ddx)\leq\int_{\bRb}\varphi^\ell(x)\mu_{s_k}^\ell(\ddx)=2\rho\,.
\end{equation*}
Moreover, for all $t\in[r_k;s_k]$, $\mu_t^\ell\in\mathcal{H}$. Applying Proposition \ref{prop:a:othermeasure}, we can construct a stationary state $(\widetilde{\nu}^1,\ldots,\widetilde{\nu}^M)\in\mathcal{S}_\sigma\bigcap\mathcal{H}$. This is a contradiction.

\noindent{}{\bf Step 2.} \emph{The details of the proof of the second statement are in Subsection~\ref{tomhanks2}.}

We use the previous result: all the limiting values are stationary states. Since the interacting potentials are quadratic, the function $\Lambda_1$ (defined as in Definition~\ref{tototo} with $q_i=1$ for any $1\leq i\leq d$) is a bijection from $\mathcal{S}_\sigma$ to $\Lambda_1\left(\mathcal{S}_\sigma\right)\subset\bRb^M$, hence we deduce that $\Lambda_1$ is a bijection from $\mathcal{A}_\sigma$ to $\mathcal{C}_\sigma:=\Lambda_1\left(\mathcal{A}_\sigma\right)\subset\bRb^M$.\\
Due to the continuity of the functions $t\mapsto\int_{\bRb^d}x\mu_t^\ell(x)\ddx$, we deduce that $\mathcal{C}_\sigma$ is path-connected. This implies that $\mathcal{A}_\sigma$ is also path-connected.

\noindent{}{\bf Step 3.} \emph{The proof of the third point is presented in Subsection~\ref{tomhanks3}.}

We proceed a \emph{reductio ad absurdum} by assuming the existence of $(\nu^1,\ldots,\nu^M)\in\mathcal{A}_\sigma\subset\mathcal{S}_\sigma$ such that $\Upsilon_\sigma\left(\nu^1,\ldots,\nu^M\right)\neq L_\sigma$. We recall that $\Lambda_1$ is a bijection of $\mathcal{A}_\sigma$ to $\mathcal{C}_\sigma\subset\bRb^M$. Due to the continuity of $\Lambda_1$ and $\Upsilon_\sigma$, there exists a closed set $\mathcal{D}_\sigma\subset\bRb^M$ that contains the point $\Lambda_1\left(\nu^1,\ldots,\nu^M\right)$ in its interior and such that $\Upsilon_\sigma(\mu^1,\ldots,\mu^M)\neq L_\sigma$ for all states $(\mu^1,\ldots,\mu^M)$ satisfying $\Lambda_1\left(\mu^1,\ldots,\mu^M\right)\in\mathcal{D}_\sigma$.

We now use similar arguments to the ones in the first step. For some $1\leq\ell\leq M$, there exists a smooth function with compact support $\varphi^\ell$, a constant $\rho>0$ and two increasing sequences $\left(r_k\right)_{k\in\mathbb{N}}$ and $\left(s_k\right)_{k\in\mathbb{N}}$ such that $r_k<s_k$ and for all $t\in[r_k;s_k]$, we have
\begin{equation*}
\rho=\int_{\bRb}\varphi^\ell(x)\mu_{r_k}^\ell(\ddx)\leq\int_{\bRb}\varphi^\ell(x)\mu_t^\ell(\ddx)\leq\int_{\bRb}\varphi^\ell(x)\mu_{s_k}^\ell(\ddx)=2\rho\,.
\end{equation*}
Moreover, for all $t\in[r_k;s_k]$, $\Lambda_1\left(\mu_t^\ell\right)\in\mathcal{D}_\sigma$. By applying Proposition \ref{prop:a:othermeasure}, we construct a stationary state $(\widetilde{\nu}^1,\ldots,\widetilde{\nu}^M)\in\mathcal{S}_\sigma$ with free energy equal to $L_\sigma:=\lim_{t\to+\infty}\Upsilon_\sigma(\mu_t^1,\ldots,\mu_t^M)$ and such that $\Lambda_1\left(\widetilde{\nu}^1,\ldots,\widetilde{\nu}^M\right)\in\mathcal{D}_\sigma$. By construction of the set $\mathcal{D}_\sigma$, we have arrived at a contradiction.

\mathversion{bold}
\subsubsection{The set \texorpdfstring{$\mathcal{A}_\sigma$}{Asigma} is included into \texorpdfstring{$\mathcal{S}_\sigma$}{Ssigma}}
\label{tomhanks}
\mathversion{normal}
We proceed a \emph{reductio ad absurdum} by assuming the existence of $(\nu^1,\ldots,\nu^M)\in\mathcal{A}_\sigma$ such that $(\nu^1,\ldots,\nu^M)$ is not a stationary state. According to Proposition \ref{edeka}, there exist $1\leq\ell\leq M$, $q\in\left(\mathbb{N}^*\right)^d$ and $\rho>0$ such that 
\begin{equation*}
\inf_{(\mu^1,\ldots,\mu^M)\in\mathcal{S}_\sigma}\left|\int_{\bRb^d} x^q\mu^\ell(x)\ddx-\int_{\bRb^d} x^q\nu^\ell(x)\ddx\right|\geq\rho\,.
\end{equation*}
Let $H^q_\kappa\subset\bRb^M$ be defined as
\begin{equation*}
H^q_\kappa:=\bRb^{\ell-1}\times\left[\int_\bRb x^q\nu^\ell(x)\ddx-\kappa;\int_\bRb x^q\nu^\ell(x)\ddx+\kappa\right]\times\bRb^{M-\ell}\,.
\end{equation*}
It is a closed bounded set with nonempty interior. Let us define $\mathcal{H}_\kappa^q:=\Lambda_q^{-1}\left(\left\{H^q_\kappa\right\}\right)$. By construction, $(\nu^1,\ldots,\nu^M)\in\mathcal{H}_{\frac{\rho}{4}}^q$ and $\mathcal{H}_{\frac{\rho}{2}}^q\bigcap\mathcal{S}_\sigma=\emptyset$. In particular, $\mu_\infty^\ell\notin\mathcal{H}_{\frac{\rho}{2}}^q$ where $\mu_\infty^\ell$ has been introduced in Proposition~\ref{jigsaw}.

Nonetheless, $(\mu_\infty^1,\ldots,\mu_\infty^M)$ and $(\nu^1,\ldots,\nu^M)$ are adherence values of the family of $M$-uple measures $\left\{(\mu_t^1,\ldots,\mu_t^M)\,;\,t\geq0\right\}$. Consequently, for any $\ell$ such that $1\leq \ell\leq M$, there exist two increasing sequences $(r_k)_{k\in\mathbb{N}}$ and $(s_k)_{k\in\mathbb{N}}$ such that for all $k\in\mathbb{N}$, $\mu_{r_k}^\ell\in\partial\mathcal{H}_{\frac{\rho}{4}}^p$, $\mu_{s_k}^\ell\in\partial\mathcal{H}_{\frac{\rho}{2}}^p$ and for all $r_k<t<s_k$, we have
\begin{equation*}
\frac{\rho}{4}<\left|\int_{\bRb^d}x^q\mu_t^\ell(x)\ddx-\int_{\bRb^d} x^q\nu^\ell(x)\ddx\right|<\frac{\rho}{2}\,.
\end{equation*}
By construction of the sequences $(r_k)_{k\in\mathbb{N}}$ and $(s_k)_{k\in\mathbb{N}}$, $\int_{\bRb^d}x^q\mu_{r_k}^\ell(x)\ddx=\int_{\bRb^d} x^q\nu(x)\ddx+\frac{\rho}{4}\theta_k^{1}$ and $\int_{\bRb^d} x^q\mu_{s_k}^\ell(x)\ddx=\int_{\bRb^d}x^q\nu(x)\ddx+\frac{\rho}{2}\theta_k^{2}$ where $\theta_k^1:=\pm1$ and $\theta_k^2\in[-1;1]$.

We can extract two subsequences (we continue to write $r_k$ and $s_k$ for the comfort of the reading) such that there exists $\theta\in\{-1\,;\,1\}$, independent from the index $k$, which satisfies
\begin{align*}
\theta\int_{\bRb^d}x^q\mu_{r_k}^\ell(x)\ddx=\theta\int_{\bRb^d} x^q\nu^\ell(x)\ddx+\frac{\rho}{4}\,,\quad\theta\int_{\bRb^d} x^q\mu_{s_k}^\ell(x)\ddx=\theta\int_{\bRb^d}x^q\nu^\ell(x)\ddx+\frac{\rho}{2}
\end{align*}
and for all $t\in[r_k;s_k]$, we have
\begin{align*}
\theta\int_{\bRb^d}x^q\nu^\ell(x)\ddx+\frac{\rho}{4}\leq\theta\int_{\bRb^d} x^q\mu_t^\ell(x)\ddx\leq\theta\int_{\bRb^d}x^q\nu^\ell(x)\ddx+\frac{\rho}{2}\,.
\end{align*}
Moreover, for all $t\in[r_k;s_k]$, $(\mu_t^1,\ldots,\mu_t^M)\in\mathcal{H}_{\frac{\rho}{2}}^q$. Without loss of generality, we assume that $\theta=1$.\\[2pt]
We apply Proposition \ref{prop:a:othermeasure} with a smooth function with compact support which is equal to $x^q$ if $|x|\leq R$ and equal to $0$ if $|x|\geq R+1$.\\
Taking $R$ sufficiently large, we deduce the existence of $\left(\widetilde{\nu}^1,\ldots,\widetilde{\nu}^M\right)\in\mathcal{H}_{\frac{\rho}{2}}^q\bigcap\mathcal{A}_\sigma\bigcap\mathcal{S}_\sigma$. However, $\mathcal{S}_\sigma\bigcap\mathcal{H}_{\frac{\rho}{2}}^q=\emptyset$. We have arrived at a contradiction.

\mathversion{bold}
\subsubsection{The set \texorpdfstring{$\mathcal{A}_\sigma$}{Asigma} is path-connected}
\mathversion{normal}
\label{tomhanks2}
According to the previous paragraph, the set $\mathcal{A}_\sigma$ is included in $\mathcal{S}_\sigma$, the set of stationary states. We now consider the application $\Lambda_q$ from the set of the probability measures to $\bRb^M$ as defined in Definition~\ref{tototo}.

By $\mathcal{C}_\sigma$, we denote the set $\Lambda_1\left(\mathcal{A}_\sigma\right)$. Due to the continuity of the application $t\mapsto\Lambda_1(\mu_t)$, we deduce that the set of limiting values of the family $\left\{\Lambda_1\left(\mu_t\right)\,;\,t\geq0\right\}$ is path-connected. In other words, the set $\mathcal{C}_\sigma$ is path-connected.

Consequently, for all $M$-uple of probability measures $(\mu_0^1,\ldots,\mu_0^M)$ and $(\mu_1^1,\ldots,\mu_1^M)$ in $\mathcal{A}_\sigma$, there exists an application from $[0;1]$ to the set of the $M$-uple of probability measures, $\theta\mapsto(\mu_\theta^1,\ldots,\mu_\theta^M)$ such that for all $1\leq\ell\leq M$, the functions $\theta\mapsto\int_{\bRb^d} x\mu_\theta^\ell(x)\ddx$ are continuous.

We deduce that for all $x\in\bRb$, for all $1\leq\ell\leq M$, the function $\theta\mapsto V_\ell(x)+\sum_{j=1}^Ma_jF_{\ell j}\ast\mu_\theta^j(x)$ is continuous. We know that for all $\theta\in[0;1]$, the measure $\mu_\theta^\ell$ is absolutely continuous with respect to the Lebesgue measure with a density $\mu_\theta^\ell$ that satisfies
\begin{equation*}
\mu_\theta^\ell(x)=\frac{\exp\left[-\frac{2}{\sigma_\ell^2}\left(V_\ell(x)+\sum_{j=1}^Ma_jF_{\ell j}\ast \mu_\theta^j(x)\right)\right]}{\int_{\bRb^d}\exp\left[-\frac{2}{\sigma_\ell^2}\left(V_\ell(y)+\sum_{j=1}^Ma_jF_{\ell j}\ast \mu_\theta^j(y)\right)\right]\ddy}\,.
\end{equation*}
Hence, for any smooth function with compact support $\varphi^\ell$, the application $\theta\mapsto\int_{\bRb^d}\phi^\ell(x)\mu_\theta^\ell(x)\ddx$ is continuous. Therefore, the set $\mathcal{A}_\sigma$ is path-connected.

\mathversion{bold}
\subsubsection{The free energy is constant on \texorpdfstring{$\mathcal{A}_\sigma$}{Asigma}}
\mathversion{normal}
\label{tomhanks3}
We introduce $L_\sigma:=\lim_{t\to+\infty}\Upsilon_\sigma(\mu_t^1,\ldots,\mu_t^M)$. We know that for any $(\mu^1,\ldots,\mu^M)\in\mathcal{A}_\sigma$, we have the estimate $\Upsilon_\sigma(\mu^1,\ldots,\mu^M)\leq L_\sigma$. In particular, the free energy of a limiting value is less than $L_\sigma$, the limit of the free energy.

We will argue again by contradiction; we assume that there exist $(\nu^1,\ldots,\nu^M)\in\mathcal{A}_\sigma$ and $\eta>0$ satisfying $\Upsilon_\sigma\left(\nu^1,\ldots,\nu^M\right)=L_\sigma-2\eta$. We use arguments similar to those of the previous subsections.

Due to the continuity of $\Lambda_q$, we deduce that there exists $\rho>0$ sufficiently small such that $\Upsilon_\sigma(\mu^1,\ldots,\mu^M)\leq L_\sigma-\eta$ for all $(\mu^1,\ldots,\mu^M)\in\mathcal{S}_\sigma\bigcap\mathcal{H}_\rho^q$ where $\mathcal{H}_\rho^q$ is defined like in Subsection~\ref{tomhanks}, as the set of the stationary states $(\mu^1,\ldots,\mu^M)$ such that
\[
\left|\int_{\bRb^d} x^q\mu^\ell(x)\ddx-\int_{\bRb^d}x^q\nu^\ell(x)\ddx\right|\leq\rho\,,
\]
for all $1\leq\ell\leq M$.

Nevertheless, the stationary state $(\mu_\infty^1,\ldots\mu_\infty^M)\in\mathcal{A}_\sigma$ introduced in Proposition~\ref{jigsaw} satisfies the equality $\Upsilon_\sigma\left(\mu_\infty^1,\ldots\mu_\infty^M\right)=L_\sigma$. Hence, $(\mu_\infty^1,\ldots\mu_\infty^M)\notin\mathcal{H}_\rho$.

We now proceed exactly as in Subsection~\ref{tomhanks}. We obtain the existence of two increasing sequences $(r_k)_{k\in\mathbb{N}}$ and $(s_k)_{k\in\mathbb{N}}$, of $1\leq\ell_0\leq M$ and $\theta\in\{-1\,;\,1\}$, all independent of the index $k$, which satisfy
\begin{align*}
\theta\int_{\bRb^d} x^q\mu_{r_k}^{\ell_0}(x)\ddx=\theta\int_{\bRb^d} x^q\nu^{\ell_0}(x)\ddx+\frac{\rho}{4}\,,\quad\theta\int_{\bRb^d} x^q\mu_{s_k}^{\ell_0}(x)\ddx=\theta\int_{\bRb^d} x^q\nu^{\ell_0}(x)\ddx+\frac{\rho}{2}
\end{align*}
and for all $t\in[r_k;s_k]$, we have
\begin{align*}
\theta\int_{\bRb^d} x^q\nu^{\ell_0}(x)\ddx+\frac{\rho}{4}\leq\theta\int_{\bRb^d} x^q\mu_t^{\ell_0}(x)\ddx\leq\theta\int_{\bRb^d} x^q\nu^{\ell_0}(x)\ddx+\frac{\rho}{2}\,.
\end{align*}
Moreover, for all $t\in[r_k;s_k]$, $(\mu_t^1,\ldots,\mu_t^M)\in\mathcal{H}_\rho^q$. Without any loss of generality, we assume that $\theta=1$. We apply Proposition \ref{prop:a:othermeasure} with a smooth function with compact support equal to $x^p$ if $|x|\leq R$ and equal to $0$ if $|x|\geq R+1$. By taking $R$ sufficiently large, we deduce the existence of a $M$-uple of measures $\left(\widetilde{\nu}^1,\ldots\widetilde{\nu}^M\right)\in\mathcal{H}_{\frac{\rho}{2}}^q\bigcap\mathcal{A}_\sigma$ such that $\Upsilon_\sigma\left((\widetilde{\nu}^1,\ldots\widetilde{\nu}^M\right)=L_\sigma$. Nonetheless, by construction of $\mathcal{H}_{\frac{\rho}{2}}^q$, we have $\Upsilon_\sigma\left((\widetilde{\nu}^1,\ldots\widetilde{\nu}^M\right)\leq L_\sigma-\eta$. We have arrived at a contradiction.

\subsection{Proof of Corollary~\ref{piccolo}}
According to Theorem~\ref{chichi}, the set of the limiting values $\mathcal{A}_\sigma$ is a path-connected subset of $\mathcal{S}_\sigma$ in which the free-energy is constant. Due to the hypothesis of Corollary~\ref{piccolo}, any path-connected subset of $\mathcal{S}_\sigma$ with constant free-energy is a single element. This achieves the proof.

\subsection{Proof of Corollary~\ref{piccolo2}}

\noindent{}{\bf Step 1.} We introduce $\mathcal{S}_0$, the set of stationary states $(\mu^1,\ldots,\mu^M)$ such that there exist decreasing sequences $(\sigma_\ell^k)_{k\in\mathbb{N}}$ which go to $0$ for any $1\leq\ell\leq M$ and a sequence $\left(\nu^{1,\sigma_1^k},\ldots\nu^{M,\sigma_M^k}\right)_{k\in\mathbb{N}}$ of stationary states of diffusion \eqref{eq: mean-field SDE} which converges weakly to $(\mu^1,\ldots,\mu^M)$, as $k$ goes to infinity.

\noindent{}{\bf Step 2.} By proceeding exactly like in \cite{JOTP} (Proposition 3.10, Lemma 3.3, Theorem 3.7 and Proposition 3.8 of the aforementioned paper hold under the set of assumptions of this section), we obtain that the set $\mathcal{S}_0$ is included into $\left\{\delta_{(a_1,\cdots,a_M)}\,;\,a_1\in\mathcal{G}_1\,,\cdots,a_M\in\mathcal{G}_M\right\}$.

\noindent{}{\bf Step 3.} We deduce that for any $\rho>0$, there exists $\sigma_0>0$ sufficiently small such that for all $\sigma_1,\ldots,\sigma_M<\sigma_0$, for all $\nu\in\mathcal{S}_\sigma$ and for all $1\leq\ell\leq M$, we have $\displaystyle\min_{a_\ell\in\mathcal{G}_\ell}\int_{\bRb^d}|x-a_\ell|^2\nu^\ell(x)\ddx\leq\rho$.

\noindent{}{\bf Step 4.} Theorem~\ref{chichi} tells us that $\mathcal{A}_\sigma\subset\mathcal{S}_\sigma$. This achieves the proof.

\bigskip

Let us remark that we do not have \emph{a priori} the existence of $a_\ell\in\mathcal{G}_\ell$ such that
\begin{equation*}
\limsup_{\sigma_1,\ldots,\sigma_M\to0}\limsup_{t\to+\infty}\int_{\bRb^d}|x-a_\ell|^2\mu_t^\ell(x)\ddx=0\,.
\end{equation*}

\subsection*{Acknowledgments} 
\RV{We would like to thank the referees for their useful comments and suggestions which have helped us to improve the presentation of the paper.}

The research of HD was supported by EPSRC Grant No. EP/Y008561/1. GP is partially supported by an ERC-EPSRC Frontier Research Guarantee through Grant No. EP/X038645, ERC Advanced Grant No. 247031 and a Leverhulme Trust Senior Research Fellowship, SRF$\backslash$R1$\backslash$241055.

JT was supported by the French ANR grant METANOLIN (ANR-19-CE40-0009).

\appendix
%
%
\section{Algebraic study of Stationary States in the Small-Noise Limit}
\label{app:pogacar}

In this section, we will focus on solving the algebraic equation~\eqref{eq: m-systems}. Each solution to this system will be called a ``candidate solution''. If we can construct such a candidate solution and if we can show that it satisfies~\eqref{linegalite} then we have, at least for $\sigma_0$ sufficiently small, a stationary state. This follows from Proposition~\ref{thunder}.

To simplify the analysis, we will assume that the synchronization hypothesis, Assumption~\ref{asp: assumption:synchronization} holds. Formally, the convexity of the interacting potentials compensates for the lack of convexity of the confining ones. That is, we assume that for each measure $(\nu^1,\cdots,\nu^M)$, the effective potential $W_k$ is convex for any $1\leq k\leq M$. Here, the effective potential $W_k$ is defined as:
\[
W_k(x):=V_k(x)+\sum_{\ell=1}^M a_\ell F_{k\ell}\ast\mu^\ell(x)\,.
\]
When $\nu^k=\mu^k$ is an invariant probability measure, the potential $W_k$ is the potential that drives the $k$-th SDE, starting from the steady state. In particular, for any $m^0$, $m_k^0$ is always the unique global minimizer of $W_k^{(m^0)}$ so that we can always apply Proposition~\ref{thunder}.

Trying to verify that this assumption is satisfied may seem non-trivial. However, when $F_{k\ell}$ is quadratic, we can easily prove in the multi-species model as well as in the one-species model that the convexity of the potentials $V_k+\sum_{\ell=1}^m a_\ell F_{k\ell}\ast\mu^\ell$ is \emph{equivalent} to the inequality

\[
\sum_{\ell=1}^M a_\ell\alpha_{k\ell}\geq \sup_{\bRb^d}-\nabla^2V_k\,,
\]
for any $k=1,\ldots, M$. 

\begin{remark}
Since $\sum_{\ell=1}^M a_\ell=1$ we have 
\[
\sum_{\ell=1}^M a_\ell\alpha_{k\ell}\geq \left(\min_{1\leq\ell\leq M}\alpha_{k\ell} \right)\sum_{\ell=1}^M a_{\ell}=\min_{1\leq\ell\leq M}\alpha_{k\ell}.
\]
It implies that if $\min_{\ell}\alpha_{k\ell}\geq \sup (-V''_k)$, then the synchronization condition~\ref{asp: assumption:synchronization} is fulfilled.
\end{remark}

We also emphasize the fact that that for any $1\leq k\leq M$ and any $m=(m_1,\cdots,m_M)\in\bRb^M$, the functions $W_k^{(m)}$ defined as
\[
W_k^{(m)}(x):=V_k(x)+\frac{1}{2}\sum_{\ell=1}^M a_\ell\alpha_{k\ell}\left(x-m_\ell\right)^2\,,
\]

are convex.
%
%
\subsection{Study of the Algebraic equations \texorpdfstring{\eqref{eq: m-systems}}{(23)} under the Structural Assumption}

We recall that the structural assumption is
\begin{equation*}
    \frac{V_k}{\sigma_k^2}=\frac{\bar{V}}{\sigma^2} \quad \forall k, \quad\text{and}\quad \frac{\alpha_{kl}}{\sigma_k^2}=\alpha_\ell, \quad k=1, \dots M,
\end{equation*}
and that the algebraic system \eqref{eq: m-systems} reads
\[
V_k'(m_k)+\Big(\sum_{\ell=1}^M a_\ell \alpha_{k\ell}\Big)m_k=\sum_{\ell=1}^Ma_\ell\alpha_{k\ell}m_\ell\quad \text{for all}\quad k=1,\ldots, M,
\]
which are equivalent to the following system; 
\[
\frac{V_k'(m_k)}{\sigma_k^2}+\Big(\sum_{\ell=1}^M a_\ell\frac{ \alpha_{k\ell}}{\sigma_k^2}\Big)m_k=\sum_{\ell=1}^Ma_\ell\frac{\alpha_{k\ell}}{\sigma_k^2}m_\ell\quad \text{for all}\quad k=1,\ldots, M.
\]
Under the structural assumption, the above system reduces to
\[
\frac{\bar{V}'(m_k)}{\sigma^2}+\Big(\sum_{\ell=1}^M a_\ell\frac{\alpha_\ell}{\sigma^2}\Big)m_k=\sum_{\ell=1}^Ma_\ell\frac{\alpha_\ell}{\sigma^2} m_\ell\quad \text{for all}\quad k=1,\ldots, M.
\]
This implies that for any $1\leq i\neq j\leq M$
\[
\bar{V}'(m_i)+\Big(\sum_{\ell=1}^Ma_\ell\alpha_{\ell}\Big)m_i=\bar{V}'(m_j)+\Big(\sum_{\ell=1}^Ma_\ell\alpha_{\ell}\Big)m_j
\]
which can be rewritten as
\[
(m_i-m_j)\left(\int_{0}^1 \bar{V}''(m_i+ t(m_j-m_i))\,\ddt+ \sum_{\ell=1}^Ma_\ell\alpha_{\ell}\right)=0.
\]
Thus, for any $1\leq i\neq j\leq M$ we have $m_i=m_j$ or
\[
\int_{0}^1 \bar{V}''(m_i+ t(m_j-m_i))\,\ddt+ \sum_{\ell=1}^Ma_\ell\alpha_{\ell} =0.
\]
As a consequence, if the synchronization condition is satisfied, that is, if $\bar{V}''(x)>-\sum_{\ell=1}^Ma_\ell\alpha_{\ell}$ for all $x$, then the latter equation has no solution. Therefore, we deduce that
\[
m_1=m_2=\ldots=m_M=m^* \quad\text{where}\quad \bar{V}'(m^*)=0
\]
are the only solutions of the algebraic system \eqref{eq: m-systems}. 
For example, for the double-well potential $V(x)=\frac{x^4}{4}-\frac{x^2}{2}$, if $\sum_{\ell=1}^Ma_\ell\alpha_{\ell}>1$, then we must have
\[
m_1=m_2=\ldots=m_M=m^*, \quad V'(m^*)=0.
\]
It is clear that in the general case (where we only assume that $V_i$ are the same), $m_1=\ldots=m_M=m^*$ with $ V'(m^*)=0$ is always a solution, but how to establish a simple criterion for the (non-)existence of other solutions as above is not obvious.
%
%
\subsection{The Two-Species Case in One Dimension}

We consider now the case where
\[
d=1, V_1=V_2=V, \quad \alpha=\alpha_{1,2}=\alpha_{2,1}. 
\]

Since we have assumed that the synchronization assymption holds, the candidates are of the form $(m_1,m_2)$. In particular, any equilibrium point converges, in the small-noise limit, to Dirac measures of the form $\delta_{(m_1,m_2)}$.

Then, the previous equation becomes the following.
\begin{align}
&V'(m_1)+(1-a)\alpha(m_1-m_2)=0\label{m1-2}\\
&V'(m_2)-a\alpha(m_1-m_2)=0\,.\label{m2-2}
\end{align}
From \eqref{m1-2} and \eqref{m2-2}, we obtain
\begin{align}
&m_1=m_2+\frac{V'(m_2)}{a \alpha},\label{m1-2bis}\\
&V'\Big(m_2+\frac{V'(m_2)}{a \alpha}\Big)+\frac{1-a}{a}V'(m_2)=0\label{m2-2bis}
\end{align}
Using the fundamental theorem of calculus,
\[
F(b)-F(a)=\int_0^1 \frac{\dd}{\ddt}F(a+t(b-a))\,\ddt=(b-a)\int_0^1 F'(a+t(b-a))\,\ddt,
\]
\eqref{m2-2} can be written as
\begin{equation}
\label{eq: m2-3}
\frac{V'(m_2)}{a}\Big[\frac{1}{\alpha}\int_0^1 V''\left(m_2+t\frac{V'(m_2)}{a\alpha}\right)\,\ddt+1 \Big]=0,
\end{equation}
which implies that $m_2$ satisfies $V'(m_2)=0$ or
\begin{equation}
\label{eq: eqn2 for m2}    
\int_0^1 V''\Big(m_2+t\frac{V'(m_2)}{a\alpha}\Big)\,\ddt+\alpha=0.
\end{equation}
As a consequence, if $V''(x)> -\alpha$ for all $x$ then the latter has no solution.

\begin{remark}
We point out that $V''+\alpha>0$ is not the synchronization assumption. In fact, $\alpha$ is just the cross-interaction parameter. The synchronization assumption is $V''+a\alpha_{11}+(1-a)\alpha>0$. 

\end{remark}

Next, we note that \eqref{eq: eqn2 for m2} can be written as
\[
(V''(m_2)+\alpha)+\frac{V'(m_2)}{a\alpha}\int_0^1\int_0^1 t V^{(3)}\Big(m_2+s t \frac{V'(m_2)}{a\alpha}\Big)\,\dds\,\ddt=0.
\]
Generally, suppose $V\in C^k$, then \eqref{eq: eqn2 for m2} can be written as
\begin{align}
\label{eq: general eqn for m2}
&\alpha+\sum_{i=0}^{k-3}\frac{1}{(i+1)!}V^{(i+2)}(m_2)\Big(\frac{V'(m_2)}{a \alpha}\Big)^i\\
&\nonumber+\Big(\frac{V'(m_2)}{a\alpha}\Big)^{k-2}\int_{[0,1]^{k-1}} \prod_{i=1}^{k-2}t_i^{k-1-i} V^{(k)}\Big(m_2+t_{k-1}\prod_{i=1}^{k-2}t_i^{k-1-i} \frac{V'(m_2)}{a\alpha}\Big) \, \ddt_{k-1}\,\ldots\, \ddt_1=0.
\end{align}
In particular, if $V$ is a polynomial of degree $k-1$, then $V^{(k)}=0$, thus the last term in \eqref{eq: general eqn for m2} vanishes, and it becomes
\begin{equation}
\label{eq: eqn for m2 V polynomial}
\alpha+\sum_{i=0}^{k-3}\frac{1}{(i+1)!}V^{(i+2)}(m_2)\Big(\frac{V'(m_2)}{a \alpha}\Big)^i=0,   
\end{equation}
which is a polynomial equation of degree $(k-3)(k-2)$. This implies that the system \eqref{m1-2}-\eqref{m2-2} can have maximum $(k-3)(k-2)+(k-2)=(k-2)^2$ solutions. This fact can alternatively be derived from Bezout's theorem for the polynomial system \eqref{m1-2}-\eqref{m2-2}. Note also that if $V$ is an even polynomial, then so is the LHS of \eqref{eq: eqn for m2 V polynomial}. The number of real roots of a polynomial equation in any interval can be found exactly using Sturm's theorem. 

We also observe that if $\tilde m$ is a critical point of $V$ (i.e., $V'(\tilde m)=0$) and $V''(\tilde m)\neq -\alpha$, then it is not a solution of \eqref{eq: eqn for m2 V polynomial}.

\subsection{The two-species system in a double-well potential}
\[
V(x)=\frac{x^4}{4}-\frac{x^2}{2}, \quad V'(x)=x^3-x,\quad V''(x)=3x^2-1.
\]
Thus, if $\alpha>1$, there are no other equilibria apart from the critical points of $V$.

Now, consider $\alpha<1$
\begin{align*}
0&=\frac{1}{\alpha}\int_0^1 V''\left(m_2+t\frac{V'(m_2)}{a\alpha}\right)\,\ddt +1
\\&=\frac{1}{\alpha}\int_0^1 \Big[3\left(m_2+t\frac{m_2^3-m_2}{a\alpha}\right)^2-1\Big]\,\ddt+1
\\&=\frac{1}{\alpha}\Big[(3m_2^2-1)+\frac{3m_2^2(m_2^2-1)}{a \alpha}+\frac{m_2^2(m_2^2-1)^2}{a^2\alpha^2}\Big]+1
\end{align*}
which can be obtained in fact using \eqref{eq: eqn for m2 V polynomial}. This gives, for $X:=m_2^2$
\begin{equation}
\label{eq: cubic for m2}
X^3+(3a\alpha -2) X^2 +[3a\alpha(a\alpha-1)+1] X+a^2\alpha^2(\alpha-1)=0.    
\end{equation}
Note that $3a\alpha(a\alpha-1)+1>0$. If $a\alpha\geq \frac{2}{3}$ then $3a\alpha-2\geq 0$, the coefficients of this cubic equation has only one change of signs. Therefore, by Descartes' rule of sign, it has exactly one positive root. When $0<a\alpha<2/3$, it may have one or three positive roots (including multiplicity).

We point out that the synchronization assumption is $a\alpha_{11}+(1-a)\alpha>1$, which is different from hypothesis $a\alpha\geq \frac{2}{3}$.

The following lemma determines the conditions for the coefficients of a general cubic equation so that it has a certain number of positive roots. 
\begin{lemma}
Consider the cubic equation
\[
P(x)=x^3+ px^2+qx+r=0
\]
and the associated discriminant
\[
\Delta=\Delta(p,q,r):=-27 r^2+18pqr-4q^3-4p^3 r+p^2q^2.
\]
Then
\begin{enumerate}[(i)]
    \item $P$ has three real, distinct, positive roots if and only if
\begin{align*}
    \Delta> 0, \quad p<0, \quad q>0, \quad r<0.
\end{align*} 
  \item $P$ has two real positive roots (i.e, a double root $x_1>0$ and a single root $x_1\neq x_2=x_3>0$) if and only if
\begin{align*}
    \Delta=0, \quad p^2\neq 3q,\quad p<0, \quad q>0, \quad r<0.
\end{align*} 
  \item $P$ has  one positive root and two non-real complex conjugate roots if and only if
\begin{align*}
    \Delta<0, \quad r<0.
\end{align*} 
  \item $P$ has a triple positive root if and only if
\begin{align*}
        \Delta=0, \quad p^2= 3q,\quad p<0.
\end{align*} 
  \end{enumerate}
\end{lemma}

Applying the above lemma, consider the following inequality.
\begin{multline}
\label{eq:inqAB}
-27 (B^2(A-1))^2+18 (3B-2)(3B^2-3B+1)B^2(A-1)-4(3B^2-3B+1)^3-4(3B-2)^3 B^2(A-1)
\\+(3B-2)^2(3B^2-3B+1)^2>0
\end{multline}
where $0<A=\alpha<1$, $0<B=a \alpha<2/3$. It is much more convenient to work in terms of $A$ and $B$ than $a$ and $\alpha$. \eqref{eq:inqAB} can be written as
\[
B^2(27 B^4-54 AB^3+(27A^2+18)B^2-18 AB-(1-4A))<0,
\]
which is equivalent to
\begin{align*}
&\qquad  (27 B^4-54 AB^3+(27A^2+18)B^2-18 AB-(1-4A))<0
\\& \Longleftrightarrow \Big[B^2-AB+\big(\frac{2\sqrt{1-A}}{3\sqrt{3}}+\frac{1}{3}\big)\Big]\Big[B^2-AB+\big(-\frac{2\sqrt{1-A}}{3\sqrt{3}}+\frac{1}{3}\big)\Big]<0\\
&\Longleftrightarrow B^2-AB+\big(-\frac{2\sqrt{1-A}}{3\sqrt{3}}+\frac{1}{3}\big)<0,
\end{align*}
because
\[
\Big[B^2-AB+\left(\frac{2\sqrt{1-A}}{3\sqrt{3}}+\frac{1}{3}\right)\Big]>0\quad \text{since}\quad A^2-4\left(\frac{2\sqrt{1-A}}{3\sqrt{3}}+\frac{1}{3}\right)<0 \quad\text{for}\quad 0<A<1.
\]
Solving the last inequality for $0<A<1,~ 0<B<2/3$ we get
\[
0<A<\frac{1}{4}\quad \text{and}\quad 0<B<\frac{1}{2}\left(A+\sqrt{A^2+\frac{8\sqrt{1-A}}{3\sqrt{3}}-\frac{4}{3}}\right).
\]
and
\[
1/4\leq A<2/3\quad \text{and}\quad \frac{1}{2}\left(A-\sqrt{A^2+\frac{8\sqrt{1-A}}{3\sqrt{3}}-\frac{4}{3}}\right)<B<\frac{1}{2}\left(A+\sqrt{A^2+\frac{8\sqrt{1-A}}{3\sqrt{3}}-\frac{4}{3}}\right).
\]
Note that for $A>2/3$
\[
B^2-AB+\big(-\frac{2\sqrt{1-A}}{3\sqrt{3}}+\frac{1}{3}\big)>0.
\]
%
%
\section{Quadratic confining potentials}
\label{app:quadr}

We perform explicit calculations for the case of quadratic potentials.
\[
V_1(x)=\frac{1}{2} x^T p x, \quad V_2(x)=\frac{1}{2}x^T q x,
\]
where $p,q\in\R^{d\times d}$ are two positive symmetric definite matrices.

Using the following Gaussian integrals, where $A\in\R^{d\times d}$ is a given symmetric positive definite matrix and $B\in\R^d$ is a given vector:
\begin{align*}
&\int_{\R^d}\exp\Big(-\frac{1}{2}x^T A x+ x^TB\Big)\,\ddx=\sqrt{\frac{(2\pi)^d}{\det A}}\, \exp\Big(\frac{1}{2}B^T A^{-1}B\Big),
\\ &\int_{\R^d}x\exp\Big(-\frac{1}{2}x^T A x+ x^TB\Big)\,\ddx=\sqrt{\frac{(2\pi)^d}{\det A}}\, \exp\Big(\frac{1}{2}B^T A^{-1}B\Big) A^{-1}B,
\end{align*}
we calculate
\begin{align*}
(i)\quad &\int_{\R^d} \exp\Big(-\frac{2}{\sigma^2}\big[V_1(x)+a\big(\frac{1}{2}x^T\alpha_{11} x-x^T\alpha_{11} m_1\big)+(1-a)\big(\frac{1}{2}x^T\alpha_{12} x-x^T\alpha_{12} m_2\big)\big]\Big)\,\ddx
\\&=\int_{\R^d} \exp\Big(-\frac{2}{\sigma^2}\Big[\frac{1}{2}x^T(p+a\alpha_{11}+(1-a)\alpha_{12})x-x^T\big(a\alpha_{11} m_1+(1-a)\alpha_{12} m_2\big)\Big]\Big)\,\ddx
\\&=\int_{\R^d} \exp\Big(-\frac{1}{2}x^T A_1 x+x^TB_1\Big)\,\ddx
\\&=\sqrt{\frac{(2\pi)^d}{\det A_1}}~ \exp\Big(\frac{1}{2}B_1^T A_1^{-1}B_1\Big),
\\ \\
(ii)\quad &\int_{\R^d} x\exp\Big(-\frac{2}{\sigma^2}\big[V_1(x)+a\big(\frac{1}{2}x^T\alpha_{11} x-x^T\alpha_{11} m_1\big)+(1-a)\big(\frac{1}{2}x^T\alpha_{12} x-x^T\alpha_{12} m_2\big)\big]\Big)\,\ddx
\\&=\int_{\R^d} x\exp\Big(-\frac{1}{2}x^T A_1 x+x^TB_1\Big)\,\ddx
\\&=\sqrt{\frac{(2\pi)^d}{\det A_1}}~ \exp\Big(\frac{1}{2}B_1^T A_1^{-1}B_1\Big)A_1^{-1}B_1
\\ \\
(iii)\quad& \int_{\R^d}\exp\Big(-\frac{2}{\sigma^2}\big[V_2(x)+a\big(\frac{1}{2}x^T\alpha_{21} x-x^T\alpha_{21} m_1\big)+(1-a)\big(\frac{1}{2}x^T\alpha_{22} x-x^T\alpha_{22} m_2\big)\big]\Big)\,\ddx
\\&=\int_{\R^d} \exp\Big(-\frac{1}{2}x^T A_2 x+x^TB_2\Big)\,\ddx
\\&=\sqrt{\frac{(2\pi)^d}{\det A_2}}~ \exp\Big(\frac{1}{2}B_2^T A_2^{-1}B_2\Big),\\ \\
(iv)\quad& \int_{\R^d} x \exp\Big(-\frac{2}{\sigma^2}\big[V_2(x)+a\big(\frac{1}{2}x^T\alpha_{21} x-x^T\alpha_{21} m_1\big)+(1-a)\big(\frac{1}{2}x^T\alpha_{22} x-x^T\alpha_{22} m_2\big)\big]\Big)\,\ddx
\\&=\int_{\R^d} x \exp\Big(-\frac{1}{2}x^T A_2 x+x^TB_2\Big)\,\ddx
\\&=\sqrt{\frac{(2\pi)^d}{\det A_2}}~ \exp\Big(\frac{1}{2}B_2^T A_2^{-1}B_2\Big)A_2^{-1}B_2,
\end{align*}
where
\begin{align*}
A_1&=\frac{2}{\sigma^2} (p+a\alpha_{11}+(1-a)\alpha_{12}),\\
B_1&=\frac{2}{\sigma^2}\big(a\alpha_{11} m_1+(1-a)\alpha_{12} m_2\big),\\
A_2&=\frac{2}{\sigma^2} (q+a\alpha_{21}+(1-a)\alpha_{21}),\\
B_2&=\frac{2}{\sigma^2}\big(a\alpha_{21} m_1+(1-a)\alpha_{22} m_2\big).
\end{align*}
Substituting these computations to \eqref{eq: m-systems} we obtain
\begin{subequations}
 \begin{align*}
 m_1&=A_1^{-1} B_1=(p+a\alpha_{11}+(1-a)\alpha_{12})^{-1}\big(a\alpha_{11} m_1+(1-a)\alpha_{12} m_2\big),
 \\ m_2&=A_2^{-1} B_2=(q+a\alpha_{21}+(1-a)\alpha_{22})^{-1}\big(a\alpha_{21} m_1+(1-a)\alpha_{22} m_2\big).
 \end{align*}
\end{subequations}
This system can be simplified to 
\begin{subequations}
 \begin{align*}
 (p+(1-a)\alpha_{12})m_1-(1-a)\alpha_{12}m_2&=0,\\
-a\alpha_{21}m_1+ (q+a\alpha_{21})m_2&=0,
 \end{align*}
\end{subequations}
which is equivalent to
 \begin{equation}
\label{eq: m-system-reduced}     
 m_2=\Big(\frac{1}{1-a}\alpha_{12}^{-1}p+I\Big)m_1\quad\text{and}\quad
\Big[-a \alpha_{21}+(q+a\alpha_{21})\Big(\frac{1}{1-a}\alpha_{12}^{-1}p+I\Big)\Big] m_1=0.
\end{equation}
Thus for all $\sigma>0$:
\begin{enumerate}[(1)]
    \item If $-a \alpha_{21}+(q+a\alpha_{21})\Big(\frac{1}{1-a}\alpha_{12}^{-1}p+I\Big)$ is non-singular, then \eqref{eq: m-system-reduced} has a unique solution $m_1=m_2=0$. In this case, the coupled McKean-Vlasov diffusions has a unique equilibrium where both $\mu$ and $\nu$ are Gaussian measures
    \begin{align*}
    &\mu(x)\propto \exp\Big(-\frac{1}{\sigma^2}x^T(p+a \alpha_{11}+(1-a)\alpha_{12})x\Big),
    \\&\nu(x)\propto \exp\Big(-\frac{1}{\sigma^2}x^T(q+a \alpha_{21}+(1-a)\alpha_{22})x\Big).
    \end{align*}
In particular, this is the case when $\alpha_{21}=\alpha_{12}$.
    \item If $-a \alpha_{21}+(q+a\alpha_{21})\Big(\frac{1}{1-a}\alpha_{12}^{-1}p+I\Big)$ is singular, then \eqref{eq: m-system-reduced} has infinitely number of solutions. Thus the coupled McKean-Vlasov diffusions has infinitely number of equilibria (including the above Gaussian measures). 
    \end{enumerate}
It is interesting to note that the condition does not depend on $\alpha_{11}$ and $\alpha_{22}$, that is, it does not depend on self-interacting potentials.

%
%
\section{A Technical Lemma Used in the Proof of Theorem~\ref{chichi}}
\label{app:converg}

In this appendix, we present the following proposition, which is used several times in the proof of Theorem~\ref{chichi} to construct elements of $\mathcal{S}_\sigma$ satisfying some hypotheses.
\begin{prop}
\label{prop:a:othermeasure}
Let the assumptions of Theorem~\ref{chichi} hold. We also assume the existence of two polynomial functions $\mathcal{P}$ and $\mathcal{Q}$, a smooth function $\varphi^\ell$ from $\bRb^d$ to $\bRb$ with compact support such that
$\left|\varphi^\ell(x)\right|\leq\mathcal{P}\left(|x|\right)$ and
$\left|\nabla\varphi^\ell(x)\right|^2\leq\mathcal{Q}\left(|x|\right)$, $\kappa>0$ and two sequences $(r_k)_{k\in\mathbb{N}}$ and $(s_k)_{k\in\mathbb{N}}$ which go to infinity such that
for all $r_k\leq t\leq s_k<r_{k+1}$, we have
\begin{equation*}
\kappa=\int_{\bRb^d}\varphi^\ell(x)\mu_{r_k}^\ell(x)\ddx\leq\int_{\bRb^d}\varphi^\ell(x)\mu_t^\ell(x)\ddx\leq\int_{\bRb^d}\varphi^\ell(x)\mu_{s_k}^\ell(x)\ddx=2\kappa\,.
\end{equation*}
Thus, there exists a stationary state $(\nu^1,\ldots,\nu^M)\in\mathcal{A}_\sigma\bigcap\mathcal{S}_\sigma$ that satisfies $\int_{\bRb^d}\varphi^\ell(x)\nu^\ell(x)\ddx\in[\kappa;2\kappa]$. Moreover, $\displaystyle\Upsilon_\sigma(\nu^1,\ldots,v^M)=\lim_{t\to+\infty}\Upsilon_\sigma(\mu_t^1,\ldots,\mu_t^M)$.
\end{prop}
\begin{proof}

{\bf Step 1.} We begin to prove that
$\displaystyle\liminf_{k\longrightarrow+\infty}\left(s_k-r_k\right)>0$. We
introduce the function
\begin{equation*}
\Phi(t):=\int_{\bRb^d}\varphi^\ell(x)\mu_t^\ell(x)\ddx\,.
\end{equation*}
This function is well defined, as $\left|\varphi^\ell(x)\right|$ is bounded by $\mathcal{P}\left|x|\right)$. The derivation of $\Phi$ and an integration by parts lead to
\begin{align*}
\Phi'(t)&=-\int_{\bRb^q}\left\langle\nabla\varphi^\ell(x)\,;\,\frac{\sigma_\ell^2}{2}\nabla \mu_t^\ell(x)+\mu_t^\ell(x)\left(\nabla V_\ell(x)+\sum_{j=1}^Ma_j\nabla F_{\ell j}\ast \mu_t^j(x)\right)\right\rangle\ddx\\
&=-\int_{\bRb^d}\left\langle\nabla\varphi^\ell(x)\,;\,\eta_t^\ell(x)\right\rangle\mu_t^\ell(x)\ddx\,.
\end{align*}
The Cauchy-Schwarz inequality implies
\begin{equation*}
\left|\Phi'(t)\right|\leq\sqrt{-\frac{1}{a_\ell}\xi'(t)}\sqrt{\int_{\bRb^d}\left|\nabla\varphi^\ell(x)\right|^2\mu_t^\ell(x)\ddx}\,,
\end{equation*}
by reminding that $\xi(t):=\Upsilon_\sigma\left(\mu_t^1,\ldots,\mu_t^M\right)$. The
quantity $\left|\nabla\varphi^\ell(x)\right|^2$ is bounded by $\mathcal{Q}\left(|x|\right)$ and
$\int_{\bRb^d}|x|^{2N_0}\mu_t(x)\ddx$ is uniformly bounded with respect to $t\geq1$ for all $N_0\in\mathbb{N}$. So, there exists $C>0$ such that
$\int_{\bRb^d}\left|\nabla\varphi^\ell(x)\right|^2\mu_t^\ell(x)\ddx\leq C^2$ for all
$t\geq1$. We deduce
\begin{equation}
\label{eq:a:majoration}
\left|\Phi'(t)\right|\leq C\sqrt{\left|\xi'(t)\right|}\,.
\end{equation}
By definition of the two sequences $(r_k)_{k\in\mathbb{N}}$ and $(s_k)_{k\in\mathbb{N}}$, we have
\begin{equation*}
\Phi(s_k)-\Phi(r_k)=\kappa\,.
\end{equation*}
Combining this identity with Inequality \eqref{eq:a:majoration} and the monotonicity of $\xi$ yields
\begin{equation*}
C\int_{r_k}^{s_k}\sqrt{-\xi'(t)}\ddt\geq\kappa\,.
\end{equation*}
We apply the Cauchy-Schwarz inequality and we obtain
\begin{equation*}
C\sqrt{s_k-r_k}\sqrt{\xi(r_k)-\xi(s_k)}\geq\kappa\,.
\end{equation*}
Moreover, $\xi(t)$ converges as $t$ goes to infinity (see Lemma \ref{lem:fr:conv}). It implies the convergence of $\xi(r_k)-\xi(s_k)$ to $0$ when $k$ goes to infinity. Consequently, $s_k-r_k$ converges to infinity so
\begin{equation*}
\liminf_{k\longrightarrow+\infty}s_k-r_k>0\,.
\end{equation*}

\noindent{}{\bf Step 2.} By Lemma \ref{lem:fr:conv}, $\Upsilon_\sigma(\mu_t)-L_\sigma=-\int_t^\infty\xi'(s)\dd s$ converges to $0$. As $\xi'$ is nonpositive, we deduce that $\sum_{k=N}^\infty\int_{r_k}^{s_k}\xi'(s)\dd s$ also converges to $0$ when $N$ goes to infinity. As $\displaystyle\liminf_{k\longrightarrow+\infty}s_k-r_k>0$, we deduce the existence of an increasing sequence $q_k\in[r_k;s_k]$ which goes to infinity and such that $\xi'\left(q_k\right)$ converges to $0$ as $k$ goes to infinity. Furthermore, $\int_{\bRb^d}\varphi^\ell(x)\mu_{q_k}^\ell(x)\ddx\in[\kappa;2\kappa]$, for all $k\in\mathbb{N}$.

\noindent{}{\bf Step 3.} By proceeding similarly to the proof of \cite[Theorem 2.7]{JOTP}, we extract a subsequence of $(q_k)_{k\in\mathbb{N}}$ (we continue to write it $q_k$ to simplify the reading) such that $(\mu_{q_k}^1,\ldots,\mu_{q_k}^M)$ weakly converges to a stationary state $(\nu^1,\ldots,\nu^M)$ such that $\int_{\bRb^d}\varphi^\ell(x)\nu^\ell(x)\ddx\in[\kappa;2\kappa]$.

\noindent{}{\bf Step 4.} Since $\xi'(q_k)$ goes to $0$ as $k$ goes to infinity, we can apply Proposition~\ref{jigsaw}. Thus, we deduce $\Upsilon_\sigma(\nu^1,\ldots,\nu^M)=L_\sigma:=\lim_{t\to+\infty}\Upsilon_\sigma(\mu_t^1,\ldots,\mu_t^M)$.
\end{proof}

\bibliographystyle{alphaabbr}
\bibliography{refs}

\end{document}